\documentclass[twoside,english,12pt]{amsart}
\usepackage{amsfonts,amsmath,amssymb,amsthm}
\usepackage{subfigure,graphicx,color}
\usepackage{float}
\usepackage{xspace,accents}
\usepackage{enumitem}
\usepackage{dutchcal}
\usepackage{mathrsfs}

\usepackage{tikz}

\usepackage{hyperref}
\hypersetup{
    colorlinks=true,       
    linkcolor=blue,          
     citecolor=blue,        
    filecolor=magenta,      
    urlcolor=cyan           
  }
\newcommand{\BB}{\mathbb{B}}

\newcommand{\DD}{\mathbb{D}}

\newcommand{\HH}{\mathbb{H}}

\newcommand{\RR}{\mathbb{R}}

\newcommand{\VV}{\mathbb{V}}
\newcommand{\WW}{\mathbb{W}}
\newcommand{\XX}{\mathbb{X}}

\def\caH{{ \mathcal H }}

\def\caN{{\mathcal N}}

\def\cad{\mathcal d}
\def\cah{\mathcal h}

\def\ux{{\underline x}}

\def\us{{\underline s}}
\def\uy{{\underline y}}

\def\ua{{\underline a}}
\def\oa{{\overline a}}
 \def\ob{{\overline b}}

 \def\uz{{\underline z}}
 \def\oz{{\overline z}}

 \def\ut{{\underline t}}
 \def\ot{{\overline t}}

 \def\osigma{{\overline \sigma}}

\def\uT{{\underline T}}

\def\sft{\mathsf t}

\def\eps{\varepsilon}

\def\bfh{{\bf h}}

\def\bfn{{\bf n}}
\def\bft{{\bf t}}
\def\bfu{{\bf u}}
\def\bfv{{\bf v}}
\def\bfw{{\bf w}}

 \newcommand\caO{{\mathcal O}}

\renewcommand{\d}{\partial}
\newcommand{\ovl}[1]{\overline{#1}}
\newcommand{\udl}[1]{\underline{#1}}

\DeclareMathOperator{\dist}{dist}
\DeclareMathOperator{\Span}{Span}
\DeclareMathOperator{\supp}{supp}
\DeclareMathOperator{\singsupp}{sing~supp}

\DeclareMathOperator{\sign}{sign}
\DeclareMathOperator{\WF}{WF}

\DeclareMathOperator{\Ends}{\d}
\DeclareMathOperator{\End}{\d}

\newcommand{\eg}{e.g.\@\xspace}
\newcommand{\resp}{resp.\@\xspace}

\newcommand{\pp}{a.e.\@\xspace}

\newcommand{\nhd}{neighborhood\xspace}
\newcommand{\nhds}{neighborhoods\xspace}

\newcommand{\bichar}{bicharacteristic\xspace}

\newcommand{\dcone}{double-cone\xspace}

\newcommand{\dcones}{double-cones\xspace}
\newcommand{\Dcones}{Double-cones\xspace}
\newcommand{\DC}{\DD}

\newcommand{\dd}{domain of determinacy\xspace}

\newcommand{\dds}{domains of determinacy\xspace}
\newcommand{\Dds}{Domains of determinacy\xspace}

\newcommand{\suff}{sufficiently\xspace}

\newcommand{\Equiv}{\Leftrightarrow}
\newcommand{\Imply}{\Rightarrow}

\newcommand{\N}{\mathbb{N}}

\newcommand{\R}{\mathbb{R}}

\renewcommand{\O}{\mathcal O}

\newcommand{\caL}{\mathcal L}
\newcommand{\M}{\mathcal M}
\newcommand{\caM}{\mathcal M}

\newcommand{\cushion}{\DC_c}

\newcommand{\tgamma}{\tilde{\gamma}}

\newcommand{\tcaN}{\tilde{\caN}}

\newcommand{\Con}{\ensuremath{\mathscr C}}

\newcommand{\Cinf}{\ensuremath{\mathscr C^\infty}}
\newcommand{\Cinfc}{\ensuremath{\mathscr C^\infty_{c}}}

\newcommand{\dup}[2]{\langle #1, #2 \rangle}

\newcommand{\gd}{g^\delta}

\newcommand{\gR}{\mathsf{g}}

\newcommand{\nablaL}{\nabla_{\!\!\caL}}

\DeclareMathOperator{\sgn}{sgn}

\newcommand{\taum}{\tau_{\max}}

\newcommand{\ba}{\bar{a}}
\newcommand{\bb}{\bar{b}}

\newcommand{\bm}{\bar{m}}

\newcommand{\bs}{\bar{s}}

\newcommand{\bsigma}{\bar{\sigma}}

\newcommand{\by}{\bar{y}}

\newcommand{\bt}{\bar{t}}
\newcommand{\bx}{\bar{x}}

\newcommand{\Norm}[2]{{\| #1 \|}_{#2}}

\newcommand{\norm}[2]{{|#1|}_{#2}}

\renewcommand{\qedsymbol}{$\blacksquare$}

\newcommand{\Minkowski}{\mathbb M}

\newtheorem{theorem}{Theorem}[section]
\newtheorem{lemma}[theorem]{Lemma}
\newtheorem{proposition}[theorem]{Proposition}
\newtheorem{corollary}[theorem]{Corollary}

\newtheorem{remark}[theorem]{Remark}
\newtheorem{example}[theorem]{Example}
\newtheorem{hypothesis}[theorem]{Hypothesis}
\newtheorem{definition}[theorem]{Definition}

\newtheorem*{open*}{Open question}

\newcommand{\et}{\ensuremath{\text{and}}}

\newcommand{\avec}{\ensuremath{\text{with}}}
\newcommand{\si}{\ensuremath{\text{if}}}
\newcommand{\pour}{\ensuremath{\text{for}}}

\newcommand{\where}{\ensuremath{\text{where}}}

\newcommand{\atime}{\mathsf t}
\newcommand{\atimeF}{\atime^+}
\newcommand{\atimeP}{\atime^-}

\newcommand{\atimeFd}[1]{\atime^{+,#1}}

\newcommand{\Future}{\mathsf F}
\newcommand{\Past}{\mathsf P}

\newcommand{\Cone}{\Gamma}

\newcommand{\y}[1]{y^{(#1)}}
\newcommand{\x}[1]{x^{(#1)}}
\newcommand{\yt}[1]{t^{(#1)}}

\newcommand{\bld}[1]{\mbox{\boldmath $#1$}}

\newcommand*\circled[1]{\tikz[baseline=(char.base)]{
    \node[shape=circle,draw,inner sep=1pt] (char) {#1};}}

\newcommand{\point}{\alpha}

\newcommand{\usq}{\uy}
\newcommand{\bsq}{\by}

\newcommand{\Null}{Z}
\newcommand{\Timelike}{\Gamma}

\numberwithin{equation}{section}
\begin{document}

\title[Homotopies and domain of determinacy]{Timelike curves: homotopies and domain of determinacy}

\author{J\'er\^ome Le Rousseau and Jeffrey B. Rauch}
\address{J. Le Rousseau. Laboratoire analyse, g\'eom\'etrie et applications,
 Universit\'e Sorbonne Paris-Nord, CNRS, 
 Villetaneuse, France.}
\email{jlr@math.univ-paris13.fr}
\address{J. Rauch. Department of Mathematics, University of Michigan, Ann Arbor, Michigan, USA.}
\email{rauch@umich.edu}

\date{\today}

\begin{abstract}
  This paper studies domains of determination of 
linear strictly hyperbolic second order operators $P$. For an 
open set $\caO$, a set $Z$ is a domain
of determination when the values of solutions of the differential equation
$Pu=0$ are determined on $Z$  by their values in $\caO$.  
Fritz John's global Holmgren theorem implies that points that can
be reached by deformations of noncharacteristic hypersurfaces 
with initial surface and boundaries in $\caO$ 
belong to  a domain of determination  provided that
local uniqueness  holds
at  noncharacteristic surfaces.   Using spacelike hypersurfaces
yields sharp finite speed results whose domains of determination
are described in terms of influence curves that never exceed
the local speed of propagation.
This paper studies deformations of  noncharacteristic
nonspacelike hypersurfaces.  We prove that points reachable by 
(repeated) deformations by noncharacteristic nonspacelike
hypersurfaces coincide exactly with the set of points
reachable by  (repeated) homotopies of timelike arcs whose
initial curves and endpoints belong to $\caO$.   
When the set $\caO$ is a small neighborhood of a forward  timelike
arc connecting $a$ to $b$, a natural candidate for $Z$ is the intersection of the future of 
$a$ with the past of $b$.  This candidate is exact for D'Alembert's equation.
We prove that it is also exact when $a,b$ are sufficiently close points on a  fixed timelike
arc.  The timelike homotopy 
 criterion  fuels the construction of surprising examples for which the domain of determination
 is strictly larger (resp. strictly smaller) than the future-intersect-past 
 candidate.
\end{abstract}

\maketitle

{\sc Dedication.}  It is an honor to contribute this article in celebration of
the 100th anniversary of the birth of Peter Lax, a major 
architect of modern mathematical analysis and computational
science.  He directed the thesis of the second author and
his deep mathematical work
inspired  both of us throughout our careers.
His leadership of the global mathematical community
in particular supporting computational mathematics
and the unity of applied and pure mathematics
helped transform the mathematical sciences.
His stories, jokes, humor, humility, and empathy
made meetings with him a pleasure
and his friendship a treasure.

\tableofcontents
\section{Introduction}
This paper studies \dds for solutions of linear
partial differential equations $Pu=0$.
Suppose that $\caL$ is a connected smooth manifold without boundary,
and $P$ is a second-order linear partial differential operator with
coefficients infinitely differentiable.
\begin{definition}
  Suppose that $\caO\subset\caL$ is an open subset.  An open set
  $Z\subset\caL$ is {\bf a \dd of $\caO$} when $u\in
  H^2_{loc}(\caL)$, $Pu=0$ on $\caL$, and $u=0$ on $\caO$ imply that
  $u=0$ on $Z$.
\end{definition}
If $\{Z(g): g \in G\}$ is a family of \dds of $\caO$
then $\cup_{g\in G}Z(g)$ is a \dd of $\caO$.
\begin{definition}
  \label{def:domainofdetermination}
  For open $\caO\subset\caL$, the union of all \dds
  is the largest \dd of $\caO$.  It is denoted
  $Z_\caO$, and is called {\bf the \dd} of $\caO$.
\end{definition}
In particular $Z_\caO$ is open.  If $\caO_1\subset\caO_2$, then
$Z_{\caO_1}\subset Z_{\caO_2}$. Then, $Z_\caO$ is the intersection of
all $Z_{\caO'}$ with $\caO\subset \caO'$. Therefore, the definition of
\dd can be extended to arbitrary sets.
\begin{definition}
 \label{def:domain_of_determination2}
 If $A\subset \caL$ is any subset, {\bf the \dd of
   $A$}, denoted $Z_A$, is the intersection of \dds
 of open sets $\caO\supset A$,
 \begin{align*}
   Z_A := \mathop{\cap}_{\caO \in \mathscr V (A)} Z_\caO, \quad \where \ \ 
   \mathscr V (A) =  \big\{ \caO\,:\, \caO\ \text{is open and}\ A \subset \caO \big\}.
 \end{align*}
\end{definition}
Note that $Z_A$ may not be open if $A$ is not open.

The problem is to determine large \dds.  Their
structure depends in subtle ways on the global behavior of the
coefficients.  For the hyperbolic operators that  are  our main interest,
Lorentzian geometry plays a central role.  Before jumping into that,
recall some history and introduce key ideas.

\medskip
{\bf  Unique continuation for analytic elliptic operators.}   The oldest results concern
holomorphic functions and harmonic functions.  The 
underlying operators are $\d/\d \ovl{z}$
and $\Delta$ respectively. Solutions are given by  convergent
Taylor series.
Therefore,  if $u$ is a solution  on $\caL$
and $u$ vanishes on a nonempty open  $\caO$, 
then $u=0$ on the connected set $\caL$.
The proof applies to elliptic partial differential
operators of arbitrary order whose coefficients are given by convergent Taylor series.

\medskip
{\bf Deformations of hypersurfaces and John's theorem.}     
The preceding result has an independent  proof that  does not
depend on the analyticity of the solutions.  The
zeros are propagated from $\caO$ to larger sets using
Fritz John's Global Holmgren Theorem. 
\begin{definition}
\label{def: reachable point} 
  Suppose that 
  $\caL$ has dimension $1+d$, and $\O\subset \caL$ is  open.
  A {\bf continuous deformation of  $\bld{d}$-dimensional 
  embedded hypersurfaces  with initial surface and boundaries in 
  $\bld{\caO}$} is 
  a pair $(\M, F)$ where $\M$ is a 
  $d$-dimensional compact  $\Con^1$  manifold 
  with boundary and   
   $F:[0,1] \times \M \to \caL$ is a $\Con^0$-map satisfying   
  \begin{enumerate}[align=left,labelwidth=0.8cm,labelsep=0cm,itemindent=0.8cm,
  leftmargin=0cm,label=\bf{(\roman*)}]
  \item For each $a \in [0,1]$, $F_a: \M \to \caL$, with
    $F_a(.) = F(a,.)$, is a $\Con^1$-embedding.
  \item $F(\{0\} \times \M) = F_0(\M) \subset \O$.
  \item  $F \big([0,1]\times \d \M\big) = \cup_{a \in [0,1]}  F_a(\d\M)
    \subset \O$.
  \end{enumerate}
   $F_0(\M)$ is called the {\bf initial manifold}, and, $F
  \big([0,1]\times \d \M\big)$ is called the {\bf boundary}.  $F
  \big([0,1]\times \M\big)$ is called the set of {\bf points reached
    by the deformation}. 
\end{definition}
Figure~\ref{fig: h1 O} sketches two common situations.
%
\begin{figure}
  \begin{center}
    \subfigure[\label{fig: h1 O a}]
    {\resizebox{2.4cm}{!}{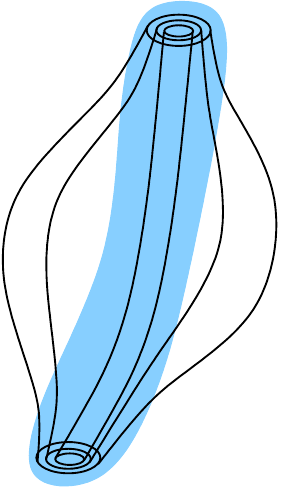}}
    \qquad    \qquad 
     \subfigure[\label{fig: h1 O b}]
     {\resizebox{2.8cm}{!}{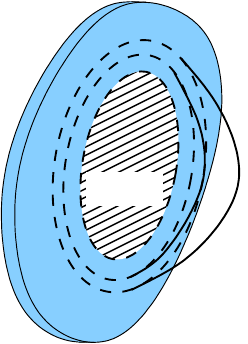}}
    \caption{Deformation of a tube on the left and disk on the right.  $\caO$ is blue.}
  \label{fig: h1 O}
  \end{center}
\end{figure}
\begin{definition}
  \label{def:unique_continuation}
An embedded hypersurface $\Sigma \subset \caL$ has {\bf local
  uniqueness in the Cauchy problem} when for each $y\in \Sigma $ there
is a \nhd $\Omega\subset \caL$ with $y\in \Omega$ so that if $u\in
H^2(\Omega)$, $Pu=0$ on $\Omega$, and, $u=\nabla u=0$ on $\Sigma
\cap\Omega$, then $u$ vanishes on a \nhd of $y$.
\end{definition}
\begin{theorem} [\bf John's Global Holmgren]
  \label{theorem:John global}
  Suppose that $\O\subset\caL$ is  open  and $F$ is a continuous
  deformation of $d$-dimensional hypersurfaces with initial surface and
  boundary in $\caO$.  If for all $0\le a\le 1$ there is local
  uniqueness in the Cauchy problem at $F_a(\M)$, then $F([0,1]\times
  \M) \subset Z_\O$.
\end{theorem}
The proof is sketched in Section \ref{sec:JT}.  In John's paper
\cite{John:49} the operators are real analytic and the hypersurfaces
$F(\M)$ are noncharacteristic.  For them, uniqueness in the Cauchy
problem is assured by Holmgren's theorem.

For elliptic operators any hypersurface is noncharacteristic.  Since
$\caL$ is connected, one easily constructs noncharacteristic
deformations reaching any $y\in \caL$ with initial surface and
boundary in $\caO$.  This yields an alternate proof of unique
continuation for elliptic operators with analytic coefficients.  For
this and other examples, see Section 1.8 of \cite{JBR:PDE} where a
slightly less general notion of deformation is used.

\medskip
{\bf Unique continuation for nonanalytic second-order elliptic
  operators.}  Carleman estimates suffice to prove local uniqueness
for the Cauchy problem for second-order  scalar elliptic
operators and arbitrary $\Con^1$-hypersurfaces; see \eg \cite[Chapter
  8]{Hoermander:63} or \cite[Chapter 5]{LRLR:V1}.  Therefore the
preceding argument proves that {\it solutions of scalar second-order
  elliptic equations, $Z_\caO=\caL$ for any open $\caO$.}  Therefore
for all $y\in \caL$, $Z_{\{y\}}=\caL$.

\medskip
{\bf Finite speed for hyperbolic operators.}  
Our results on propagation of zeros concern scalar second order
operators 
$P(t,x,\d_{t},\d_x)$ 
on 
$\caL = \R^{1+d}_{t,x}$ with smooth real valued coefficients.   The operator is strictly
hyperbolic at all $(t,x)$  with $dt$ timelike.  This means that the principal
symbol $P_2$ satisfies for all $(t,x)$ and all nonzero covectors $\xi\in
T^*_{t,x}(\R^{1+d})$ not parallel to $dt$, the equation $P_2(t,x,\tau dt+\xi)=0$ has two
distinct real roots $\tau$.  Proposition \ref{prop:strict hyperbolic}
shows that multiplying $P$ by $-1$ if need be, $P_2$ is
a quadratic form of signature $-1,1\dots,1$ so defines a Lorentzian
structure on $\R_{t,x}^{1+d}$.  Elements of Lorentzian geometry are recalled
in Appendix~\ref{app:Lorentzian_geometry}. That $dt$ is timelike is equivalent
to the fact that the level surfaces of $t$ are spacelike.

For $P$, a  hypersurface $\Sigma$ is noncharacteristic 
if  $P_2(y,\eta)\neq 0$ for all $y \in \Sigma$ and $\eta\in
T^*_{y}(\R^{1+d})\setminus 0$ conormal to $\Sigma$.  Equivalently, its
Lorentzian normals have nonzero lengths. 

 A  hypersurface in a Lorentzian manifold is {\bf spacelike},
 {\bf characteristic}, or, {\bf noncharacteristic nonspacelike} when
 its Lorentzian normals have negative, zero, or, positive lengths
 respectively.  A spacelike or noncharacteristic nonspacelike
 hypersurface remains the same under small perturbations on compact
 subsets.  Therefore, a continuous deformation of noncharacteristic
 hypersurfaces is either a continuous deformation of spacelike
 hypersurfaces or a continuous deformation of noncharateristic
 nonspacelike hypersurfaces.

Timelike vectors are tangent vectors with strictly negative length.
For $y = (t,x)\in \R^{1+d}$, the {\bf forward timelike cone}
$\Cone^+_y\subset T_{y}(\R^{1+d})$ is the component of
$\d/\d t$ in the set of timelike tangent vectors at $y$.
A {\bf forward timelike  curve} is Lipschitz and its
tangent at all points $y$ where it is differentiable  belongs to
$\Cone^+_y$. A {\bf Forward causal curves} is Lipschitz and its
tangent at all points $y$ where it is differentiable   belongs to
$\ovl{\Cone^+_y}\setminus 0$.

In order for the Cauchy problem with initial data on a level set of $t$
to be well posed, Hypothesis~\ref{hyp: metric} is invoked.  It requires that 
partial derivatives of all orders of the coefficients belong to $L^\infty(\R^{1+d})$.
Second it requires that $dt$ is uniformly timelike in the sense that
the  Lorentzian length of $dt$ is $\le -B_1<0$.  Equivalently
the level surfaces of $t$ are strictly spacelike.   The Lorentzian lengths
of unit tangent vectors to these level surfaces have length $\ge B_2>0$.  Finally, the determinant
of the Lorentz metric coefficients is $\ge B_3>0$ in absolute value.   This
guarantees that the distinct roots do not approach each other at infinitly.
With these, the analysis of the Cauchy problem with energy estimates
growing at most exponentially in time is proved with the 
multiplier $\d_t u$.   The conditions are necessary.   If the $B_1$
and $B_2$
hypotheses failed the initial surface could degenerate to characteristic
at infinity.  If the $B_3$ hypothesis failed the natural energy would not
be equivalent to $H^1$ at infinity.
For the Schwarzschild metric, $\d_t$ is not timelike.  Hypothesis 3.1 excludes it.

Suppose that $\omega$ is an open subset of $\RR^d$
and consider solutions $u$ whose Cauchy data at $t=0$ vanish 
on $\omega$.
A classical energy argument proves
local uniqueness in the Cauchy problem
at all spacelike surfaces.  
Therefore,  for each $x\in \omega$, there is an 
open subset $\Omega_x\subset\RR^{1+d}$ with $(0,x)\in \Omega_x$ and $u|_{\Omega_x}=0$.  
Then $u=0$ on the open subset   $\Omega:=\cup_{x\in \omega} \Omega_x$
of  $\RR^{1+d}$ that contains 
$\{0\}\times\omega$.   
The problem of describing accurately the finite speed of propagation is the same
as finding large sets of determination of $\Omega$.  Equivalently finding large
subsets of $Z_\Omega=Z_\omega$.

For $y=(t,x)$, {\bf the future} $\Future (y)$ is the union of
forward timelike curves starting at $y$ and {\bf the closed future}
$\ovl{\Future (y)} $ is the union of forward causal curves starting
at $y$.  {\bf The past $\Past(y)$} is the union of
backward timelike curves from $y$ and {\bf the closed past
  $\ovl{\Past(y)}$} is the union of backward causal curves from $y$.

\medskip
With Hypothesis~\ref{hyp: metric}, \cite{JMR:05}  proved that
{\it John's Theorem applied to deformations of spacelike surfaces}
implies that solutions of $Pu=0$ on $\R^{1+d}$ satisfy
\begin{equation}
  \label{eq:sharp}
  \supp u  \subset  \mathop{\cup}_{x \notin \omega}
 \ovl{\Future(0,x)}\, \cup\, \ovl{\Past(0,x)}.
\end{equation}
This sharp  finite speed result
takes account of both the dependence of speeds on location and
direction.
This sharp finite speed
is equivalent to the lower bound
\begin{align*}
  Z_\omega = Z_\Omega \supset \R^{1+d}  \setminus
  \mathop{\cup}_{x \notin \omega}
 \ovl{\Future(0,x)}\, \cup\, \ovl{\Past(0,x)}.
\end{align*}

{\bf Noncharacteristic nonspacelike deformations and double cones.}  
   This  paper investigates propagation
of zeros for second-order scalar strictly hyperbolic operators 
given by John's theorem with  continuous deformations of 
noncharacteristic nonspacelike hypersurfaces.  
The Timelike Homotopy Theorem stated below gives a convenient  description
of the points  that are reachable by such
deformations.

The study is motivated
by control theory.  A classical duality argument shows that if  observations on an open set $\caO$
suffice to uniquely determine solutions, there is approximate controllability
with controls supported in $\caO$.
The Timelike Homotopy Theorem
is a tool for proving  observability, hence approximate controllability.
\begin{definition}
  \label{def: reachable point2}
  Suppose that $\caO\subset\caL$ is an open subset.
  Define $\cad^1 \O$ to be the set of points reachable from $\caO$ by
  noncharacteristic nonspacelike deformations with initial surfaces and
  boundaries contained in $\caO$.  Proposition \ref{prop: d1O open
    set} proves that $\cad^1(\caO)$ is open.  Define inductively for
  $n\ge 2$, $\cad^{n} \O = \cad^1 (\cad^{n-1} \O)$, and,
  $\cad^{\infty}\O= \cup_{n \ge 1} \cad^{n}\O$.  $\cad^{\infty}\O$ is
  called {\bf the set of points reachable from $\bld{\O}$ by iterated
    noncharacteristic nonspacelike deformations}.
\end{definition}
\begin{remark}
  \label{remark: d 1 dinfty O}
  Then $\cad^1(\cad^{\infty}\O)= \cad^{\infty}\O$ so 
 $ \cad^{\infty}(\cad^{\infty}\O)=\cad^{\infty}\O$.
  Indeed, if $y \in \cad^1(\cad^{\infty}\O)$ then there exists a function
  $F$ and $\M$ as in Definition~\ref{def: reachable point} such that 
  $y \in F([0,1]\times \M)$ and $K = \M_0\cup F
  \big([0,1]\times \d \M\big) \subset \cad^{\infty}\O$. Since $K$ is
  compact then $K \subset \cad^{n}\O$ for some $n\in \N$.
  Thus, $y \in  \cad^1(\cad^{n}\O) = \cad^{n+1}\O$. 
\end{remark}
\begin{definition} 
\label{def: reachable point - homotopy}
A {\bf homotopy of timelike curves} is with $a<b$, a map
\begin{align*}
  \XX \in \Con^0\big([0,1];
\Con^1([a,b]; \caL) \big),
\end{align*}
 with $[a,b]\ni s \mapsto \XX(\sigma, s)$ timelike for all
 $\sigma\in[0,1]$.

For $\sigma \in [0,1]$ define $\XX_\sigma := \XX(\{ \sigma\} \times
[a,b])$.  The set $\XX_0$ is called the {\bf initial curve}.  The set
$\Ends{\XX} := \XX \big([0,1]\times \{a,b\}\big)$ is called the {\bf
  boundary}.
  
  For $y\in \XX([a,b]\times[0,1] )$ the homotopy {\bf reaches} $y$.
  Define $\cah^1\caO$ to the be set of {\bf points reachable from
    $\bld{\caO}$ by a timelike homotopy} with initial curve and boundary in $\caO$.
\end{definition}
Figure~\ref{fig: homotopy} sketches an example of a homotopy of timelike curves.
\begin{figure}
  \begin{center}
    \resizebox{3cm}{!}{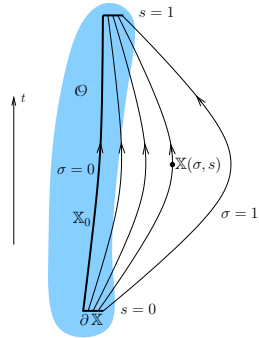}
    \caption{A homotopy of
timelike curves as in Definition~\ref{def: reachable point -
  homotopy}.  $\caO$ is blue.}
  \label{fig: homotopy} 
  \end{center}
\end{figure}
\begin{definition}
  \label{def: reachable point - homotopy2}
  Suppose $\caO$ is an open subset of $\caL$. 
  Proposition \ref{prop: h1O open set} proves that $\cah^1\caO$ is
  open.  Define inductively for $n\ge 2$, $\cah^{n} \O = \cah^1
  (\cah^{n-1} \O)$, and, $\cah^{\infty}\O= \cup_{n \ge 1} \cah^{n}\O$.
  $\cah^{\infty}\O$ is called the {\bf set of points reachable from
    $\bld{\caO}$ by iterated timelike homotopies}.
  \end{definition}
\begin{remark}
  \label{remark: h 1 hinfty O}
  As in Remark
  \ref{remark: d 1 dinfty O},
  $\cah^1(\cah^{\infty}\O)= \cah^{\infty}(\cah^{\infty}\O)=\cah^{\infty}\O$. 
\end{remark}
The next theorem, from Section~\ref{sec:Timelike homotopies}, is a
result in Lorentzian geometry.
\begin{theorem}[\bf Timelike Homotopy Theorem]
  \label{theorem:homotopy theorem}
  If  $\O$ is an open subset of a Lorentzian manifold $\caL$, then $\cah^\infty \O = \cad^\infty \O$.
\end{theorem}
John's theorem then yields the following.
\begin{theorem}
  \label{theorem:ZO homotopy}
  Suppose that $\caL$ is manifold with Lorentzian structure defined by
  a second-order real scalar strictly hyperbolic operator $P$.  In
  addition suppose that there is uniqueness in the Cauchy problem for
  $P$ at every noncharacteristic nonspacelike hypersurface.  Then, for
  all open sets $\caO\subset \caL$, $\cah^{(\infty)}\O \subset Z_\O$.
\end{theorem}

The next  theorem, from the  seminal papers of
   \cite{Robbiano:91,Tataru:95},  refined in
   \cite{Hoermander:97,RZ:98,Tataru:99, LL:19},
gives hypotheses 
 guaranteeing   unique continuation at noncharacteristic nonspacelike hypersurfaces.
\begin{theorem}[Robbiano-Tataru]
   \label{thm:Robbiano-tataru}
   If 
   $P$ is a
   scalar real second-order strictly hyperbolic operator and
   $\Omega$ is an open set  on which there are coordinates $(t,x)$ so
   that $dt$ is timelike and the coefficients of  $P$ are  real analytic
   in $t$, then on $\Omega$ there is 
    uniqueness in the Cauchy problem  at all noncharacteristic nonspacelike
    $\Con^1$ hypersurfaces. 
 \end{theorem}
\begin{open*}
  Examples in \cite{Alinhac:83} and \cite[Theorem 1.4 and Remark
    1.5]{zuilly:83} show that uniqueness in the Cauchy problem at
  noncharacteristic nonspacelike hypersurfaces can fail for smooth
  complex lower-order terms. The case of smooth {\it real} coefficients is a
  challenging open problem.
\end{open*}
\begin{remark}
  \label{remark:C1vsCinf}
  Theorem~\ref{thm:Robbiano-tataru} is stated in the literature for more regular hypersurfaces. Appendix~\ref{sec:hypersurface Robbiano-Tataru} shows that $\Con^1$ regularity suffices. 
\end{remark}

\bigskip
\begin{definition}
  \label{def: dcone intro}
  If $\gamma$ is a forward  timelike open arc joining $y_1$ to
  $y_2$, {\bf the double cone} $\DC(\gamma):= \Future(y_1)\cap \Past(y_2)$.
\end{definition}
Figure~\ref{fig:dcones-Minkowski} provides examples of  \dcones in the Minkowski space. 

\begin{example}
  \label{ex:Minkowski-DD}
  Consider  d'Alembert's wave equation
$\Box u=0$.  The corresponding Lorentzian manifold is the Minkowski
  space $\Minkowski^{1+3}$.  Consider different types of open
  timelike arcs $\gamma$ and open sets $\caO$ such that $\gamma \subset \caO \subset \DC(\gamma)$.
 
 {\bf 1.}   For the timelike arc $\gamma=\{(t,x): x=0, \ -T< t<T \}$,
it is easy to see that 
$\cad^1\caO=\DC(\gamma)$. John's theorem implies that
$Z_\caO\supset \DC(\gamma)$.

For all points $\uy=
(\ut,\ux)\notin \ovl{\DC(\gamma)}$ with $\ut\ge 0$, the line through
$\uy$ with speed equal to $-\ux/\|\ux\|$ is a null geodesic.
Parameterize it as $t\mapsto (t,x(t)):=\zeta(t)$.  It lies outside $\ovl{\DC(\gamma)}$. The curve $\lambda(t):= (\zeta(t),
|\ux |,\ux)$ is a bicharacteristic for $P=\Box$ lying above
$\zeta$.  Choose Cauchy data localized near $\zeta(0)$ and away from
$\ovl{\DC(\gamma)} \cap \{t=0\}$ and such that $\lambda(0) \in \WF(u)$ for
$u$ the associated solution of $P u=0$. By sharp finite speed, $\supp
u \cap \DC(\gamma)=\emptyset$.  Since $\uy \in \singsupp u$
one has $\uy \in \supp u$, proving that $\uy \notin Z_{\caO}$.

A symmetrical argument treats $\ut\le 0$
yielding $\DC(\gamma)=\cad^1\caO=Z_\caO$ using that $Z_{\caO}$ is
open. Then, $Z_\gamma = \DC(\gamma)$.

 {\bf 2.} Suppose that $y_2\in \Future(y_1)$ and consider the open timelike
segment  $\gamma(s)=(1-s)y_1+sy_2$ with $0<s<1$. 
The Lorentz invariance of $\Box$ allows one to conjugate this problem
to the preceding one so 
$\DC(\gamma)=\cad^1(\caO)= Z_\caO = Z_\gamma$.  
 
{\bf 3.}
Suppose that $\gamma(t)=(t,x(t))$ is a timelike line joining $y_1$ to $y_2$. 
    Define $x_\eps(t) =x(t) + \eps (0,\cos t, \sin t)$.
    Then,
    for $0<\eps\ll 1$ the path $\gamma_\eps(t)=(t,x_\eps(t) ) $ is
  a timelike helical path tightly wound about $\gamma$.
In this case too
  $Z_\gamma =Z_\caO= \DC(\gamma)$.   In this case noncharacteristic nonspacelike deformations
  are not obvious.  In contrast, timelike homotopies 
reaching arbitrary points in $\DC(\gamma)$
are easy to construct.
A conjugation reduces to a helix about a stationary line.
A first homotopy decreasing the radius of the helix
leads to  the stationary line.  Then, one has the geometry
of the first example and timelike homotopies reaching arbitrary points of $\DC(\gamma)$
are easy.  
\end{example}
\begin{figure}
  \begin{center}
    \subfigure[\label{fig:dcone-1}]
    {\resizebox{3.5cm}{!}{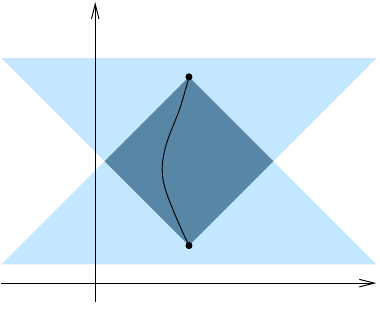}}
    \qquad    \qquad 
    \subfigure[\label{fig:dcone-2}]
    {\resizebox{3.5cm}{!}{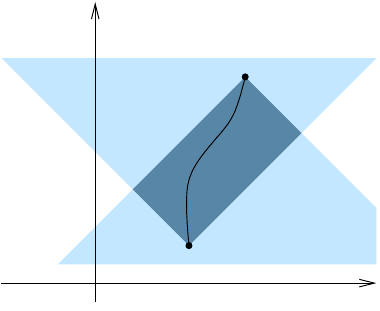}}
    \qquad    \qquad
    \subfigure[\label{fig:dcone-3}]
    {\resizebox{3.5cm}{!}{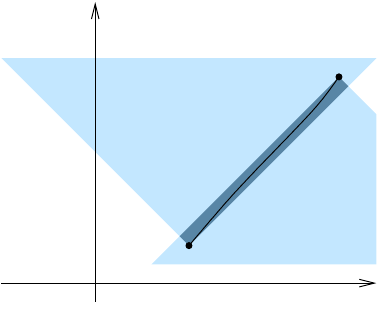}}
    \caption{Three \dcones in the Minkowski space.}
  \label{fig:dcones-Minkowski}
  \end{center}
\end{figure}

\medskip
{\bf Double cones beyond  Minkowski: Doldrums.}  
Suppose that  $P$ is a 
second-order real strictly hyperbolic  operator and 
$\gamma$
is a smooth  open timelike arc.
Denote
\begin{align}
  \label{eq:dn_gamma}
  \cad^n \gamma = \mathop{\cap}_{\caO} \cad^n\caO,\quad n \in \N^* \cup \{+\infty\},
\end{align}
with the intersection  over all open sets containing $\gamma$.

 Here are  ten reasonable sounding assertions involving the sets
$Z_\gamma$,  $\cad^n(\gamma)$, and, $\DC(\gamma)$.    
They are all true when $P$ is D'Alembert's operator.  They are all 
true for small double cones.

\smallskip
\begin{tabular}{lll}
{\bf 1.}  $Z_\gamma \,\subset\, \ovl{\DC(\gamma)}$.\\
{\bf 2.}    $\DC(\gamma)   \,\subset\, Z_\gamma$.\\
{\bf 3.}    $Z_\gamma\subset \cad^1(\gamma)$.  
& {$\bf 3'$.} $Z_\gamma \subset \cad^n(\gamma)$ for $n$ large.
& {$\bf 3^{\prime\prime}$.}   $Z_\gamma \subset  \cad^\infty(\gamma)$.\\
{\bf 4.}    $\DC(\gamma) \subset \cad^1(\gamma)$.
& {$\bf 4'$.}     $\DC(\gamma) \subset  \cad^n(\gamma)$   for $n$ large.
&{$\bf 4^{\prime\prime}$.}    $\DC(\gamma) \subset \cad^\infty(\gamma)$.\\
{\bf 5.}    $\DC(\gamma)$ is a topological ball.
& {$\bf 5'$.}   There is a finite upper bound & \!\! \!\!\!on the genus of $\DC(\gamma)$.
\end{tabular}

\smallskip
Note that $\cad^{1}\O \subset \cad^{n}\O \subset
  \cad^{\infty}\O \subset \DC (y,y')$; see Proposition~\ref{prop:
    dcones upperbound for Oinfty} below. Hence, if assertions
  $\bld{4}$, $\bld{4'}$, and $\bld{4''}$ did hold, inclusions could be replaced by
  equality. We however insist that these assertions do not hold.

\smallskip
{\it None of these assertions is true in general.} Even for operators
  $P$ with time independent coefficients.  The Timelike Homotopy
Theorem is used to  construct counterexamples as
well as in proving correctness in the small.  The two most striking 
facts are that $Z_\gamma$ can be striclty larger than $\DC(\gamma)$, and,
the difficulty of proving that small $\DC(\gamma)$ are exactly equal to $Z_\gamma$.

To show that the ten assertions are false we construct two examples.
The examples involve regions where the speed of propagation is very
small.  We call them {\bf doldrums} after the regions of calm in the
intertropical Atlantic zone.  Sailors avoid doldrums to reach their
destinations rapidly.  It is the same for waves.  Figure~\ref{fig:
  doldrum-intro} sketches wavefronts of a plane wave in the presence
of a doldrum.  The wavefront is sketched at successive times showing
that they wrap around the doldrum in a sort of embrace that creates a
void in the future sketched in Figure~\ref{fig:
  doldrum-intro}. If paired with a similar behavior for the
  past one obtains \dcones with peculiar properties.
\begin{figure}
  \begin{center}
    \resizebox{4.5cm}{!}{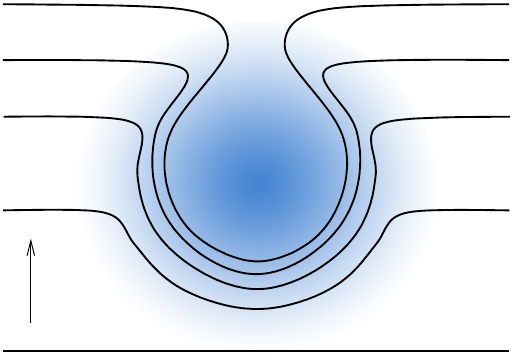}
    \caption{Sketch of a plane wave moving upward and interacting with a low speed region. The bluer the region, the smaller the wave speed.}
  \label{fig: doldrum-intro} 
  \end{center}
\end{figure}
The first doldrums example in
  Section~\ref{sec:First_doldrum_example} disproves {\bf 2},
  $\bld{4}$, $\bld{4'}$, and $\bld{4''}$: it exhibits a pair of
  timelike arcs with the same endpoints that are not homotopic. The second doldrums example in
  Section~\ref{sec:Second_doldums_example}  disproves 
  {\bf 1}, $\bld{3}$, $\bld{3'}$, and $\bld{3''}$: it
  exhibits points outside $\ovl{\DC(\gamma)}$ that lie in the \dd of
  $\gamma$. This astonished us.
  The second doldrums example is also used to obtain a double
 cone as in Figure~\ref{fig: doldrums3b} explaining the failure of {\bf
   5} and  $\bld{5'}$.

The previous negative results may appear even more striking in
  view of the following positive results that are proven in the
  present article.
  \begin{theorem}  
    [smalll double cones] Suppose that $P$ has local uniqueness in the
    Cauchy problem at noncharacteristic nonspacelike surfaces, and,
    $\gamma(t)=(t,x(t))$ for $-\infty<a\le t\le b<\infty$ is a compact
    timelike arc.  Then, there is $\eps>0$ so that if $\tgamma$ is a
    subarc on a time interval of length $\le \eps$, then
    $\cad^{\infty}(\tgamma) = {\cah}^{\infty}(\tgamma) = \DD(\tgamma)
    = Z_{\tgamma}$.  The set is diffeomorphic to a ball.
  \end{theorem}
\begin{theorem}
[double cones with temporal symmetry] Suppose that the prinicpal
symbol of $P$ is independent of $t$ and contains no cross terms
$\partial_t\partial_{x_j}$.  If $\gamma$ is a compact interval on a
line parallel to the $t$ axis, then $\cad^{\infty}(\gamma) =
{\cah}^{\infty}(\gamma) = \DD(\gamma) \subset Z_\gamma.  $
\end{theorem}
\begin{remark}  The absence of cross terms is equivalent to
  the Lorentz orthogonality of the timelike lines $\{\xi=0\}$ and the
  spacelike hyperplane $\{\tau=0\}$ in the cotangent spaces at all $y
  \in \L$.
\end{remark}

We know three places where timelike homotopies play a role
resembling what they do here.  In \cite{Rauch77,RS:81}, timelike
homotopies for $\Box$ serve as sketetons for the bare hands
construction of noncharacteristic nonspacelike deformations.  Alvarez
and Sanchez \cite{Alvarez-Sanchez:15} construct timelike curves
connecting a pair of points that are not homotopic by timelike
connections.  The obstruction is doldrums.

An article that is a close cousin\footnote{The authors thank A.~Stromaier for pointing out this connection after our article was posted on arxiv} of this one is~\cite{strohmaier}.  It propagates zeros from a set
$\caO$ to a set called the envelope of $\caO$
that is generated from $\caO$ using  deformations
of all noncharacteristic surfaces.  By combining
everything in one it generates a set that is large but
very difficult to  determine.  
In contrast we achieve simplicity
for nonspacelike deformations.
Though timelike homotopies
do not appear in the definition of envelope, timelike curves
do.  A part of the proof of Theorem 3.2 in \cite{strohmaier}  
resembles the proof of the Timelike Homotopy Theorem.

{\bf Outline.}
Section~\ref{sec:Hyperbolic_operators_Lorentz_unique_continuation}
proves results announced in the Introduction including the Timelike
Homotoy Theorem.  Section~\ref{sec:First-arrival_future_past} studies first arrival times
and the structure of the future and pasts.  It uses the elements of
Pseudo-Riemannian Geometry in the Lorentzian case.  Section~\ref{sec:dcones}
 studies relations of futures and \dcones in the
large.  The $\ovl{\Future(z)}$ is said to plunge into
$\Future(y)$ if all future causal curves starting at $z$ enter
$\Future(y)$.  In the Minkowski setting this can only happen if $z\in
\Future(y)$.  The final Section~\ref{sec:Double-cones_dds} gives the surprisingly
subtle proof that small double cones have all the ten reasonable
properties.  
 A sufficient condition ruling
out plunging is  crucial in studying small double
cones. 
That section has two doldrum examples.  The first gives a
double cone $\DC(q_1,q_2)$ and two timellike arcs $\gamma_1$ and
$\gamma_2$ connecting $q_1$ to $q_2$ so that there are points of
$\gamma_2$ that are not in $Z_{\gamma_1}$.  The second doldrum example
has the futures hugging the doldrum set.  There are three brief
appendices.  Appendix~\ref{app:Lorentzian_geometry} recalls basic
notions of Lorentzian geometry.  Appendix~\ref{sec:hypersurface
  Robbiano-Tataru} is devoted to showing that $\Con^1$ regularity of
hypersurfaces is sufficient for uniqueness in the Cauchy problem.  At
least $\Con^2$ is standard practice.  Appendix~\ref{app:speeding up
  and slowing down metrics} describes metrics parameterized by
$\delta>0$ so that for $\delta>1$ (resp. $\delta<1$) the speeds are
faster (resp. slower) than those of the original problem.

\section{John's theorem, Lorentzian structure, timelike homotopies}
\label{sec:Hyperbolic_operators_Lorentz_unique_continuation}

This section  contains  proofs of four results from  the Introduction.

\subsection{Hyperbolicity and Lorentz structure}
\label{sec:Hyperbolicity and Lorentz structure}
Suppose that $\caL$ has dimension $1+d$.  Denote by $y$ points in
$\caL$, by $(y,\eta)$ points in $T^*\caL$, and, by $p(y,\eta)$ the
principal symbol of $P$.

\begin{definition}  
$P$ is {\bf strictly hyperbolic with timelike codirection $\bld{\omega}$ }
when the following hold, 
\begin{enumerate}[align=left,labelwidth=0.8cm,labelsep=0cm,itemindent=0.8cm,
    leftmargin=0cm,label=\bf{(\roman*)}]
  \item  $\omega$
 is a smooth nonvanishing
one-form on $\caL$.
\item  $p(y,\omega(y)) \neq 0$  for all $y \in \caL$.
\item 
 If $(y,\eta) \in T^*\caL$ with $\eta \notin \R \omega(y)$, $p(y, \tau \omega(y)+ \eta)=0$ has two distinct real roots $\tau$.
\end{enumerate}
\end{definition}
\begin{remark}  Since $\caL$ is connected and the principal symbol is real valued,
$p(y,\omega(y))$ has one sign.
Multiplying $P$ by $-\sgn p(y,\omega(y))$ reduces to the case $p(y,\omega(y))<0$.
That is condition {\bf (i)} of the Proposition \ref{prop:strict hyperbolic}.
\end{remark}

Define $g^*_y$ to be the unique symmetric bilinear form on
$T^*_y\caL$ so that $g^*_y(\eta,\eta)= p(y,\eta)$. The next
proposition shows that $g^*_y$ is a Lorentzian
metric on $T^*_y\caL$.
\begin{proposition}
  \label{prop:strict hyperbolic}
  Suppose that $\VV$ is a vector space, $N\in \VV\setminus 0$, and
  $q(\eta)$ is a real quadratic form on $\VV$.  Denote by
  $\phi$ the unique symmetric bilinear form so that
  $\phi(\eta,\eta)=q(\eta)$.  Then the following are equivalent.
  \begin{enumerate}[align=left,labelwidth=0.8cm,labelsep=0cm,itemindent=0.8cm,
    leftmargin=0cm,label=\bf{(\roman*)}]
  \item \label{prop:strict hyperbolic1} $q(N)<0$ and the equation
    $q(\tau N + \eta) =0$ has two distinct real roots $\tau$ if $\eta\notin \R N$.
  \item \label{prop:strict hyperbolic2}  $\phi$ is a Lorentz metric on $\VV$ with $N$ timelike.
  \end{enumerate}
\end{proposition}
\begin{proof}
\ref{prop:strict hyperbolic2} $\Imply$ \ref{prop:strict hyperbolic1}.   There exist  coordinates $\tau,\xi$ on
$\VV$ so that $N=(1,0,\dots, 0)$ and $\phi= -\tau^2 + \xi\cdot \xi$.  
That \ref{prop:strict hyperbolic1}  is satisfied is then immediate. 

\ref{prop:strict hyperbolic1}  $\Imply$ \ref{prop:strict hyperbolic2}.  Dividing $q$ and $\phi$ by 
$-q(N)$ reduces to the case $q(N)=\phi(N,N) =-1$.
Expanding and completing the square yields, with $r_N(\eta) = \phi(\eta,\eta) +\phi(\eta,N)^2$,
\begin{align*}
q(\tau N + \eta)
=- \tau^2+2\phi(\eta, N) \tau+\phi(\eta,\eta)
=-\big(\tau - \phi(\eta,N)\big)^2 + r_N(\eta).
\end{align*}
For $\eta\notin \R N$ there are two distinct real roots so
$r_N(\eta)>0$.   
If $W$ a subspace of $\VV$ complementary to 
$\R N$, then $r_N(\eta)$ is positive definite on $W$.
It follows that $\phi$ has signature $-1,1,\dots, 1$.  In particular
$\phi$ is nondegenerate. 
\end{proof}

If $\bfv \in T_y\caL$, there exists a unique $\bfv^\flat \in T^*_y\caL$
such that $\dup{\eta}{\bfv} = g^*_y (\eta, \bfv^\flat)$ for all $\eta \in
T_y^* \caL$.     
For $ \bfv, \bfw \in T_y\caL$,
define
\begin{align}
  \label{eq: dual metric}
  g_y (\bfv, \bfw) = g^*_y (\bfv^\flat, \bfw^\flat) .
\end{align}
Then, $(\caL, g)$ is a Lorentzian manifold. 
In local coordinates, 
if the metric 
$g_y = (g_{ij}(y))_{1\leq i,j \leq d+1}$  then
one has 
$(\bfv^\flat)_i = g_{ij}(y) \bfv^j$ with the usual Einstein summation
convention. 
In addition,
$g_y^* = (g^{ij}(y))_{1\leq i,j \leq d+1}$, that is $g_y^* = (g_y)^{-1}$.

\subsection{Noncharacteristic-nonspacelike hypersurface deformations}
\label{sec:Noncharacteristic-nonspacelike hypersurface deformation}
\begin{proposition}
  \label{prop: d1O open set}
  Suppose $\caO$ is an open subset of $\caL$.
  $\cad^{1}\O$, from Definition \ref{def: reachable point2}, is an open subset of $\caL$. 
\end{proposition}
\begin{proof}
  If $\by \in \O$ then $\cad^1 \O$ is a \nhd of $\by$. 
 For $\by \in \cad^1 \O \setminus \O$, we prove that a \nhd of $\by$ belongs
 $\cad^1\caO$. Since $\by\in\cad^1 \O$, $\by$ is reached by a 
   noncharacteristic nonspacelike hypersurface deformation
as in  Definition~\ref{def: reachable point}, with  $\by = F(\ba, \bs)$, 
$\ba \in ]0,1]$, and, $\bs \in ]0,1[$.

\vskip.1cm {\bf Step 1.}  
The hypersurface $\caM_{\oa}$ is noncharacteristic and nonspacelike.
Choose $\bfv$ Lorentz orthogonal to $\caM_{\oa}$ at  $\by$.
In local coordinates, where $\by$ is the origin, consider the line
$\Sigma$ in $\R^{1+d}$ given by $\alpha \mapsto \alpha \bfv$.
Then, $\Sigma$ is transverse to $\caM_{\oa}$ at $\by$.
 Define on a \nhd of $(\by,0)\in
\caM_{\oa}\times \Sigma$,
\begin{align}
  \label{eq: local diffeo deformation}
(m,\bfw)   \mapsto
\psi(m,\bfw) = m  + \bfw
\ \in \R^{1+d}\,.
\end{align}
The transversality  implies that
the derivative of $\psi$ at $(\by,0)$ is invertible.  Therefore 
$\psi$ is a local $\Con^1$-diffeomorphism.

\vskip.1cm

{\bf Step 2.}  Next construct deformations sweeping out a \nhd of $0$ that are equal to
$F_a$ for $0\le a\le \oa$. 
The deformation is continued
by modifying $\caM_{\oa}$ only near $\by$ using
$\psi$ from Step~1.  Choose $\chi\in \Cinf(\caM_{\oa})$ supported
near $\by$, identically equal to 1 on a smaller \nhd, and with  $0\le \chi\le 1$.  
The support and 
$\delta>0$ are chosen so that  $\psi$ is a diffemorphism on
a \nhd of $\{\supp \chi \times [-\delta,\delta]\bfv$. 

If $y\in \psi(\{\chi=1\}\times [-\delta,\delta]\bfv)\setminus \caM_{\oa}$, write 
$y=\psi(m,\alpha \bfv)$.  
 Define 
$\omega=\sign(\alpha)$.  The point $y$ is reached by the deformation
 $F(a,m)=
\psi(m, (a-\oa)\, \chi(m)\,\omega \bfv)$ with $(a,m)\in [\oa, \oa +|\alpha|]\times\caM_{\oa}$. 

As ${\rm dist} (y, \caM_{\oa} )$ tends to zero, the deformed surfaces  for $a>\oa$ converge to
$\caM_{\oa}$ in the $\Con^1$~topology.  
Since $\caM_{\oa}$ is noncharacteristic and nonspacelike, it follows that
there is an $\eta>0$  so that 
 for ${\rm dist} (y, \caM_{\oa} )<\eta$,
 the  deformations are noncharacteristic and nonspacelike. 
\end{proof}

\subsection{Timelike homotopies}
\label{sec:Timelike homotopies}
\begin{proposition}
  \label{prop: h1O open set}
  Suppose $\caO$ is an open subset of $\caL$. $\cah^{1}\O$, from Definition \ref{def: reachable point - homotopy},
  is an open subset of $\caL$.
\end{proposition}

\begin{proof}
   If $\by \in \O$ then $\cah^1 \O$ is a \nhd of $\by$.
   Consider $\by \in \cah^1 \O \setminus \O$. It is reached by a
homotopy of
timelike curves $\XX$ as in Definition~\ref{def: reachable
   point - homotopy}, with  $\by = \XX(\bsigma, \bs)$ with
   $\bsigma \in ]0,1]$ and $\bs \in ]a,b[$.  

{\bf Step 1.}  In local coordinates where $\by$ is the origin, define
$\Sigma$ to be  the Lorentz orthogonal to  $\d_s {\mathbb X}(\by)$.
Then, $\Sigma$ is transverse to 
$\XX_{\bsigma}$ at $\by$. 
Choose a norm on $\Sigma$, denoted $|\,.\,|$, with smooth balls.
As in Step~1 of Proposition~\ref{prop: d1O open set},
define the local $\Con^1$-diffeomorphism
\begin{align}
  \label{eq: local diffeo homotopy}
  (\point, w)  \mapsto \psi(\point, w) = \point  + w,
\end{align}
on a \nhd of $(\by,0)\in \XX_{\bsigma}
\times \Sigma$.

\vskip.1cm

{\bf Step 2.} The homotopy is not modified for $0 \leq \sigma \leq \osigma$.
For $\sigma>\osigma$
the arc $\XX_\osigma$ is deformed only in a small \nhd of $0$
using the map $\psi$
from Step 1. Choose real valued $\chi\in \Cinf(\XX_{\osigma})$ supported in a
small \nhd of $0$ and identically equal to 1 on a smaller \nhd.
The support and $\delta>0$ are chosen so that
$\psi$ is a diffeomorphism on a \nhd of $\{\supp \chi\}\times \{w\in \Sigma:|w|\le \delta\}$.

If $y\in \psi(\{\chi=1\}\times \{| w |\le \delta\})\setminus \XX_{\osigma}$, write 
$y=\psi(\udl{\point}, \udl{w})$.  
 Express $\udl{w}$ in polar form $\udl{w}=r \omega$ with
$|\omega|=1$.  The point $y$ is reached by the continuous homotopy of
smooth curves $\XX(\sigma, s) = \psi(\point,
(\sigma-\osigma)\chi(\point)\,\omega)$ with $\osigma\le \sigma \le
\osigma+r$ and $\point = \XX(\bsigma, s)$.  These deformations sweep out a \nhd of $0$.

As $\dist(y, \XX_{\osigma})$ tends to 0, the deformed curves for $\sigma>0$
converge to $\XX_{\osigma}$ in the $\Con^1$~topology.  Therefore there  is an 
$\eta>0$ so that if $\dist(y, \XX_{\osigma})<\eta$
 the deformed curves are timelike.
\end{proof}

\begin{proof}[\bf Proof of the Timelike Homotopy Theorem~\ref{theorem:homotopy theorem}]
   {\bf Step 1. $\bld{ \cah^1 \O \subset \cad^\infty \O}$.}
Consider a homotopy of timelike curves $\XX$ as in
Definition~\ref{def: reachable point - homotopy}.  Set
\begin{align*}
J
=  \big\{\sigma\in [0,1]: \ 
  \XX\big([0,\sigma]\times [a,b]\big) \subset \cad^\infty \O
\big\}.
\end{align*}
By continuity of $\XX$, $J$ is open in $[0,1]$ as $\cad^\infty \O$ is
open, and $0\in J$ since $\XX_0 \subset \O\subset \cad^\infty
\O$.  By connectedness, it suffices to prove that $J$ is closed, that
is, to show that if $\sigma_n$ increases to $\bsigma$ and $\sigma_n\in
J$ then $\bsigma \in J$. Since $\Ends{\XX} \subset \O$, it suffices to
show that each point of $\XX\big(\{\bsigma\} \times ]a,b[\big)$ is in
            $\cad^\infty \O$.  Consider $\by = \XX(\bsigma,\bs)$
            with $\bs \in ]a,b[$.

Without  loss of generality,  suppose $\bs = 0$ . Use the local
$\Con^1$-diffeomorphism $\psi$ in \eqref{eq: local diffeo homotopy} and denote by $V$ its range.
The point $\XX(\sigma,0)$ converges to $\by$ as $\sigma \to
\bsigma$. Define the functions $\point(\sigma,s)$ and $\bfw(\sigma,s)$
such that $\XX(\sigma,s) = \psi(\point(\sigma,s),\bfw(\sigma,s))$ for
$(\sigma,s)$ close to $(\bsigma, \bs)$. They are continuous in
$\sigma$ and $\Con^1$ in $s$.  One has $\point(\sigma,s) \in \XX_{\bsigma}$ and $\point(\sigma,s) \to
\XX(\bsigma,s)$ and $\bfw(\sigma,s) \to 0$ as $\sigma \to \bsigma$,
uniformly in $s$.

With $\Sigma$ from the proof of Proposition \ref{prop: h1O open set},
define
the smooth hypersurface
\begin{align*}
  \M =\{ (z,\bfu):\, z\in  [-1,1]  \ \et\  \bfu\in \Sigma
  \ \avec \ \norm{\bfu}{}= 1\} \subset [-1,1] \times \Sigma.
\end{align*}
Choose $\chi \in \Cinfc(]-1,1[)$ such that  $0\leq \chi \leq 1$ and
            $\chi=1$ on $[-1/2, 1/2]$.
Suppose $\eps>0$, and $\delta>0$. Define the 
deformation of $\M$, \begin{align*}
  H: [0,1]\times \M &\to [-1,1]\times \Sigma
  \\
  (a, z, \bfu) &\mapsto (z, \bfh(a, z, \bfu)),
  \quad \where \ \ \bfh(a, z, \bfu) = (\delta + \eps a \chi(z)) \bfu,
\end{align*}
giving a surface of revolution $r= \delta + \eps a \chi(z)$ in
cylindrical coordinates. The initial surface is a cylinder of radius
$\delta$. The width of the final surface is $\sim \eps$ if $\delta \ll \eps$. The surfaces
are carried to $\caL$ using the diffeomorphism $\psi$ and the timelike curve $s\mapsto \XX(\sigma,s)$.  Define the
hypersurface deformation $F^\sigma$ by $F^\sigma = \psi \circ G^\sigma$ with
\begin{align*}
   [0,1] \times \M \ni (a,z, \bfu) 
   \mapsto
  G^\sigma (a,z, \bfu) = \big(\point(\sigma, \eps^{1/2} z)\, ,
  \bfw(\sigma, \eps^{1/2} z)\big)
  + \big(0, \bfh(a, z, \bfu)\big).
\end{align*}
The deformed surface has an $\eps^{1/2}$ extent in the $s$
variable.  

Since $\psi\big(\point(\sigma, \eps^{1/2} z)\, , \bfw(\sigma,
\eps^{1/2} z)\big) = \XX(\sigma, \eps^{1/2} z)$, one finds 
\begin{align*}
  \d_z F^\sigma(a,z, \bfu)
  &=
  \eps^{1/2} \big(\d_s \XX(\sigma, \eps^{1/2} z) + O(\eps)\big)
  + d \psi \big( G^\sigma(a,z, \bfu)\big) (0, \d_z \bfh(a, z, \bfu))\\
  &= \eps^{1/2} \big(\d_s \XX(\sigma, \eps^{1/2} z) + O(\eps)\big)
  + \eps a \chi'(z)  d \psi \big( G^\sigma(a,z, \bfu)\big)(\bfu).
\end{align*}
For $\eps$ \suff small, this is a timelike vector implying that
$F^\sigma(\{a\}\times \M)$ is noncharacteristic-nonspacelike by
Proposition~\ref{prop:hyper}-(iii), uniformly in $\sigma$ near
$\bsigma$, $a \in [0,1]$, $z \in [-1,1]$, $\delta\in [0,\eps]$. Fix
such a value for $\eps$ requiring moreover $\XX\big(\{\bsigma\} \times
[-2\eps^{1/2}, 2\eps^{1/2}]\big) \subset V$.

For $\sigma$ near $\bsigma$, there exists an unique  $s_\sigma$ such that
$\XX(\sigma, s_\sigma) - \by \in \Sigma$ and $s_\sigma \to 0$ as
$\sigma \to \bsigma$. In particular, $\alpha(\sigma, s_\sigma) = \by$. Choose $n \in \N$ such  that   $2 |s_{\sigma_n}| \leq\eps^{1/2}$ and $\XX(\sigma_n,s) \in V$ and $|\bfw(\sigma_n, s)| \leq
\eps/2$ if $s\in [-\eps^{1/2}, \eps^{1/2}]$. Set $\sigma=  \sigma_n$.

Since $\XX_\sigma \subset \cad^\ell \O$ for some $\ell$ as $\XX_\sigma$ is
compact, then $F^\sigma(\{0\}\times \M) \cup F^\sigma([0,1]\times \d\M) \subset
\cad^\ell \O$ for $\delta$ chosen \suff small.

Set $\bfv= \bfw(\sigma, s_\sigma)/ |\bfw(\sigma, s_\sigma)|$ and $z =
\eps^{-1/2} s_\sigma$.  For $\delta$ \suff small, there exists $a\in
    [0,1/2]$ such that $\delta + \eps a = |\bfw(\sigma, s_\sigma)|$. Then, as
    $|z| \leq 1/2$, one has $\chi(z)=1$ and one finds $F^\sigma(a,z,-\bfv) = \by$ as
    \begin{align*}
      G^\sigma(a,z,-\bfv)
      = \big(\alpha(\sigma, s_\sigma), \bfw(\sigma, s_\sigma)\big)
      + \big( 0, -(\delta + \eps a) \bfv \big)
      = \big( \alpha(\sigma, s_\sigma), 0 \big)
      = (\by, 0).
    \end{align*}
    Hence, $\by \in
    \cad^{\ell+1} \O$.

\medskip
{\bf Step 2. $\bld{\cad^1 \O \subset \cah^\infty \O}$.}
Suppose $F:[0,1] \times \M \to \caL$ is a
noncharacteristic-nonspacelike hypersurface deformation as in
Definition~\ref{def: reachable point}.  Define $\M_a = F \big(\{ a\}
\times \M\big)$ and $J = \big\{a\in [0,1]: \ \M_a \subset \cah^\infty
\O\big\}$.  By continuity of $F$, $J$ is open in $[0,1]$ as
$\cah^\infty \O$ is open, and $0\in J$ since $\M_0 \subset \O \subset
\cah^\infty \O$. By connectedness, it suffices to prove that $J$ is
closed, that is, to show that if $a_n$ increases to $\ba$ and $a_n\in
J$ then $\ba \in J$.  Since $F \big([0,1] \times \d\M\big)\subset \O$,
it suffices to show that $F\big(\{\ba\} \times (\M \setminus
\d\M)\big)\subset \cah^\infty \O$. Thus, consider $\by = F\big(\ba,
\bm)$ for some $\bm\in \M \setminus \d\M$. If $\by \in \cup_{n \in \N}
\M_{a_n}$ then $\by \in \cah^\infty \O$; assume that $\by \notin
\cup_{n \in \N} \M_{a_n}$.

Use the local $\Con^1$-diffeomorphism $\psi$ in \eqref{eq: local
  diffeo deformation} and denote by $V$ its range. For $V$ small,
$\Sigma$ is transverse to all hypersurfaces $\M_a$ in $V$. For each
$a$ this transverality implies that $\psi(m,\bfv)\in \M_a$ is
equivalent to $\bfv = G_a(m)$ with $G_a: \M_{\ba} \to \Sigma$ a
$\Con^1$ map, continuous in $a$. Set $\M_{\ba} \ni m \mapsto \psi_a(m)
=\psi(m, G_a(m))$, $\Con^1$ in $m$ and continuous in $a$.  Then,
$\{\psi_a(m)\} = \M_a \cap (m + \Sigma)$. Set $\by_a = \psi_a(\by)$.

Suppose $\bft$ is a timelike vector at $\by$, tangent to $\M_{\ba}$
(see Proposition~\ref{prop:hyper}-(iii). Suppose $a\leq  \ba$.  Set
$z(a,s) = \by_a + s \bft$.  For $|s|$ small, $z(a,s) = \psi(
\alpha(a,s), \bfu)\in V$, for some $\alpha(a,s)\in \M_{\ba}$ and $\bfu
\in \Sigma$. Define $Z(a,s) = \psi_a(\alpha(a,s)) \in \M_a$,
continuous in $a$ and $\Con^1$ in $s$. One has $Z(a,0) =\by_a$. For
$|s|$ small and $a$ close to $\ba$, $\d_s Z(a,s)$ is close to $\d_s Z(\ba, 0) = \bft$
so is timelike.

Choose $\chi \in \Cinfc(]-1,1[)$ such that $0\leq \chi \leq 1$ and
$\chi(0)=1$. 
For $\eps>0$, $\sigma \in [0,1]$, and $s \in [-1,1]$, define
\begin{align*}
  \XX^a(\sigma, s) = Z(a,\eps^{1/2} s)
  - \eps \sigma \chi(s) \bfu_a/|\bfu_a|,
  \quad \bfu_a = G_a(\by). 
\end{align*}
This is well defined if $a=a_n$ for some $n\in \N$. Then, $\bfu_a\neq 0$
since $\by \notin \cup_{n \in \N} \M_{a_n}$.  For $\eps$ small,
$\XX^a(\sigma, s) \subset V$ and $\d_s \XX^a(\sigma, s)$ is timelike,
uniformly in $a$ near $\ba$, $\sigma\in [0,1]$ and $s\in [-1,1]$,
implying that $\XX^a$ is a homotopy of timelike
curves. Fix such a value for $\eps$.

One has $\XX^a_0 \cup \Ends{\XX^a} \subset \M_a$.  Note that
$\bfu_{\ba}=0$ and choose $n\in \N$ such that $|\bfu_{a_n}|\leq
\eps$. Set $a = a_n$. Then, $\XX^a(\sigma,0)=\by_a- \bfu_a=\by$ for
$\sigma=\eps^{-1} |\bfu_a| \in [0,1]$. As $\M_a \subset \cah^\ell\O$
for some $\ell$, then $\by \in \cah^{\ell+1}\O$.

\vskip.1cm {\bf Step 3. End of proof.}
Prove $\cah^n \O \subset \cad^\infty \O$ for all $n \in \N$ by induction. It holds for
  $n=1$. Assume it holds
  for some $n$. Then 
  \begin{align*}
    \cah^{n+1} \O=  \cah^1 \big(\cah^n \O \big)
    \subset \cah^1 \big(\cad^\infty \O \big),
  \end{align*}
  by the induction hypothesis. 
  With  Step~1, one has $\cah^{n+1} \O \subset \cad^\infty \big(\cad^\infty \O
  \big) = \cad^\infty \O$, with the last equality given by
  Remark~\ref{remark: d 1 dinfty O}. Then, prove $\cad^n \O \subset \cah^\infty
  \O$ for all $n \in \N$ by the same argument simply exchanging the
  roles of the symbols $\cad$ and $\cah$, and using Step~2.
\end{proof}

\begin{figure}
  \begin{center}
    \subfigure[\label{fig: h1 O bis-a}]
    {\resizebox{4cm}{!}{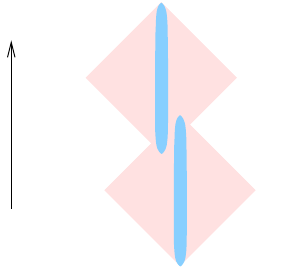}}
    \qquad    \qquad 
     \subfigure[\label{fig: h1 O bis-b}]
     {\resizebox{4cm}{!}{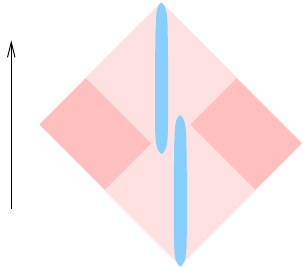}}
    \caption{Example where $\cah^1 \O \subsetneq \cah^2 \O$. $\O$ is blue.}
  \label{fig: h1 O neq h2 O}
  \end{center}
\end{figure}

\begin{figure}
  \begin{center}
  \resizebox{7cm}{!}{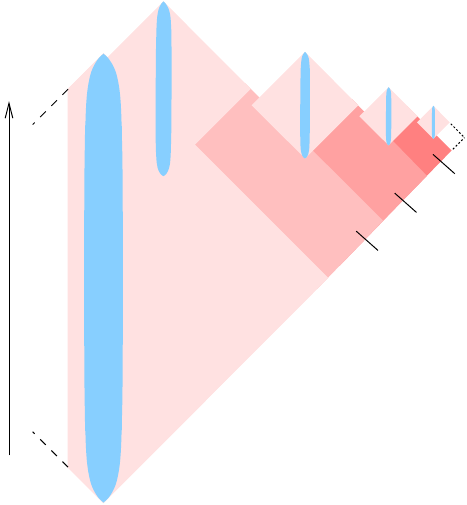}
  \caption{Example with $\cah^1 \O \subsetneq \cah^2 \O \subsetneq
\cah^3\O \subsetneq \cdots \subsetneq \cah^n\O$. $\O$ is blue.}
  \label{fig: h1 O neq h2 O-bis}
\end{center}
\end{figure}

\begin{figure}
  \begin{center}
    \subfigure[\label{fig: h1 O ter-a}]
    {\resizebox{7cm}{!}{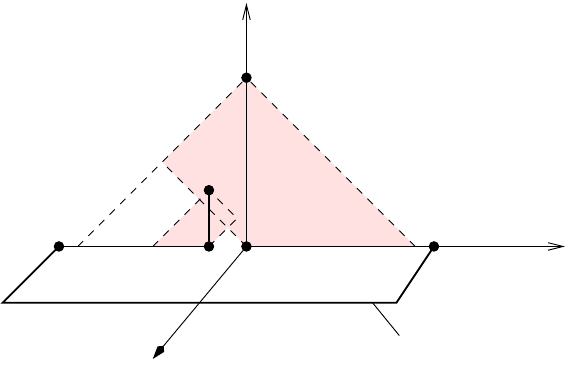}}
    \qquad   
     \subfigure[\label{fig: h1 O ter-b}]
     {\resizebox{7cm}{!}{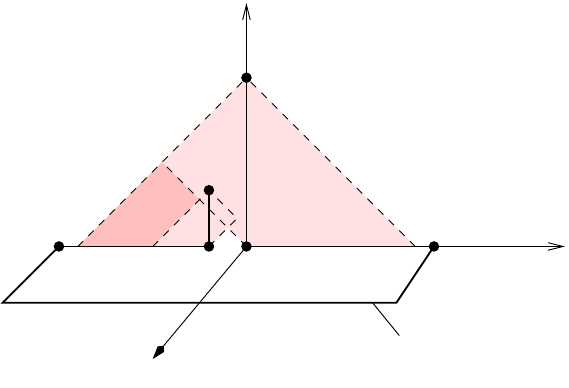}}
    \caption{The points $A$, $B$, $C$, $D$, $E$, and $F$ lie in the
      $(t,x_1)$ plane.  $D,E,B,C$ lie on a line in $t=0$.  Dotted
      lines are null lines.}
  \label{h1 O neq h2 O-ter}
  \end{center}
\end{figure}

\subsection{Examples with strictly increasing $\cah^n$}
The  sets $\cah^n \O$ are nondecreasing. 
The example of
Figure~\ref{fig: h1 O neq h2 O} has $\cah^1 \O
\subsetneq \cah^2 \O$. The example sketched in Figure~\ref{fig: h1 O
  neq h2 O-bis} has  $\cah^n \O \subsetneq \cah^{n+1} \O$ for all $n$.
Both are for  Minkowski space $\Minkowski^{1+1}$.

The final example, in $\Minkowski^{1+2}$, has $\cah^1 \O \subsetneq \cah^2 \O$, with $\O$ connected. 
 Consider Figure~\ref{h1 O
  neq h2 O-ter} with a continuous piecewise linear curve $\rho$ starting at point $A$ and ending at point $F$.
An important feature is that the part $\rho$ joining point $B$ to point $E$ is spacelike. Choose a small \nhd  $\O$ of $\rho$. In the figure, $\O$ is
chosen \suff small to be only represented by $\rho$.

\subsection{     {\bf Outline of proof of John's theorem \ref{theorem:John global}.} } 
\label{sec:JT}
Define
\begin{align*}
J:=\{a\in [0,1]: u=0\ \text{on a  \nhd  of}\ F([0,a]\times\caM) \} .
\end{align*}
Then $J$ is open in $[0,1]$.  Since the initial surface is in $\caO$
it follows that $0\in J$.  Define $\ba:=\sup J>0$.  It suffices to
prove that $\ba=1$.
     
Since the boundary of the deformation lies in $\caO$, $u=0$ on a \nhd
of $\d (F_{\ba}) $.  If $0<\ba<1$, the continuity of traces at
noncharacteristic surfaces implies that $u=\nabla u$ at interior
points of $F_{\ba}$.  Uniqueness in the Cauchy problem at $F_{\ba}$
implies that $u=0$ on a \nhd of $F_{\ba}$ contradicting the definition
of $\ba$.  Therefore $\ba=1$.  \hfill \qedsymbol \endproof

\section{Exponential map, arrival time, future, and  past sets}
\label{sec:First-arrival_future_past}

In this section and those that follow, we specialize to the case $\caL =
\R \times \R^d$.  Points of $\caL$ are denoted $y=(t,x)$.  
 The time component of $y$ is denoted  $t(y)$.
At places, one writes $y = (y_0, y_1, \dots, y_d)$ with $y_0
= t$ and $y_i = x_i$, $i=1, \dots, d$.

\begin{hypothesis}
  \label{hyp: metric}$\phantom{-}$
  \begin{enumerate}[align=left,labelwidth=*,labelsep=0cm,itemindent=*,leftmargin=0cm,label=\bf{(\roman*)}]
   
  \item For all $\alpha\in \N^{1+d}$, $\d_y^\alpha g^{ij} \in
    L^\infty(\caL)$, $i, j \in \{ 0, \dots, d\}$.
  
  \item $\d_t$ is uniformly timelike and any time slice $\{ t=C\}
    \simeq \R^d$ is uniformly spacelike, that is, there is a $C_1>0$
    so that for all $y\in \caL$,
\begin{align*}
  g_{00}(y) = g_y (\d_t, \d_t) \leq -C_1
  \  \ \et \  \
  g^{00}(y)= g_y^* (d_t, d_t)
  = g_y(d_t^\sharp, d_t^\sharp) \leq -C_1.
\end{align*}

\item There exists $C_3>0$ such that for all $y \in \caL$, $|\det g_y| \geq C_2$.
  \end{enumerate}
\end{hypothesis}
Refer to the introduction for a discussion of  these conditions.

The {\it sharp} $\sharp$ isomorphism is the inverse of the {\it flat} $\flat$ isomorphism
introduced in \eqref{eq: dual metric}: if $\eta \in T_y^*\caL$ then
$\eta^\sharp \in T_y\caL$ is such that $g_y (\eta^\sharp, \bfv) = \dup{\eta}{\bfv}$, giving in local coordinates
$(\eta^\sharp)^i = g_y^{ij}\,  \eta_j$.

With the time orientation given by $\d_t$, $(\caL, g)$ is a
space-time. Section~\ref{sec:case L=R1+d} recalls useful properties
of this space-time and serves for reference. 

\subsection{Elements of pseudo-Riemannian geometry applied to the Lorentzian case}
\label{sec:elements_pseudo-Riemannian}

Recall some facts in pseudo-Riemannian geometry.  They will be applied
in the Lorentzian setting.  In particular, the exponential map and
correponding normal coordinates are crucial. See B. O'Neill's
excellent book \cite{ONeill} for more details.

By definition, {\sl a {\bf
    pseudo-Riemannian manifold $\caM$} has a smooth field of symmetric
  nondegenerate bilinear forms on the tangent bundle}.  The bilinear
form gives the {\it flat} diffeomorphism $\flat$ from $T \caM$ to $T^*
\caM$ and so yields a corresponding form on $T^* \caM$.  Denote the
latter by $p(y,\eta) = \sum g^{i j}(y)\,\eta_i\eta_j$.  

The function
$p$ has Hamiltonian vector field $H_p$ on $T^* \caM$ whose integral
curves satisfy Hamilton's equations,
\begin{equation}
\label{eq:hp}
\frac{dy}{ds} = \nabla_\eta p,
\qquad
\frac{d\eta}{ds}
= -\nabla_y p\,.
\end{equation}
The Hamiltonian  $p$ is constant on integral
curves. Since $p$ is homogeneous of degree two in 
$\eta$,  if $(y(s), \eta(s))$ is an integral curve and $\lambda>0$, then
$\big(y(\lambda s), \lambda \eta(\lambda s)\big)$ is  an integral curve.

The diffeomorphism $\flat$ yields a bijection from vector fields on $T
\caM$ to vector fields on $T^* \caM$.  The geodesic flow is generated
by a vector field on $T \caM$. The vector field $H_p$ gives a flow on
$T^* \caM$.  The bijection maps one to the other.  The projections on
$\caM$ of the integral curves of $H_p$ are the geodesics of $\caM$.

The vector field on $T \caM$ from the geodesic flow generates a one
parameter group of diffeomorphisms $s\mapsto \Phi_s$ on $T \caM$.
Any point $(y,0)\in T \caM$ is a fixed point of $\Phi_s$.  In
addition, for all reals $\lambda$, $\Phi_{\lambda s}
(y,\bfv)=\Phi_s(y, \lambda \bfv)$, if $\bfv \in T_{y}\caM$.  The map $\Phi_1$ is a
smooth map from $T\caM$ to itself. 
  For all $y \in \caM$, the exponential map, $\exp_y: T_y\caM \to \caM$, is 
 its base component: $\exp_y(\bfv) = \pi_\caM \Phi_1(y, \bfv)$ for $\bfv
\in T_y \caM$. The scaling law implies that for all $y\in \caM$
the image by $\exp_y$ of a line through the origin of $T_y \caM$ is a
geodesic, called a {\bf radial geodesic}.  For $(y,\bfv)\in T \caM$, $\bfv$ is
the derivative at $s=0$ of the radial geodesic $s\mapsto \exp_y(s \bfv)$.
This implies that  the map $\exp_y$ has differential at the origin equal to the identity map
on $T_y \caM$.  Therefore $\exp_y$ is locally invertible.

\begin{definition}
\label{normalcoord}
A \nhd $\caN$ of $y$ in $\caM$ is called a {\bf normal \nhd} when
there is a star shaped $\widetilde\caN$ of $0\in T_y(\caM)$ so that
${\rm exp}$ is a diffeomorphism from $\widetilde\caN$ to $\caN$.  In
this case the set $\widetilde \caN$ defines normal coordinates and is
also called a normal \nhd.

A normal \nhd is {\bf convex} when it is a normal
\nhd of each of its points.
\end{definition}

Proposition 3.31 in \cite{ONeill} proves that 
a normal \nhd $\caN$ of $y$ is star shaped for geodesics
in the sense that for each $q\in \caN$, there is a unique 
geodesic from  $y$ to $q$.

\begin{proposition} 
  \label{prop:normal2}
There is a $r>0$ so that for all $y \in \R^{1+d}$ the set
$\caN=exp_y(\widetilde \caN)$ with
$\widetilde{\caN}=\{ \bfv \in T_y(\R^{1+d}): \ \sum_0^d (v^j)^2<r^2\}$ is a convex normal  \nhd.
 \end{proposition} 
\begin{proof}[{\bf Sketch of Proof}]  The inverse function theorem implies that
  there is an $r>0$ to that the set is a normal \nhd.
  Indeed, the derivative of $\exp_y$ at $0\in T_y(\R^{1+d})$ is equal
  to the identity.
  
  The global coordinates $(y_0,\dots, y_{d})$ on $\R^{1+d}$ induce
  global coordinates $(y,\bfv)$ on $T(\R^{1+d})$.  Our
  hypotheses imply that  $\d_{y,\bfv}^\alpha \Phi_1 \in L^\infty(\R^{ 2(1+d ) })
  $ for each $\alpha$.  This guarantees the existence of the desired $r$.
    
    The proof of Proposition~5.7 in \cite{ONeill} proves that one can choose
    a smaller $r$ so  that the \nhd is convex, and that can be 
    done uniformly in $y$.
\end{proof}
Recall that a timelike or null vector $\bfv$ is called forward if $g_y(\d_t, \bfv) <0$.
For a timelike  vector it equivalently means that $\bfv$ is in the same component of the cone of 
timelike vectors as $\d_t$; see Lemma~\ref{lemma: time-like cone}.
A timelike or null vector $\bfv$ is backward if and only if $-\bfv$ is forward
\begin{definition}  
\label{def:fivecones} If $\caL$ is Lorentzian and $y\in \caL$,  then
$T_y \caL\setminus 0$ is the disjoint union of five  connected cones.
\begin{enumerate}[align=left,labelwidth=0.8cm,labelsep=0cm,itemindent=0.8cm,
  leftmargin=0cm,label=\bf{(\roman*)}]
\item The open convex cone $\Timelike^+_y$ (resp.  $\Timelike^-_y$) of 
{\bf forward (resp. backward) timelike  vectors}.

\item The open cone of {\bf strictly spacelike vectors}.
\item The codimension 1 cone $\Null^+_y$ (resp. $\Null^-_y$) of {\bf
  forward (resp. backward) null vectors}.
\end{enumerate}
One also uses `future' for `forward' and `past' for `backward' equivalently. Define $\Lambda^\pm(y) = \exp_y(\Null_y^\pm)$. 
\end{definition}

\begin{lemma}
\label{lem:timelikecross}
Suppose $y\in \caL$ and $\caN=\exp_y(\tcaN)$ is a normal
\nhd of $y$.  At each point $q\in \caN \cap \Lambda^+_y$, $q \neq y$,
$\Lambda^+(y)$ is a smooth embedded hypersurface that is
characteristic.  Future oriented piecewise $\Con^1$ timelike curves at $q$ cross the
tangent plane  $T_q \Lambda^+(y)$ transversally and enter
$\exp_y(\Timelike^+_y)$.
\end{lemma}
\begin{proof}
   Since $\caN$ is a normal \nhd, $\exp_y\big|_{\tilde{\caN}}$ is a
   diffeomorphism. Since $\Null^+_y$ is smooth at all
   points other than the tip its diffeomorphic image,
   $\Lambda^+(y)\cap\caN$, is smooth at all points other than the tip.

   Suppose $q = \exp_y(\bfv)$ with $\bfv \in \Null^+_y$.  By the
   Gauss lemma, $w = d \exp_y(\bfv) (\bfv) \in T_q \caL$ yields the only null
   direction in $T_q \Lambda^+(y)$ and $(T_q \Lambda^+(y))^\perp =
   \Span\{w\}$ implying that $T_q \Lambda^+(y)$ is characteristic.

   Suppose $\rho$ is a future timelike curve and $\rho(s) =q$.  Then
   $\rho'(s)$ is timelike and $g_q(\rho'(s), w) <0$ as $w$ is a future
   null vector. Denote by $\bfu$ the only vector such that $d \exp_y
   (\bfv) (\bfu) = \rho'(s)$.  Then, $g_y(\bfu, \bfv) <0$ by the Gauss Lemma
   implying that $\bfu$ points towards $\Timelike^+_y$. Hence, $\rho$ enters
   $\exp_y(\Timelike^+_y)$.
\end{proof}

\subsection{Arrival time, future and past sets}
\begin{definition}
  \label{def: first arrival time}
Suppose $y\in \caL$. For $x \in \R^d$, define
$\atimeF_{y}(x)$ to be the {\bf future first-arrival time} at $x$ from $y$, that is,
\begin{multline}
  \label{eq: first arrival time}
  \atimeF_{y}(x) = \inf \{ t: \ \text{there is a future Lipschitz
    causal curve} \ \gamma \ \text{such that}\\ \gamma(0) = y \ \et
  \ \gamma(s)= (t,x) \ \avec \ s>0 \}.
\end{multline}
\end{definition}
Introduce the function
\begin{align}
  \label{eq: function tau max}
  \taum (t,x, \xi)= \max \{ \tau: \ p(t, x , \tau, \xi) =0\}, \qquad (t,x, \xi) \in \R^{1+d}\times \R^d.
\end{align}
\begin{proposition}[Section 7 in \cite{JMR:05}]
  \label{prop: JMR}
  Suppose $y =(\ut,\ux) \in \caL$ and $\Omega$ is a bounded open set of
   $\R^d$.  
   \begin{enumerate}[align=left,labelwidth=0.8cm,labelsep=0cm,itemindent=0.8cm,
       leftmargin=0cm,label=\bf{(\roman*)}]
     \item The function $\atimeF_y(x)$ is continuous in both $x$ and $y$. 
     
     \item For all $x \in \Omega$, a Lipschitz 
  causal curve $[0,S] \ni s \mapsto \gamma(s) = \big(t(s), x(s)\big)$  with
  $\gamma(0) = y$ and $\gamma(S)  = \big(\atimeF_{y}(x), x\big)$ achieves
       the infimum in \eqref{eq: first arrival time}.
  For all such causal curves $\gamma$  and
  all $s,s' \in [0,S]$ one has the dynamic programing principle
  \begin{align*}
    \atimeF_{\gamma(s)}\big(x(s') \big) = t(s'),  \quad \text{if} \ s < s'.
  \end{align*}
  
\item 
  The function $\Omega \ni x \mapsto  \atimeF_{y}(x)$ is a uniformly
  Lipschitz 
  solution to
  \begin{align}
    \label{eq: tmax equation}
    \taum\big( \atimeF_{y}(x), x, -d_x \atimeF_{y}(x) \big)=1,
    \ \text{\pp}\  x \in \R^d,
    \quad \atimeF_{y}(\ux) = \ut.
  \end{align}
  Moreover, it is the largest such solution in the sense that if
  $\zeta$ is uniformly Lipschitz function solution on $\Omega$ to
  \begin{align*}
    \taum\big( \zeta (x), x, -d_x \zeta (x) \big) \leq 1,
    \ \text{\pp}\  x \in \R^d,
    \quad \zeta (\x{0}) \leq  \yt{0},
  \end{align*}
  then $\zeta \leq \atimeF_{y}$ on $\Omega$.
\item If $\caN$ is a normal \nhd of $y$, then
  \begin{align*}
    \Lambda^+(y) \cap \caN = \{ (\atimeF_y(x),x): \ x \in \R^d\} \cap \caN.
  \end{align*}
  \end{enumerate}
\end{proposition}

\bigskip
Similarly, $\atimeP_{y}(x)$ denotes the first-arrival time at $x$
in the past for backward causal curves starting at $y$. 
If $\usq =(\ut,\ux)$ and $\bsq= (\bt,\bx)$, then 
\begin{align}
  \label{eq: equiv tau past futur}
  \bt = \atimeF_{\usq}(\bx)
  \ \ \Equiv \ \ 
  \ut = \atimeP_{\bsq}(\ux).
\end{align}
Proposition~\ref{prop: JMR} can be adapted to $\atimeP_{y}(x)$.  Consequently,
if $\Omega$ is a bounded open set of $\R^d$ and $I$ a bounded
interval of $\R$, then $\atimeF_{(\ut,\ux)}(x)$ and
$\atimeP_{(\ut,\ux)}(x)$ are uniformly Lipschitz with respect to
$(\ut,\ux,x) \in I\times \Omega \times \Omega$.

\bigskip
\begin{definition}
  \label{def: future, past}
The future and  the past  of $y\in \caL$ are the open sets 
\begin{align*}
  &\Future(y) = \{ q\in \caL: \ \text{there exists a smooth
    future timelike curve from} \ y \ \text{to}\ q \},\\
  &\Past(y) = \{
  q \in \caL: \ \text{there exists a smooth past timelike
    curve from} \ y \ \text{to}\ q \}
\end{align*}
respectively.
Their boundaries are denoted $\d \Future(y)$ and $\d \Past(y)$.
\end{definition}
Lemma~\ref{lemma: alternative future} below shows that smooth timelike curves can be
replaced by Lipschitz curves in Definition~\ref{def: future, past}.
Finite speed of propagation implies the following result.
\begin{lemma}
  \label{lemma: boundedness tip future and past}
  Suppose $y =(\ut, \ux) \in \caL$.  If $T>\ut$ then $\Future(y) \cap \{
  t \leq T\}$ and $\d \Future(y) \cap \{ t \leq T\}$ are bounded.
\end{lemma}

The following proposition describes  $\Future(y)$ near $y$  
in terms  of the exponential map.
\begin{proposition}
\label{prop:local form of F}
Suppose that $y\in \caL$ and $\caN=\exp_y(\tcaN)$ is a bounded normal
\nhd of $y$.  Define
\begin{align*}
  t_\caN
  = \inf \big\{t(q): \ q\in\exp_y(\Timelike^+_y)\,\cap\,\d\caN\big\}
  = \min \big\{t(q): \ q\in \exp_y(\ovl{\Timelike^+_y})\,\cap\,\d\caN\big\}>t(y).
\end{align*}
Then
\begin{align}
  \label{eq:twotips}
  \Future(y)\cap \{t<t_\caN\}
  =\exp_y (\Timelike^+_y) \cap \{t<t_\caN\}
  \subset \caN,
\end{align}
and
\begin{align}
    \label{eq: local identification boundary future}
    \d\Future(y) \cap \{ t < t_\caN\}
    = \Lambda^+(y) \cap \{ t < t_\caN\}
    \subset \caN.
\end{align}
  If  $q \in \d\Future(y) \cap \{ t < t_\caN\}$, it lies on a
  forward radial null geodesic in $\caN$ connecting $y$ to $q$.  Up to
  reparametrization, this is the only forward causal curve starting at
  $y$ and arriving at $q$.
\end{proposition}
\begin{proof}
The right hand side of \eqref{eq:twotips} is a subset of the left hand
side.  It suffices to show that $q\in \Future(y)\cap \{t<t_\caN\} $
implies that  $q\in \exp_y(\Cone^+_y)\cap \{t<t_\caN\}\cap\caN$.

Since $q\in \Future(y)$, choose a smooth future timelike curve $\gamma(s)$ from $y$ to $q$.
Since the initial derivative of $\gamma$ is timelike, it  belongs to $\Timelike^+_y$.
Define $\tgamma=\exp_y^{-1} \gamma$.
Then   $\tgamma$ takes values in 
$\Timelike^+_y\cap \tcaN$ near $y$.  

The key step is to show that $\tgamma$ stays in $\Timelike^+_y\cap \tcaN$ before $\gamma$ reaches $q$.
So long as $\widetilde \gamma$ takes values in $\tcaN$, Lemma
\ref{lem:timelikecross} implies that $\tgamma$ cannot touch
$\Null^+_y$.  The only way that $\gamma$ can leave $\exp_y(\Timelike^+_y)\cap \caN$ is by reaching a point $w\in \exp_y(\ovl{\Timelike^+_y}) \cap \d\caN$.
The choice of $t_\caN$ insures that $t(w)\ge t_\caN> t(q)$,
contradicting the choice of $\gamma$ as a future timelike curve
from $y$ to $q$. Hence,  $\gamma$ takes values in $\exp_y( \Timelike^+_y)\cap \caN$ where $q$ lies. Since $t(q)<t_\caN$, 
$q\in\exp_y( \Timelike^+_y)\cap \caN \cap\{t<t_\caN\}$,
completing the proof of \eqref{eq:twotips}.

Equation~\eqref{eq: local identification boundary future} follows
  from Equation \eqref{eq:twotips}.
Indeed, if $q \in \d\Future(y) \cap \{ t < t_\caN\}$, 
  \eqref{eq:twotips} implies that there is $\bfv \in
  \d\Timelike^+_y = \Null^+_y$ with $\exp_y(\bfv) = q$. Then $s\mapsto
  \exp_y(s\bfv)$ for $0\le s\le 1$ is a radial null geodesic connector.
  By Part~4 in Proposition~\ref{prop: JMR}, one
  has $q = (\atimeF_y(\ux), \ux)$ for some $\ux \in \R^d$. 

Suppose next that $\gamma(t)=(t,x(t))$ is a causal curve that starts
at $y$ and arrives at $q$.  The curve $\gamma$ is a minimizing curve
for \eqref{eq: first arrival time}.  The dynamic programming principle
together with Parts~2 and 4 in Proposition~\ref{prop: JMR} imply that
$\gamma(t) = \big(\atimeF_y(x(t)), x(t)\big) \in \Lambda^+(y)$.  The
curve $\gamma$ must be null at all points where it is differentiable.
Otherwise it would be timelike and would cross the smooth manifold
$\Lambda^+(y)$ transversally.

Since $(\Lambda^+(y) \cap \caN)\setminus \{y\}$ is smooth and
characteristic, there is a unique null vector $\bfw$ tangent to
$\Lambda^+(y)$ with time component equal to 1.  This defines a smooth
vector field tangent to $\Lambda^+(y)$ and wherever $\gamma$ is
differentiable, $\gamma^\prime=\bfw$.  Therefore $\gamma$ agrees in
$t>t(y)$ with the integral curve of $\bfw$ through $q$ traced backward
in time and approaching $y$ as $t\searrow t(y)$.  This uniquely
determines $\gamma$.
\end{proof}

\begin{proposition}
  \label{prop: minimizing curves}
  \begin{enumerate}[align=left,labelwidth=0.8cm,labelsep=0cm,itemindent=0.8cm,
      leftmargin=0cm,label=\bf{(\roman*)}]
  \item 
    Any minimizing forward causal Lipschitz curve given in
    Proposition~\ref{prop: JMR} is a (reparametrization of a) null
    geodesic.
  \item If $y \in \caL$ and $x\in \R^d$, there is a future null geodesic connecting $y$ and $(\atimeF_y(x),x)$.

  \item For $y=(\ut, \ux)\in \caL$ and $x\ne \ux$,  
    $t_y(x)$ it the  first arrival time for forward causal curves at $x$.
    It   is also
    a first arrival time for waves starting from $y$.  
    The inclusion \eqref{eq:sharp} implies that 
   for $y=(\ut,\ux)\in \caL$ and a solution of $Pu=0$ whose Cauchy data at $\ut$ are supported
 at $\ux$, 
 $u=0$ on a neighborhood of $[0,t_y(x)[ \times \{\ux\}$.   There is such a $u$ so that
 $(t_y(x), \ux)\in {\rm supp}(u)$.
 \end{enumerate}
\end{proposition}

\begin{proof}
  Suppose $y \in \caL$, $x \in \R^d$ and $q =
  (\atimeF_y(x),x)$. Suppose $[0,S] \ni s \mapsto \gamma(s) =
  \big(t(s), x(s)\big)$ with $\gamma(0) = y$ and $\gamma(S) =q$
  achieves the infimum in \eqref{eq: first arrival time}.

  Need to show that $\gamma$ is a reparametrization of a geodesic
  on a \nhd of each point $\gamma(s)$.   Given $s$ choose
  a convex normal \nhd $\caN_s$ of $\gamma(s)$.  Proposition
  \ref{prop:normal2}  
  shows that the choice can be made so that $\ut:=\inf_s t_{\caN_s}>0$.
  
  For $s=0$
  choose $\delta>0$ so that $\gamma([0,\delta])\subset\caN$.   The
  dynamic programming programming principal implies that $\gamma$
  achieves the minimal arrival time from $\gamma(0)$ for all $0\le s\le S$.
  Proposition \ref{prop:local form of F} implies
  that for $0\le t\le t_{\caN_0}$, $\gamma$ is a reparametrization of 
  a radial geodesic from $\gamma(0)$.  A similar argument applies
  at the endpoint $s=S$.
  
  For $0<s<S$  choose $\delta>0$ so that $[s-\delta,s+\delta]\subset [0,S]$
  and $\gamma([s-\delta,s+\delta]\subset \caN_s$. The dynamic
  programming principle implies that $\gamma$ achieves the first arrival
  at $\gamma(s+\delta)$ from $\gamma(s-\delta)$.   Shrinking
  $\delta$ if necessary can suppose that $t \big(\gamma(s+\delta)\big) -t \big(\gamma(s-\delta)\big)<\ut$.  Then
    Proposition \ref{prop:local form of F} implies that the first
    arrival at $\gamma(s+\delta)$ from $\gamma(s-\delta)$
    is a reparametrizaion of the unique radial geodesic from $\gamma(s-\delta)$
    to $\gamma(s+\delta)$.   Thus $\gamma$ is a reparametrized geodesic
    on a \nhd of $\gamma(s)$.

  The second part combines Proposition~\ref{prop: JMR} and the first part.

  Third part. For solutions of $Pu=0$ the wavefront set $\WF (u)$ is invariant
   under the \bichar flow.  Suppose $q = (\atimeF_y(x), x)$ for some
   $x\in \R^d$. By 
   Proposition~\ref{prop: minimizing curves}, there is a future minimizing null geodesic $\rho$
   connecting $y$ and $q$. Associated with $\rho$ is a
   \bichar $\gamma$ that projects onto $\rho$.

   As in Example~\ref{ex:Minkowski-DD}.1, choose nonzero Cauchy data at $t(y)$ with support equal to $\{y\}$
   so that the point of $\gamma$ over $y$ belongs to $\WF (u)$. Then,
   the point of $\gamma$ over $q$ belongs to $\WF (u)$.  Therefore
   $q\in \singsupp u\subset \supp u$.  This solution with initial
   support at $y$ has support that reaches $q$, thus reaching $x$ in
   the minimal time $\atimeF_y(x)$.
\end{proof}
\begin{lemma}
  \label{lemma: alternative future}
  Suppose $y, q \in \caL$. The following three statements are equivalent
  \begin{enumerate}[align=left,labelwidth=0.8cm,labelsep=0cm,itemindent=0.8cm,
      leftmargin=0cm,label=\bf{(\roman*)}]
  \item \label{lemma: alternative future1} $q \in \Future(y)$.
  \item \label{lemma: alternative future2} A Lipschitz
    future timelike curve connects $y$ and $q$.
  \item \label{lemma: alternative future3} $q = (t,x)$ and $t > \atimeF_y(x)$.
  \end{enumerate}

  The following four statements are equivalent
  \begin{enumerate}[align=left,labelwidth=0.8cm,labelsep=0cm,itemindent=0.8cm,
      leftmargin=0cm,label=\bf{(\roman*)}]
  \item \label{lemma: closure future charact1}
    $q \in \ovl{\Future(y)}$.
  \item \label{lemma: closure future charact2}
    A smooth future  causal curve
    connects $y$ and $q$.
  \item \label{lemma: closure future charact3}
    A Lipschitz future  causal curve
    connects $y$ and $q$.
  \item \label{lemma: closure future charact4} $q = (t,x)$ and $t \geq  \atimeF_y(x)$.
  \end{enumerate}
\end{lemma}
\begin{proof}
  {\bf First set of equivalences}.
  \ref{lemma: alternative future1} $\Imply$
  \ref{lemma: alternative future2} is clear.

  \ref{lemma: alternative future2} $\Imply$
  \ref{lemma: alternative future3}. Suppose $\rho(s)$ is a Lipschitz future timelike
  curve connecing $y$ and $q = (t,x)$. Then, $t \geq \atimeF_y(x)$ by
  \ref{eq: first arrival time}.  Suppose $t=\atimeF_y(x)$. Then,
  $\rho$ is a causal curve achieving the infimum in \ref{eq: first
    arrival time}. By Proposition~\ref{prop: minimizing curves}, $\rho$ is
  a reparametrization of a null geodesic contradicting its timelike
  nature. Hence, $t> \atimeF_y(x)$.

  \ref{lemma: alternative future3} $\Imply$ \ref{lemma: alternative
    future1}. Suppose $q=(t,x)$ is such that $t >\atimeF_y(x)$.
    Definition~\ref{def: speedup slowdown metric} and
  Corollary~\ref{cor: adjusting speed} imply that 
    there
  exists a slower metric $g^\delta$, $0<\delta<1$, such that $t =
  \atimeFd{g^{\delta}}_y(x)$.  Then, by Proposition~\ref{prop: minimizing curves},
  there is a null geodesic $\gamma$ for $g^\delta$ that connects $y$
  and $q$. This curve is smooth and timelike for $g$.  Thus, $q \in
  \Future(y)$.

  \medskip
      {\bf Second set of equivalences}. \ref{lemma: closure future charact2} $\Imply$ \ref{lemma: closure
    future charact3} is clear. \ref{lemma: closure future charact3}
  $\Imply$ \ref{lemma: closure future charact4} is a consequence of
  Definition~\ref{def: first arrival time}.  \ref{lemma: closure
    future charact1} $\Equiv$ \ref{lemma: closure future charact4}
  follows from the first set of equivalences.

  \ref{lemma: closure future charact1} $\Imply$ \ref{lemma: closure
    future charact2}.  Suppose $q=(t,x) \in \ovl{\Future(y)}$. If $q \in
  \Future(y)$ then \ref{lemma: closure future charact2} holds by
  the first set of equivalences. If $q \in \d \Future(y)$, then $t
  = \atimeF_y(x)$. 
  Proposition~\ref{prop: minimizing curves} give a forward null geodesic
  connecting $y$ and $q$.
\end{proof}
Lemma~\ref{lemma: alternative future} implies that
\begin{align}
  \label{eq: boundary future}
  \d\Future(y) \,=\, \big\{ (t,x)\in \caL:\  t = \atimeF_y(x) \big\}.
\end{align}

\begin{remark}
  \label{remark: lemma: closure future charact}
  By Proposition~\ref{prop: minimizing curves} and Lemma~\ref{lemma:
    alternative future}, one has $\d \Future(y) \subset \Lambda^+(y)
  \subset \ovl{\Future(y)}$. However,  $\Lambda^+(y)$ can be different from $\d
  \Future(y)$ as future null geodesics starting at $y$ can enter
  $\Future(y)$.  Proposition~\ref{prop: JMR},
  Proposition~\ref{prop:local form of F}, and \eqref{eq: boundary
    future}, imply that future null geodesics (or their
  reparametrizations) all achieve minimization in \eqref{eq: first
    arrival time} {\it near} $y$.  Section~\ref{sec:Second_doldums_example} constructs examples of null geodesic connectors for
  points far from $y$ that arrive after the first arrival time.  
\end{remark}
By Lemma~\ref{lemma: alternative future},
$\Future(z) \subset \Future(y)$ if $z \in \Future(y)$.
The following lemma gives a sharper version.
\begin{lemma}
  \label{lemma: inclusion future sets}
  Let $y,z \in \caL$. 
  \begin{enumerate}[align=left,labelwidth=0.8cm,labelsep=0cm,itemindent=0.8cm,
      leftmargin=0cm,label=\bf{(\roman*)}]
  \item  One has $\Future(z) \subset \Future(y)$ if $z \in \ovl{\Future(y)}$.
  \item One has $\ovl{\Future(z)} \subset \Future(y)$ if $z \in \Future(y)$. 
  \end{enumerate}
\end{lemma}
For a proof see Corollary~1 in \cite[Chapter~14]{ONeill}.
Assertion (ii) implies that 
\begin{align}
  \label{eq:empty_intersection_boundary_future}
  \d\Future(z) \cap \d\Future(y) = \emptyset \ \ \si \ \  z \in \Future(y).
\end{align}

The following result extends part of Lemma~\ref{lem:timelikecross}
ouside of normal \nhds.
\begin{lemma}
  \label{lemma: tangent plane future-set}
  Suppose $y\in \caL$, $q \in \d\Future(y)$, and $\d\Future(y)$ is differentiable at  $q$. Then, $\d\Future(y)$ is characterisitc at $q$.
\end{lemma}
\begin{proof}
  Choose $\gamma(s)$ to be a future null geodesic connecting $y$ to
  $q$ with $\gamma(\us)=q$.  For $0<\delta \ll 1$ consider $\gamma$ on
  $\us-\delta\le s\le \us+\delta$.  For $\delta$ sufficiently small,
  this is the first arriving causal curve from $\gamma(\us-\delta)$
  and Lemma \ref{lem:timelikecross} applies. Lemma~\ref{lemma:
    inclusion future sets} gives $\ovl{\Future(\gamma(\us-\delta))}
  \subset \ovl{\Future(y)}$.  As $\d \Future(\gamma(\us-\delta))$ and $\d
  \Future(y)$ are both differentiable at $q$, their tangent planes coincide.
  Lemma \ref{lem:timelikecross} implies they are characteristic, as $\d \Future\big(\gamma(\us-\delta)\big)$ coincides with $\Lambda^+\big(\gamma(\us-\delta)\big)$ in a \nhd of $q$. 
\end{proof}

\subsection{Foliation of a future set and local description near the tip}
\label{eq: foliation}
Given  a timelike curve $\gamma$  through
$y$,
 \eqref{eq:empty_intersection_boundary_future} provides
a foliation of $\Future(y)$  anchored at $\gamma$.
\begin{lemma}
  \label{lemma: foliation future past}
  Suppose $y =(\ut, \ux) \in \caL$,  $\ovl{t}>\ut$, and,
  $[\ut,\ot]
  \ni t \mapsto \gamma(t) = (t,x(t))$ is a future timelike curve with 
  $\gamma(\ut) = y$. Then,
  \begin{align*}
    \Future(y) \setminus \Future(\gamma(\bt))\, =\,
      \bigsqcup_{\ut < t \leq  \bt} \d\Future(\gamma(t)) \,.
  \end{align*}
 \end{lemma}
Similar formulas hold for $\Past(y)$.
The geometry is sketched in Figure~\ref{fig: tip emission}. 

\begin{proof}
  Suppose $q \in \Future(y)$ and define the continuous function $f(t) =
  \atime^-_q (x(t))-t$. If $f(\bt)>0$ then $q \in
  \Future(\gamma(\bt))$. If $f(\bt)\leq 0$, there exists $t \in
          ]t_z,\bt]$ such that $f(t) =0$ by the intermediate value
            theorem since $f(\ut)>0$. This means that $q \in
   \d\Future(\gamma(t))$. Therefore
   \begin{align*}
     \Future(q)  \setminus \Future(\gamma(\bt)) \ \subset\  \bigsqcup_{\ut < t \leq  \bt} \d\Future(\gamma(t)) \,.
   \end{align*}
    Lemma~\ref{lemma: inclusion
     future sets}
implies 
opposite inclusion.
\end{proof}
\begin{figure}
\begin{center}
  \resizebox{5.4cm}{!}{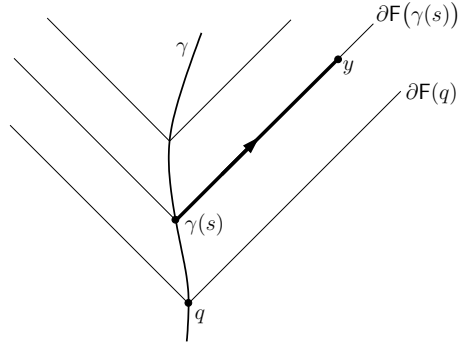}
  \caption{Foliated geometry of $\Future(y)$ near $y$.}
  \label{fig: tip emission}
\end{center}
\end{figure}

\section{\Dcones}
\label{sec:dcones}
\begin{definition}
  \label{def: double-cone}
  Suppose $y,y' \in \caL$. Define the {\bf \dcone} with past-tip $y$ and
  future-tip $y'$ as the open set
  \begin{align*}
    \DC (y,y') =\Future(y) \cap  \Past(y').
  \end{align*}
\end{definition}
Then, $\DC (y,y')$ is nonempty if and only if $y' \in \Future(y)$. One has 
\begin{align*}
 \DC (y,y') =  \{ (t,x) \in \caL: \  \atimeF_{y}(x) < t < \atimeP_{y'}(x)\}.
\end{align*}
In Definition~\ref{def: dcone intro}, the notation $\DC(\gamma) = \DC (y,y')$ is used, with $\gamma$ a timelike curve connecting $y$ and $y'$.
The \dcone is the intersection of two cones in the Minkowski case. It resembles a cone near its two tips.
\begin{remark}
  The set $\DC (y,y')$ is of importance in the Lorentzian
    literature; see \eg \cite{HawkingEllis1973}.  The word `diamond'
    is often used for the description of this set in 
    \cite{Minguzzi2019}. Here, considering the doldrums examples
    of Sections~\ref{sec:First_doldrum_example} and
    \ref{sec:Second_doldums_example}, we  refrain from using
    the word ‘diamond’, finding it misleading because of the shape one
    immediatly imagines. 
\end{remark}

\begin{lemma}
  \label{lemma: boundedness double cone}
  $\phantom{-}$
 \begin{enumerate}[align=left,labelwidth=0.8cm,labelsep=0cm,itemindent=0.8cm,
  leftmargin=0cm,label=\bf{(\roman*)}]
 \item  \Dcones are bounded. 
 \item If  $y, y', z, z' \in \caL$ with $z, z' \in \ovl{\DC (y,y')}$,  then, $\DC
   (z,z') \subset \DC (y,y')$.
 \end{enumerate}
\end{lemma}
\begin{proof}
(i).    If $\DC (y,y') \neq \emptyset$ with $y'= (t',x')$, then $\DC
  (y,y') \subset \Future(y) \cap \{ t \leq t'\}$. Boundedness follows from Lemma~\ref{lemma: boundedness tip future and past}. 
  (ii)  Follows from  Lemma~\ref{lemma: inclusion future sets}.
\end{proof}

A consequence is the following proposition.
\begin{proposition}
  \label{prop: dcones upperbound for Oinfty}
  If $y, y' \in \caL$ and $\O \subset \DC (y,y')$ is an open set, then, 
  $\cad^{\infty}\O  = \cah^{\infty}\O \subset \DC (y,y')$.
\end{proposition}
\begin{proof}
   Assume $\DC(y,y') \neq \emptyset$. If $q \in \cah^1\O$ then there
   exists a  future timelike curve joining two points $z$ and $z'$ of
   $\O$ passing through $q$. Hence $q \in \DC (z,z') \subset \DC
   (y,y')$ by Lemma~\ref{lemma: boundedness double cone}-(ii). This gives
   $\cah^1\O\subset \DC (y,y')$. Iterations give 
   $\cah^n \O  \subset \DC (y,y')$.   
 \end{proof}
Hence, \dcones are an upper-bound for set of points that can be
reached by iterated timelike homotopies from all open subsets of the
\dcone.

\subsection{\Dcones with special temporal symmetries}     
\label{sec:Dcones with special temporal symmetries}

\begin{proposition}
  \label{prop: homotopic double-cones}
 Suppose that the Lorentz metric $g = - dt^2 \, +\, \gR(.,.)$ with $\gR$ a
  time independent Riemannian metric on $\R^d$.  Suppose
  $\y{0}=(-\sft,0_{\R^d})$ and $\y{1}=(\sft,0_{\R^d})$ with $\sft>0$. Suppose
  $\O$ is an open set such that $]\y{0}, \y{1}[ \subset \O \subset
      \DC(\y{0},\y{1})$. Then, $\DC(\y{0},\y{1}) = \cah^1\O=
      \cah^\infty\O$.
\end{proposition}
By changing $x$-coordinates and using the time invariance of the metric, the result extends to
$\y{0}=(\yt{0},x)$ and $\y{1}=(\yt{1},x)$ with $\yt{0} < \yt{1}$ and $x \in \R^d$

For metrics as in Proposition~\ref{prop: homotopic double-cones},
Lorentzian objects frequently have descriptions in terms of familiar
Riemannian objects.  Subsequent sections construct examples of such
metrics with interesting properties that can be counterintuitive. A first example concerns a \dcone
with a non vertical axis for which
$\cah^\infty \caO \neq \DC$ and $\DC \not\subset Z$; see
Section~\ref{sec:First_doldrum_example}.
Two examples concern
double cones with axis on vertical lines and depend on
Proposition~\ref{prop: homotopic double-cones} and show that $\DC$ can be a strict subset of $Z$ and may have nonzero genus; see Section~\ref{sec:Second_doldums_example}.

\begin{proof}
  Set $\DC=\DC(\y{0},\y{1})$.  Denote by $d_\gR(x,x')$ the geodesic
  distance from $x$ to $x'$ in $\R^d$.  Then,  $(t,x)=y \in \Future(\y{0})$ if and only if $t>d(x,0)-\sft$, and
  $ y \in \Past(\y{1})$ if and only if $t<\sft-d_\gR(x,0)$.
  Therefore,
\begin{align*}
\DC = \big\{(t,x): \ |t|< \sft -d_\gR(x,0)\big\}.
\end{align*}
In particular, 
$(t,x) \in  \DC \  \Equiv\  [-|t|,|t|]\times \{x\} \subset \DC.$
It suffices to prove that $\DC \subset \cah^1\O$.  Then,
  Proposition~\ref{prop: dcones upperbound for Oinfty} implies that
  $\DC \subset \cah^1\O\subset \cah^\infty\O\subset\DC$.

\vskip.1cm {\bf Step 1.  Homotopy of piecewise
  smooth timelike curves.}  For $z=(t_z,x_z)\in \DC$, the mirror image
$z^*=(-t_z,x_z)\in \DC$.  Construct a timelike homotopy $\XX$ reaching
$z$ proving $z
\in \cah^1 \O$.  It uses two times $0< \uT< T< \sft$ that 
depend on $z$.

The vertical segment $[-|t_z|,|t_z|]\times\{x_z\}\in\DC$.  If
$t_z\ne 0$ define $\uT=|t_z|$.  If $t_z=0$ choose $\uT>0$ so that
$(\uT,x_z)\in \DC$. Equivalently, $(-\uT,x_z)\in
\DC$. In either case, $\sft-\uT>d_\gR(0,x_z)$. Choose $\eps>0$
so that $T= \uT + (1+\eps)d_\gR(0,x_z)< \sft$.  Then $(\uT, x_z)\in \d
\Past_{\!g'} (T,0)$ for the slower metric $g' =
g^{1/(1+\eps)}= - (1+ \eps)^{-2} dt^2 + \gR$ from  Appendix~\ref{app:speeding up and slowing down metrics} and Remark~\ref{remark:
  speedup slowdown metric}.  Choose a future reparameterized null geodesic,
$\rho(t) =(t,u(t))$, connecting $(\uT,x_z)$ to $(T,0)$ for the metric
$g'$.  The curve $\rho$ and its mirror image $\rho^*$ traced
backward are timelike for $g$.  $\rho^*$ is future oriented and connects $(-T,0)$ to  
$(-\uT,x_z)$. 

Define a continuous homotopy 
$$
 [0,1] \times [-T, T]\,\ni\,
 (\sigma,t)
 \ \ 
 \mapsto\ \
\XX(\sigma, t)\,=\, (t,x(\sigma,t))
$$
of piecewise smooth timelike curves with
intial curve, $\XX_0$, and boundary, $\Ends{\XX}$, in $\O$ and 
with   final curve including the vertical segment
 $[-\uT,\uT]\times\{x_z\}$ passing through $z$.  Unrolling the cylinder over
$\rho$ yields a curvilinear trapezoidal region with vertical edges sketched in
Figure~\ref{fig: homotopy-curtain}.

\begin{figure}
  \begin{center}
    \resizebox{4cm}{!}{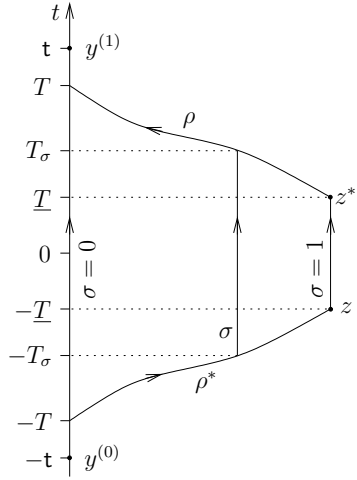}
    \caption{Flattened symmetric  `trapezoid'.}
  \label{fig: homotopy-curtain} 
  \end{center}
\end{figure}
The initial curve of the homotopy is the vertical segment on the left connecting
$(-T,0)$ to $(T,0)$.  
The final curve follows the boundary counterclockwise from lower left
corner to upper left corner.  It passes through $z$ in the vertical
segment on the right.  To be more precise, set $T_\sigma = (1-\sigma)T + \sigma \uT \geq \uT$
and
\begin{align*}
  \XX(\sigma,t)=\begin{cases}
  \rho^*(t)=(t,-u(-t))
  &\pour \ -T\leq t\leq - T_\sigma,\\
  \big(t,u(T_\sigma)\big)
  & \pour\ - T_\sigma\leq t \leq T_\sigma,\\
  \rho(t)=(t,u(t))
  & \pour\ \ T_\sigma\leq t \leq T.
  \end{cases}
\end{align*}
For $0\le \sigma\le 1$ the curve starts at $t=-T$ and
follows $\rho^*(t)$ for $\sigma(T-\uT)$ units of time. 
For the next $2 T_\sigma$ units of time, the
curve is vertical. Hence, the $x$-component of
$\XX(\sigma,t)$ is independent of $t$ on the band $|t|\le \uT$.
For the final $\sigma(T-\uT)$ units of time the curve follows
$\rho(t)$, reaching $(T,0)$.  The homotopy has mirror symmetry with
respect to $t=0$.  The time derivative is discontinuous when $t=\pm
T_\sigma$.

\medskip
 {\bf Step 2.  Smoothing the homotopy.}  
   Choose $0\le j\in \Cinfc(\R)$ with
 $\int j=1$ and   $\supp j\!\subset [-1,1]$.
  
 Smooth  $\d_t x(\sigma, .)$ as follows.
 Suppose $0<\eta<\uT$ to be chosen small in what follows. For  $-T\le t\le T$ the smoothed value 
 $\zeta_\eta(\sigma,t)$ is defined to be a weighted  average
 of $\d_t x(\sigma, .)$ on the interval 
 $I_\eta(t)=[c_\eta(t)-\eta,c_\eta(t)+\eta]\subset [-T,T]$
 of length
 $2\eta$ with center 
 $c_\eta(t) = (T-\eta)t/T$.
 When $t=-T$ the inteval is $[-T,-T+2\eta]$ and
 when $t=T$ the interval is $[T-2\eta,T]$.

   The interval $I_\eta(t)$ is parameterized by $c_\eta(t) +\eta s$
   with $-1<s<1$.  The smoothed time derivative is given in terms of
   $\d_tx(\sigma,.)$ by an integral operator with smooth kernel,
 \begin{align*}
 \zeta_\eta(\sigma,t)  = 
 \int     j(s)\,  (\d_t x)(\sigma, c_\eta(t)+\eta s)\, ds,
 \qquad
 \forall n, \ \
 \d_t^n\zeta_\eta\in \Con^0([0,1]\times [-T,T]).
 \end{align*}
 Note that as $\eta<T_\sigma$, then $\zeta_\eta=\d_t x=0$ for $|t| \leq T_\sigma-\eta$. 
 Define $x_\eta(\sigma,.)$  by
   \begin{equation}
   \label{eq:recipe}
   \d_t x_\eta(\sigma,t) = \zeta_{\eta}(\sigma,t) ,
   \qquad
   x_\eta( \sigma, 0)  = x( \sigma,0).
   \end{equation}
Define $\XX^\eta(\sigma,t)=(t, x_\eta(\sigma,t))$.

  If one replaces $\zeta_\eta$ by $\d_tx(\sigma,t)$ in
   \eqref{eq:recipe} the resulting problem has solution
   $x(\sigma,t)$.
   Since $x_\eta(\sigma,t) = x(\sigma,t)=x( \sigma,0)$ for $|t| \le T_\sigma-\eta$,
    the final curve ($\sigma=1$)
    passes through $z$  as $T_1 -\eta = \uT-\eta>0$.  
    
    Define $S\subset [0,1]\times [-T,T]$ to be the locus of jump discontinuities of $\d_t x$.  Define
  \begin{align*}
  c_1:=\Norm{\d_t x} { L^\infty([0, 1]\times[-T, T]  )  } , 
  \qquad
  c_2:=\Norm{\d_t^2 x}{L^\infty (  (   [0, 1]\times[-T, T] )  \setminus S )}.
  \end{align*}
  For $\sigma\in [0,1]$, define $S_\sigma\subset [-T,T]$ to be the 
 locus of jump discontinuities of $t\mapsto x(\sigma,t)$.
  Then,
  \begin{align*}
    \Norm{(\d_t x -\zeta_\eta)(\sigma, t)}{}
    \lesssim
     \begin{cases}
       c_1 & \pour \ \sigma \in [0,1] \
      \et \ \dist(t,S_\sigma) \leq 2\eta, \\
      \eta c_2  & \pour \  \sigma \in [0,1] \
      \et \ \dist(t,S_\sigma) \geq  2\eta.
    \end{cases}
  \end{align*}
 Subtracting the equations satisfied by $x(\sigma,t)$ and
 $x_\eta(\sigma,t)$ shows that $|\XX^\eta-\XX|\lesssim \eta$ on $[0,1]
 \times [-T,T]$. Therefore there  is an $\eta_1>0$ so that 
 for $0<\eta\leq
 \eta_1$,
  the endpoints and initial curve of $\XX^\eta$ are in $ \O$.
   
  To show that the curves $\XX^\eta(\sigma,.)$ are timelike for $g$,
  employ slower metrics $g^\delta$ from
  Appendix~\ref{app:speeding up and slowing down metrics}.
  Remark~\ref{remark: speedup slowdown metric} shows that the cone of
  future timelike vectors $\Timelike^+_q$ increase as $\delta$
  increases.  Denote by $\Timelike_q^{\delta,+}$ the cone associated
  with $\gd$. Since $\XX$ is timelike, $\d_t \XX(\sigma, t) \in
  \Timelike_{\XX(\sigma, t)}^+$.  A continuity and compactness
  argument applied to the tangents implies that there is a $\delta<1$
  so that $\d_t \XX(\sigma, t) \in \Timelike_{\XX(\sigma,
    t)}^{\delta,+}\subset \Timelike_{\XX(\sigma, t)}^{+}$.

     Uniform continuity  of
  $\XX(\sigma,t)$ and $g$  imply that if 
  $1>\delta$
   there
  is $0< \eta_2(\delta)< \eta_1$   so that
   \begin{align}
   \label{eq:uc}
   \forall
   (\sigma,t,\ut)\in [0,1]\times [-T,T]\times [-T,T],
   \qquad
   \forall
   |t - \ut| \leq \eta_2
   \ \ \Imply \ \ 
   \Timelike_{\XX(\sigma, t)}^{\delta,+}
   \subset
   \Timelike_{\XX(\sigma, \ut)}^{+}
   .
 \end{align}
For $(\sigma,\ut)\in [0,1] \times [-T,T]$, $\d_t \XX^\eta(\sigma,\ut)$
is a convex linear combination of the tangent vectors $\d_t
\XX(\sigma,t)$ with $c_\eta(t)-\eta\le \ut\le c_\eta(t)+\eta$.  For
$\eta<\eta_1$ the tangents lie in $\Timelike_{\XX(\sigma,
  t)}^{\delta,+}$.  Equation \eqref{eq:uc} implies that for
$\eta<\eta_2$, they belong to $\Timelike_{\XX(\sigma, \ut)}^{+}$.
Therefore, their convex combinations belong to $\Timelike_{\XX(\sigma, \ut)}^{+} $,
proving that
$\XX(\sigma,.)$ is timelike for $g$.
\end{proof}

\subsection{Plunging condition}
\label{sec:Plunging_condition}
The following lemma asserts that future causal curves do not leave
$\Future(y)$ after entering it. 
\begin{lemma}
  \label{lemma: plunging0}
  Suppose  $\gamma(t)=(t,x(t))$ is a future causal curve.\\[-15pt]
  \begin{enumerate}[align=left,labelwidth=0.8cm,labelsep=0cm,itemindent=0.8cm,
  leftmargin=0cm,label=\bf{(\roman*)}]
\item
  Suppose $y \in \caL$.  If
  $\gamma$ enters $\Future(y)$, there exists $T\in \R$ such that $\gamma(t) \in
  \Future(y) \  \Equiv  \ t > T$.  One says that $\gamma$ enters $\Future(y)$ at time
  $T$.

\item Suppose $\y{0}$, $\y{1} \in \caL$ and $\gamma$ enters $\DC(\y{0}, \y{1})\neq
  \emptyset$. Then, $\gamma$ enters $\Future(\y{0})$ at some time $T$ and
  $\gamma(T) \in \d\Future(\y{0}) \cap \Past(\y{1})$. One says that $\gamma$
      enters $\DC(\y{0}, \y{1})$ at time $T$.
  \end{enumerate}
\end{lemma}
The proof is based on Lemma~\ref{lemma: inclusion future sets}.
A similar statement holds for past causal curves.
\begin{definition}
  \label{def: plunging}
  Suppose $z, y \in \caL$ with $z \notin\Future(y)$.
  \begin{enumerate}[align=left,labelwidth=0.8cm,labelsep=0cm,itemindent=0.8cm,
  leftmargin=0cm,label=\bf{(\roman*)}]
 \item One says that $\ovl{\Future(z)}$ {\bf plunges} into
   $\Future(y)$ if  all future
    causal curves starting at $z$ enter $\Future(y)$.
 \item One says that $\Lambda^+ (z)$  {\bf plunges} into
   $\Future(y)$ if  all future null geodesics starting at $z$
   enter $\Future(y)$.
  \end{enumerate}
\end{definition}
Similar definition apply for  $\ovl{\Past(z)}$,  $\Lambda^- (z)$, and $\Past(y)$.

\begin{definition}
  \label{def: plunging bis}
  Suppose $z, \y{0}, \y{1} \in \caL$ with $z \notin\Future(\y{0})$ and $\DC= \DC(\y{0}, \y{1}) \neq \emptyset$.
  \begin{enumerate}[align=left,labelwidth=0.8cm,labelsep=0cm,itemindent=0.8cm,
  leftmargin=0cm,label=\bf{(\roman*)}]
 \item One says that $\ovl{\Future(z)}$ {\bf plunges} into
   $\Future(\y{0})$ {\bf through} $\DC$ if  all future
    causal curves starting at $z$ enter $\DC$.
 \item One says that $\Lambda^+ (z)$  {\bf plunges} into
   $\Future(\y{0})$ {\bf through} $\DC$ if  all future null geodesics starting at $z$
   enter $\DC$.
  \end{enumerate}
\end{definition}

\begin{lemma}
  \label{lemma: plunging1}
  Suppose $y, z \in \caL$ with $z\notin \Future(y)$. The following statements are equivalent.\\[-0.7cm]
  \begin{enumerate}[align=left,labelwidth=0.8cm,labelsep=0.1cm,itemindent=0.9cm,
  leftmargin=0cm,label=\bf{(\roman*)}]
  \item $\ovl{\Future(z)}$ plunges into $\Future(y)$.
  \item $\Lambda^+(z)$ plunges into $\Future(y)$.
  \item $\Lambda^+(z)\setminus
    \Future(y)$ is compact.
    \item $\ovl{\Future(z)} \setminus
      \Future(y)$ is compact.
  \end{enumerate}
\end{lemma}
Similar equivalent statements hold for $\ovl{\Past(z)}$ and $\Lambda^-(z)$
plunging into $\Past(y)$. Such phenomena do not occur for the Minkowski
metric.  Section~\ref{sec:Second_doldums_example} contains an
example where they occur.
\begin{proof}
  {\bf (i)} $\Imply$ {\bf (ii)} is clear.

  {\bf (ii)} $\Imply$ {\bf (iii)}. $\Lambda^+(z)$ is the union of all
  future null geodesics starting at $z$. Such null geodesics can be
  labeled by an initial unit tangent vector
  $\bfv \in \Null^+_z \cap {\mathbb S}^d$. The latter is
  compact. Since all future null geodesics starting at $z$ enter
  $\Future(y)$ then, by compactness, there exists $T>0$ such that all
  these null geodesics are in $\Future(y)$ for $t > T$. This implies
  that $\Lambda^+(z)\cap \{t>T\}\subset \Future(y)$, giving
  $\Lambda^+(z)\setminus \Future(y)$ bounded by Lemma~\ref{lemma:
    boundedness tip future and past}, hence the result.

  {\bf (iii)} $\Imply$ {\bf (iv)}.  Define $K = \Lambda^+(z) \setminus
  \Future(y)$ and $\pi K$ where $\pi$ is the projection $\pi(t,x) =x$. $\pi K$ is compact in $\R^d$.
  Suppose $(t,x) \in \ovl{\Future(z)} \setminus \Future(y)$, then
  $\atimeF_z(x) \leq t \leq \atimeF_y(x)$. Thus $(\atimeF_z(x), x) \in
  K$, that is $x \in \pi K$. This gives $t \leq \sup_{\pi K} \atimeF_y <
  \infty$.  $\ovl{\Future(z)} \setminus \Future(y)$ is bounded by
  Lemma~\ref{lemma: boundedness tip future and past}, hence the
  result.

  {\bf (iv)} $\Imply$ {\bf (i)}. Write $z= (t_z,x_z)$. As $K' =\ovl{\Future(z)} \setminus
  \Future(y)$ is compact,  $K' \subset \{t \leq T\}$
  for some $T> t_z$. Suppose $\gamma(t) = (t, x(t))$ is a future causal curve  starting at $z$. 
  Then $\gamma(t) \notin K'$ if
  $t>T$: $\gamma$ enters $\Future(y)$.
\end{proof}

\begin{lemma}
  \label{lemma: plunging2}
  Suppose  $\y{0}, \y{1}, z \in \caL$ with $z\notin\DC(\y{0},
  \y{1})\neq \emptyset$.
  
  {\bf 1.}  The following statements are equivalent.\\[-0.7cm]
  \begin{enumerate}[align=left,labelwidth=0.8cm,labelsep=0.1cm,itemindent=0.9cm,
      leftmargin=0cm,label=\bf{(\roman*)}]
  \item $\ovl{\Future(z)}$ plunges into $\Future(\y{0})$ through $\DC(\y{0}, \y{1})$.
  \item $\ovl{\Future(z)}  \setminus \Future(\y{0})\subset \Past(\y{1})  $.
  \end{enumerate}
  
   {\bf 2.}  The following two statements are equivalent.\\[-0.7cm] 
  \begin{enumerate}[align=left,labelwidth=0.8cm,labelsep=0.1cm,itemindent=0.9cm,
      leftmargin=0cm,label=\bf{(\roman*)}]
  \item  $\Lambda^+(z)$ plunges into  $\Future(\y{0})$ through $\DC(\y{0}, \y{1})$.
  \item $\Lambda^+(z) \setminus \Future(\y{0})\subset \Past(\y{1})  $.
  \end{enumerate}  

  \end{lemma}
Similar equivalent statements hold for $\ovl{\Past(z)}$ plunging into
$\Past(\y{1})$ through $\DC(\y{0}, \y{1})$.
\begin{proof} 

{\bf 1.}  {\bf (i)} $\Imply$ {\bf (ii)}. Any future causal curve
$\gamma$ starting at $z = (t_z, x_z)$ enters $\DC(\y{0}, \y{1})$ at
some time $T_\gamma>t_z$ and $\gamma([t_z, T_\gamma]) =
\gamma([t_z,+\infty[) \setminus \Future(\y{0}) \subset \Past(\y{1})$
    by Lemma~\ref{lemma: plunging0}. The result follows as
    $\ovl{\Future(z)}$ is the union of all future causal curves
    starting at $z$ by Lemma~\ref{lemma: alternative future}.
 
{\bf (ii)} $\Imply$ {\bf (i)}. $\ovl{\Future(z)} \setminus \Future(\y{0})$
is compact since bounded by Lemma~\ref{lemma: boundedness tip future
  and past}. Thus $\ovl{\Future(z)}$ plunges into $\Future(\y{0})$ by
Lemma~\ref{lemma: plunging1}, meaning all future causal curves starting
at $z$ enter $\Future(\y{0})$. The point where such a curve enters
$\Future(\y{0})$ is in the open set $\Past(\y{1})$. Therefore, the curve
enters $\DC(\y{0}, \y{1})$.

{\bf 2.}  Similar to the proof of {\bf 1.}
  \end{proof}
\begin{remark}  
The statements in {\bf 1} clearly imply their counterparts in {\bf 2}.
The converse does not hold. In Proposition~\ref{prop: DC with hole no
  plunging} and Figure~\ref{fig: doldrums3} in
Section~\ref{sec:Second_doldums_example} we construct an example where
$\Lambda^+(z)$ plunges into $\Future(\y{0})$ through $\DC(\y{0},
\y{1})$ but $\ovl{\Future(z)}$ does not plunge into $\Future(\y{0})$
through $\DC(\y{0}, \y{1})$.
\end{remark}

The following lemma provides a situation in which the plunging of
$\Lambda^+(z)$ through a \dcone cannot occur.  Denote by $\pi:
\R^{1+d} \to \R^d$ the projection $\pi(t,x) = x$.
\begin{lemma}
  \label{lemma:non_plunging_argument}
  Suppose $\y{0}, \y{1}, z =(t_z, x_z) \in \caL$ with $z\notin \DC =
  \DC(\y{0}, \y{1})\neq \emptyset$. Suppose that $x_z \notin \ovl{\pi
    \DC}$. Then, $\Lambda^+(z)$ does not plunge into
 $\Future(\y{0})$ through  $\DC(\y{0}, \y{1})$.
\end{lemma}
\begin{proof}
  The result holds if $z \in \Past(\y{1})^c$.  The proof is by contradiction.
  Suppose that $z \in
  \Past(\y{1})$ and that $\Lambda^+(z)$ plunges into $\Future(\y{0})$ through $\DC(\y{0}, \y{1})$.
  Set $G = \Lambda^+(z) \setminus \Future(\y{0})$. It is compact as it is
  closed and a subset of the bounded set $\{t\geq t_z\} \cap
  \Past(\y{1})$ by Lemma~\ref{lemma: plunging2}. Since $x_z \in \pi G
  \setminus \ovl{\pi \DC}$, there exists $\udl{x} \in \pi G$ such that
  \begin{align*}
    \dist(\udl{x}, \ovl{\pi \DC}) = \sup_{x \in \pi G} \dist(x, \ovl{\pi \DC})>0.
  \end{align*}
  Here, $\dist$ denotes the Euclidean distance.
  From the local geometry of $\Lambda^+(z)$ near $z$ observe that $x_z$
 is in the interior of $\pi G$. Thus, $\udl{x} \neq x_z$.  Set
  $\udl{t} = \atimeF_z (\udl{x})$. Then, $\udl{z} = (\udl{t}, \udl{x}) \in
 G$. 

 Consider a future null geodesic  $\rho(t)$ starting  at  $z$ and  reaching
 $\udl{z}$. By assumption it enters $\DC$ at $t=T$ for some $T>0$, and thus it enters $\Future(\y{0})$ at $t=T$; see Lemma~\ref{lemma: plunging0}. As $\rho(\udl{t}) = \udl{z}$ then $\udl{t} \leq T$ since $\udl{z} \in G$. As $\udl{x} \notin \ovl{\pi \DC}$ one has $\udl{z} \notin \ovl{\DC}$ yielding $\udl{t} < T$. Thus, $\udl{z} \notin \d\Future(\y{0})$. 

  Consider $(x_n)_n \subset \R^d \setminus \pi G$ such that $x_n \to
  \udl{x}$.  Set $t_n = \atimeF_z(x_n)$. Then, $z_n = (t_n, x_n) \in
  \Lambda^+(z) \setminus G$. Hence, $z_n \in \Future(\y{0})$. Yet, as $z_n \to \udl{z}$ and $\udl{z} \notin \d\Future(\y{0})$, a contradiction. 
\end{proof}

\section{\Dds, \dcones, and doldrums examples}
\label{sec:Double-cones_dds}
 In this section, one assumes that the unique continuation property of
 Definition~\ref{def:unique_continuation} holds across
 all noncharacteristic hypersurfaces.

 Suppose $\gamma(s)$ is a future oriented timelike curve. 
 Definition~\ref{def:domain_of_determination2} yields
 \begin{align*}
    Z_{\gamma(]a,b[)} = \mathop{\cap} Z_\O,
 \end{align*}
 for $a< b$, where the intersection is over for all open \nhds $\O$ of
 $\gamma(]a,b[)$. Similarly, and similarly to \eqref{eq:dn_gamma}, set
    \begin{align*}
      \cah^n \big(\gamma(]a,b[)\big) = \mathop{\cap} \cah^n \O,
    \qquad \cad^n \big(\gamma(]a,b[)\big) = \mathop{\cap} \cad^n \O, \quad n\in \N^* \cup \{\infty\}.
 \end{align*}

    Theorem \ref{theorem:ZO homotopy} and  Proposition~\ref{prop:
  homotopic double-cones} imply the following result.
\begin{proposition}
  \label{prop: Dcones DD split metric}
  Suppose $g = - dt^2  + \gR$, with $\gR$ a time independent
  Riemannian metric on $\R^d$, and $\gamma(t) = (t, \x{0})$, with fixed $\x{0} \in \R^d$.  If $a <
  b$ then $\DC\big(\gamma(a), \gamma(b)\big)\subset Z_{\gamma(
    ]a,b[)}$.
\end{proposition}

Suppose $\gamma$ is a timelike curve in the Minkowski space
$\Minkowski^{1+d}$. Example~\ref{ex:Minkowski-DD} proves that 
$Z_{\gamma( ]a,b[)} = \DC\big(\gamma(a), \gamma(b)\big)$.
Theorem~\ref{thm:small_double_cones}, the first result of the section, proves
that $Z_{\gamma(]a,b[)}=\DC(\gamma(a), \gamma(b))$ for $|b-a|$ is \suff small.
This
result does not require any special metric structure. Second, we
construct an example of a metric $- dt^2 + \gR$ and a timelike curve
$\gamma(s)$, for which $\DC(\gamma(a), \gamma(b))$ is not a \dd of
$\gamma(]a, b[)$. That is $Z_{\gamma(]a,b[)}$ is striclty smaller than
    $\DC(\gamma(a), \gamma(b))$.  Third, we construct an example  of a metric $- dt^2 + \gR$ for
    which $Z_{\gamma(]a,b[)}$ is strictly larger than $\DC(\gamma(a),
    \gamma(b))$.  Necessarily in these examples $|b-a|$ is not small,
    and $\gamma(s)$ is not a time axis in view of
    Proposition~\ref{prop: Dcones DD split metric} and
    Theorem~\ref{thm:small_double_cones}.
These two examples show that 
 $\DC(\gamma(a),\gamma(b))$ is an unreliable candidate for 
 $Z_{\gamma(]a,b[)}$.  
\subsection{Small \dcones}
\label{sec:small_dcones_are_exact_dds}

\begin{theorem}
  \label{thm:small_double_cones}
  Suppose $I \subset \R$ is an interval and $\gamma: I \to \caL$ is a
  timelike curve. For all  $s_0 \in I$, there exists $\eps>0$ such that $Z_{\gamma(]a,b[)} = \DC(\gamma(a),
  \gamma(b)) = \cah^1 \big(\gamma(]a,b[)\big)= \cah^\infty \big(\gamma(]a,b[)\big)$ if $s_0-\eps \leq a< b \leq s_0+ \eps$.
\end{theorem}

 Proposition \ref{prop:first_doldrum} 
(resp.
 \ref{prop:second_doldrum})
 gives examples  of large double cones where $Z_{\gamma(]a,b[)}$ is strictly smaller 
 (resp. strictly larger) than $\DC$.
When Theorem~\ref{thm:small_double_cones} applies then
 $Z_{\gamma(]a,b[)}$ only depends on the endpoints $\gamma(a)$ and
 $\gamma(b)$.

 \begin{proof}
   
   {\bfseries Step~0. How small is small?}
   
    Parameterize $\gamma$ as $I \ni t \mapsto\gamma(t) =  (t,x(t))$.  
    Without loss of generality  take 
   $I  = [-T,T]$  with  $T>0$ and $s_0=0$.

   Choose   a convex normal \nhd  $\caN_1$ of $\gamma(0)$
   as  in Proposition~\ref{prop:normal2}.  Choose $\mu>0$ so
   that the future  $\Future\big(\gamma(0)\big)\cap \{|t|<\mu\}\subset\caN_1$  and past
   $\Past\big(\gamma(0)\big)\cap \{|t|<\mu\}\subset \caN_1$ are described in 
   Proposition \ref{prop:local form of F}.
   Choose a convex normal \nhd of $\gamma(0)$
   that is contained in the band
   \begin{equation}
     \label{eq:NinBand}
     \caN\ \subset\ \caN_1\,\cap\, \{|t|<\mu\}.
   \end{equation}
   Choose $\ua$ and $\ob$ so that $-T< \ua < 0 < \bb< T$, and so that
   $\cushion = \DC(\gamma(\ua), \gamma(\bb)) \subset \caN$.  $\cushion$ is
   called the {\it cushion \dcone}.
  
  With $\pi: \R^{1+d} \to \R^d$  the projection $\pi(t,x) = x$,
   \begin{align}
     \label{eq:step0_proj}
     x(0) \notin
     \ovl{\pi \Big( \big(\d\Future(\gamma(0))
       \, \cup\,  \d\Past(\gamma(0))\big) \setminus \cushion\Big)}.
   \end{align}
   This is sketched in Figure~\ref{fig: small dcones steps0}.
   \begin{figure}
     \begin{center}
       \resizebox{9cm}{!}{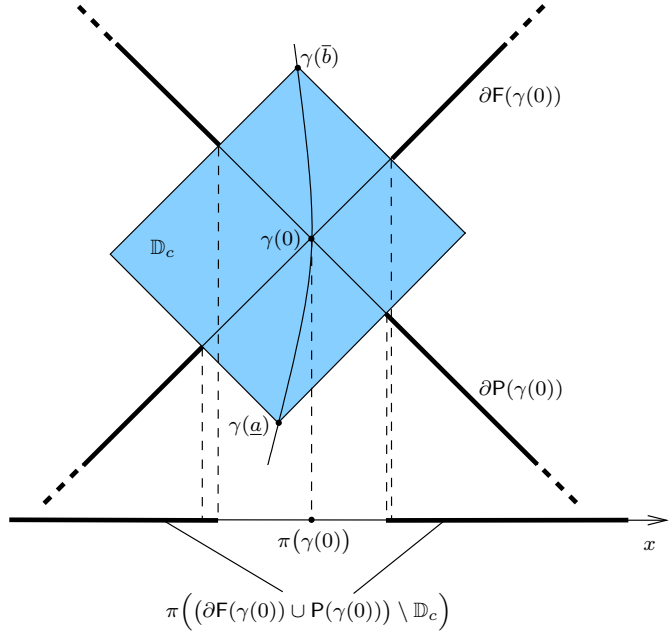}
       \caption{Sketch of the geometry described in
         \eqref{eq:step0_proj}, which leaves wiggle room for
         \eqref{eq: small DC proj condition} and \eqref{eq: small DC
           proj condition-bis} to hold. }
       \label{fig: small dcones steps0}
     \end{center}
     \end{figure}
 Finite speed  of propagation and continuity  imply there exists $0< \eps_1 <\mu$ 
   so that 
   for all $\alpha$, $\beta$ such that $-\eps_1 \leq  \alpha< \beta \leq \eps_1$ one has
    \begin{align}
      \label{eq: small DC proj condition}
      \ovl{\pi\big( \Past(\gamma(\beta)) \setminus
        \big( \Past(\gamma(\alpha))\cup \cushion \big) \big)}
     \, \cap\,  \ovl{\pi \DC(\gamma(\alpha), \gamma(\beta))}
      = \emptyset,
    \end{align}
    and
    \begin{align}
      \label{eq: small DC proj condition-bis}
      \ovl{\pi\big( \Future(\gamma(\alpha)) \setminus
        \big( \Future(\gamma(\beta))\cup \cushion\big) \big)}
     \, \cap\,  \ovl{\pi \DC(\gamma(\alpha), \gamma(\beta))}
     = \emptyset.
    \end{align}
   
Choose $\eps_2\in ]0,\eps_1[$ so that
    $\DC(]\gamma(-\eps_2),\gamma(\eps_2)[) \subset \{| t|<\mu/2\}$.
    This leaves a factor two of wiggle room compared to
    \eqref{eq:NinBand}.  Set $\eps = \min \{ |\ua|, \bb, \eps_2\}$.
    In what follows, suppose
\begin{align*}
  -\eps < a <  b <  \eps.
\end{align*}
This is the constraint on $a,b$.  Define $\DC = \DC(\gamma(a),
\gamma(b))$.

 Figure~\ref{fig: small dcones steps} sketches the regions associated
 with the steps of the proof.
\begin{figure}
  \begin{center}
    \subfigure[\label{fig: small DC-1} Steps 1--3]
    {\resizebox{5cm}{!}{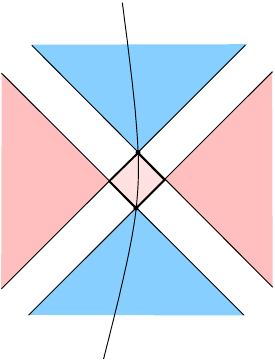}}
    \qquad    \qquad 
    \subfigure[\label{fig: small DC-2} Steps 4 and 5]
     {\resizebox{5cm}{!}{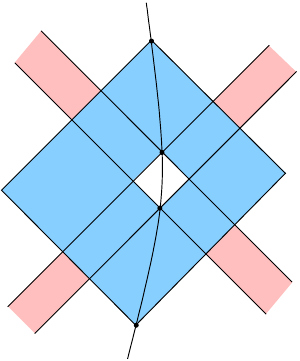}}
    \caption{Regions associated with each step of the
  proof of Theorem~\ref{thm:small_double_cones}.}
  \label{fig: small dcones steps}
  \end{center}
\end{figure}

\medskip
{\bfseries Step~1: $\bld{\cah^1\big(\gamma(]a,b[)\big) = \DC}$ and $\bld{Z_{\gamma(]a,b[)} = Z_\DC}$.} 

Since $\gamma(]a,b[)\subset\DC$, $Z_{\gamma(]a,b[)}\subset Z_\DC$.  To prove that $Z_{\gamma(]a,b[)} = Z_\DC$, it
    suffices to prove that $\DC \subset Z_{\gamma(]a,b[)}$ as it
    implies $Z_\DC \subset Z_{(Z_{\gamma(]a,b[)})}=Z_{\gamma(]a,b[)}$. In the course of this proof, one proves that $\cah^1\big(\gamma(]a,b[)\big) = \DC$.
    Suppose $\O$ is an open subset of $\DC$ with $\gamma(]a,b[)
    \subset \O$, and $z = (t_z,x_z) \in \DC$. We prove that $z \in
    \cah^1 \O$, so Theorem~\ref{theorem:ZO homotopy} implies $z \in
    Z_\O$.  Since $\O$ arbitrary, conclude that $z \in
    Z_{\gamma(]a,b[)}$.

Use the slower metrics $\gd$ with $\delta$ slightly smaller than 1
constrained as follows. The sets $\Future^\delta\big(\gamma(a)\big)$
and $\Past^\delta\big(\gamma(b)\big)$ associated with $\gd$ increase
with $\delta$ and are continuous in $\delta$.  Therefore there is an
$\delta_1<1$ so that the curve $\gamma([a,b])$ is timelike for
$g^{\delta_1}$ and $z \in \DC^{\delta_1} =
\DC^{\delta_1}\big(\gamma(a),\gamma(b)\big)$.
    
For $\eta>0$ define $\ut_z=t_z-\eta$ and $\ot_z=t_z+\eta$.  Since
$z=(t_z,x_z) \in \DC^{\delta_1}$ can choose $\eta>0$ so that
$\uz=(\ut_z,x_z)$ and $\oz=(\ot_z,x_z)$ belong to $\DC^{\delta_1}$.
Because there was a factor of two of wiggle room in the choice of
$\eps_1$, can choose $\delta_2\in ]\delta_1,0[$ so that the futures
    and pasts for $\gd$ for $|t|<\eps_1$ are described by
    Proposition \ref{prop:local form of F} for $\delta_2< \delta<0$.
    That is the constraint on $\delta$.

For $y = (t_y,x) \in \DC^\delta$, define $\varphi_1 (t,y)=
(t,\x{1}(t))$ for $t\in[T_y^-, t_y]$ and $\varphi_2 (t,y)=
(t,\x{2}(t))$ for $t\in [t_y, T_y^+]$, two future oriented
reparameterized null geodesics for the $\gd$ metric satisfying,
\begin{align*}
  \varphi_1 (T_y^-,y) \in \gamma(]a,b[), \quad
    \varphi_1 (t_y,y)=\varphi_2 (t_y,y)=y, \quad
    \varphi_2 (T_y^+,y) \in \gamma(]a,b[)\,.
\end{align*}
Note that given $y \in \DC^\delta$, $\varphi_1$ and $\varphi_2$ are
uniquely defined.

\medskip
Set
\begin{align*}
  \zeta (t)= \begin{cases}
    \varphi_1 (t, \uz)
    & \pour \ T_\uz^- \leq t \leq \ut_z,\\
    \big( \frac{\ot_z-t}{\ot_z-\ut_z} \big) \uz
    + \big( \frac{t - \ut_z}{\ot_z-\ut_z} \big) \oz
    & \pour\ \ut_z \leq t \leq \ot_z.
    \end{cases}
\end{align*}
This point follows $\varphi_1$ from a point on $\gamma(]a,b[)$ to
      $\uz$, and then is vertical between $\uz$ and $\oz$.  For
      $(\sigma,t) \in [T_\uz^-, \ot_z]\times [T_\uz^-, T_\oz^+]$,
      define the homotopy of piecewise  smooth timelike curves for the metric $g$
     by
\begin{align*}
      \XX(\sigma,t) = 
      \begin{cases}
        \zeta (t) & \pour \ T_\uz^- \leq t \leq \sigma,\\
        \varphi_2(t, \zeta (\sigma)) & \pour \ \sigma \leq t \leq T_{\zeta(\sigma)}^+,\\
        \gamma(t) & \pour \ T_{\zeta(\sigma)}^+ \leq t \leq T_\oz^+.
      \end{cases}
\end{align*}
Then $t \mapsto \XX(\sigma,t)$ follows the $\zeta(t)$ curve until it
reaches $\zeta(\sigma)$, then it follows $\varphi_2(.,\zeta (\sigma))$
until it reaches $\gamma(]a,b[)$, and finally it follows $\gamma$
    until reaching the point $\gamma(T_{z(\sigma)}^+)$. The third and
    vertical segment is needed so all the curves in the homotopy are
    defined on the same time interval $[T_\uz^-, T_\oz^+]$. The
    homotopy $\XX$ is sketched in Figure~\ref{fig: homotopy-small-DC}.
      
The time derivatives of $t \mapsto \XX(\sigma,t)$ only fails to be
continuous at three points at most: $t= \ut_z$ if $\sigma > \ut_z$,
$t= \sigma$, and $t = T_{\zeta(\sigma)}^+$ if $\sigma<\ot_z$. The
homotopy $\XX(\sigma, t)$ has initial curve, $\XX_0$, and boundary,
$\Ends{\XX}$, in $\O$ and it reaches $z$ as $\XX(\sigma, t_z)=z$ for $
t_z\leq \sigma \leq \ot_z$. Adapting the proof of
Proposition~\ref{prop: homotopic double-cones}, one smooths $\XX$
yielding a homotopy of smooth timelike curves with intitial curve and
boundary in $\caO$ and passing through $z$.  This gives $z \in \cah^1 \O$
completing the proof that $\DC \subset Z_{\gamma(]a,b[)}$.

\begin{figure}
  \begin{center}
    \resizebox{8cm}{!}{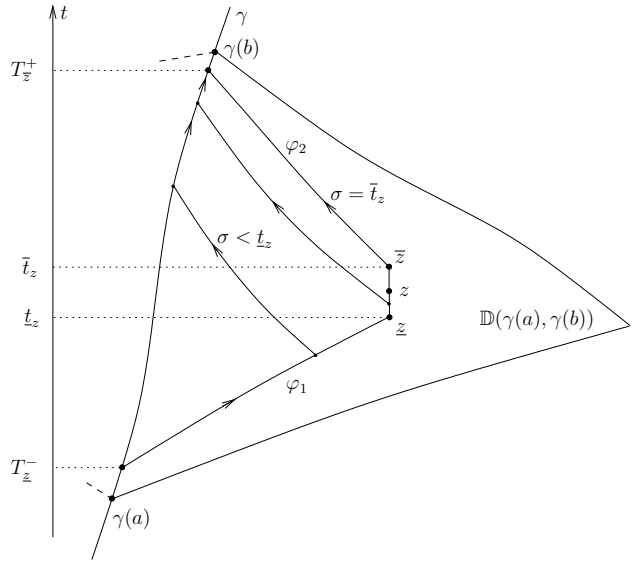}
    \caption{Constuction of a homotopy in a small \dcone.}
  \label{fig: homotopy-small-DC} 
  \end{center}
\end{figure}

\medskip
The remaining steps  prove that $Z_\DC =
\DC$.

\medskip
{\bfseries Step~2: $\bld{Z_{\DC} \cap {\Future(\gamma(a))}^c \cap
    \ovl{\Past(\gamma(b))}^c= \emptyset}$ and $\bld{Z_{\DC} \cap
    \ovl{{\Future(\gamma(a))}}^c \cap \Past(\gamma(b))^c=
    \emptyset}$.}\\ First, suppose $z\in {\ovl{\Future(\gamma(a))}}^c
\cap {\ovl{\Past(\gamma(b))}}^c$. There exists an open \nhd $U$ of $z$
such that $U \subset {\ovl{\Future(\gamma(a))}}^c \cap
{\ovl{\Past(\gamma(b))}}^c$. This gives
\begin{align*}
  y \in U
  \quad \Longrightarrow \quad
  \DC \cap \big( \ovl{\Future(y)} \cup \ovl{\Past(y)}\big)
  = \emptyset.
\end{align*}
Suppose $u^0(x) \in \Cinf(\R^d)$ is supported in the $x$-projection of
$U \cap \{ t= t_z\}$ and $x_z \in \supp(u^0)$.  Then, the solution $u$
of
\begin{align}
  \label{eq:boom}
  P u =0, \qquad u|_{t=t_z}= u^0, \ \ \d_t u|_{t=t_z}=0,
\end{align}
vanishes in $\DC$ and $z\in \supp u$. Thus, $z \notin Z_{\DC}$.

Second, suppose $z \in \d \Future(\gamma(a)) \cap
\ovl{\Past(\gamma(b))}^c$. Then $(t, x_z) \in
    {\ovl{\Future(\gamma(a))}}^c \cap {\ovl{\Past(\gamma(b))}}^c$ for
    $t_z-t>0$ small. This gives $z\notin Z_{\DC}$ as $Z_{\DC}$ is
    open.  One argues similarly to prove that $Z_{\DC} \cap
    \ovl{{\Future(\gamma(a))}}^c \cap \Past(\gamma(b))^c= \emptyset$.
    Points in $\d \Future(\gamma(a)) \cap \d \Past(\gamma(b))$ are not
    covered in Step~2.  They are treated in Step~4.

\medskip
    {\bfseries Step~3: $\bld{Z_{\DC} \cap
    \ovl{\Past(\gamma(a))} = \emptyset}$ and $\bld{Z_{\DC} \cap
    \ovl{\Future(\gamma(b))} = \emptyset}$.} \\
Suppose $y \in \Past(\gamma(a))$. Then $\d\Future(y) \cap
\d\Future\big(\gamma(a)\big) =\emptyset$ by \eqref{eq:empty_intersection_boundary_future}. Consider $z \in \d\Future(y)
\setminus \ovl{\Past(\gamma(b))}$. Then, $z \in
{\ovl{\Future(\gamma(a))}}^c \cap {\ovl{\Past(\gamma(b))}}^c$.

There exists a \bichar starting at $(y,\eta)$ with some $\eta \in
T^*_y\caL$ and reaching $(z, \zeta)$ with $\zeta \in T^*_z\caL$. This
\bichar projects onto a null geodesic connecting $y$ and $z$, thus
lying in $\Lambda^+(y)$ that identifies with $\d\Future(y)$ in the
considered \nhd by Proposition~\ref{prop:local form of F}. In
particular, it remains away from $\DC$.
          
For $u^0$ and $u^1$  supported in ${\ovl{\Future(\gamma(a))}}^c
\cap {\ovl{\Past(\gamma(b))}}^c$, the solution $u$ of 
\begin{align}
  \label{eq:boom_WF}
  P u =0, \qquad u|_{t=t_z}= u^0, \ \ \d_t u|_{t=t_z}=u^1
\end{align}
 vanishes in $\DC$.  Choose $u^0$ and $u^1$ so that $(z,\zeta) \in
 \WF(u)$.  The invarience of $\WF(u)$ under the flow of the hamilton
 field $H_p$ in \eqref{eq:hp} implies that $(y,\eta) \in
 \WF(u)$. Therefore, $y \in \supp u$, so $y \notin Z_{\DC}$. Hence,
 $\Past(\gamma(a)) \cap Z_{\DC} = \emptyset$.  Since $Z_{\DC}$ is
 open, this implies that $\ovl{\Past(\gamma(a))} \cap Z_{\DC} =
 \emptyset$.  Figure~\ref{fig: small dcones steps3} sketches the geometry of the argument of Step~3. The proof that $\ovl{\Future(\gamma(b))} \cap Z_{\DC} =
 \emptyset$ is similar.
\begin{figure}
  \begin{center}
    \resizebox{7cm}{!}{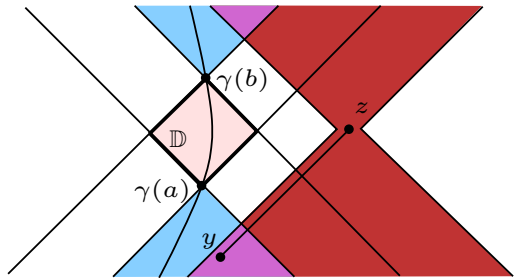}
    \caption{Geometry of the argument of Step~3. Blue and pink region as in Figure~\ref{fig: small DC-1}. In red, the region
      where the solution in \eqref{eq:boom_WF} has potentially some
      support. Purple is the intersection of the blue and red regions.}
  \label{fig: small dcones steps3}
  \end{center}
\end{figure}

\medskip
{\bfseries Step~4: $\bld{( Z_{\DC}\cap \cushion )
    \setminus \DC =\emptyset}$.}\\
First suppose $y \in \cushion \setminus \ovl{\Future(\gamma(a))}$.  The
foliation of Lemma~\ref{lemma: foliation future past}.  implies that
$y\in \d\Future(\gamma(\beta))$ for some $\ua< \beta < a$.  A future
null geodesic achieving the minimal arrival time connects
$\gamma(\beta)$ to $y$.  This null geodesic exits $\Past(\gamma(\ob))$
in finite time. At that time it exists $\cushion$.
    
Thanks to \eqref{eq:NinBand}, it lies in
$\d\Future(\gamma(\beta))\subset \Future(\gamma(\ua))$ before reaching
this exit point, thus remaining away from $\ovl{\Future(\gamma(a))}$
by \eqref{eq:empty_intersection_boundary_future}.  It exits
$\Past(\gamma(b))$ before exiting $\Past(\gamma(\ob))$, and reaches a
point $z \in \ovl{\Future(\gamma(a))}^c \cap
\ovl{\Past(\gamma(b))}^c$. One concludes that $y \notin Z_{\DC}$ as in
Step~3 with a solution of \eqref{eq:boom_WF}.
    
Second, suppose $y =(t_y,x_y) \in \cushion \cap
\d\Future(\gamma(a))$. Then, $(t,x_y) \in \cushion \setminus
\ovl{\Future(\gamma(a))}$ for $t_y-t>0$ small. Thus, $(t,x_y) \notin
Z_{\DC}$, which gives $y \notin Z_{\DC}$ as $Z_{\DC}$ is open. Figure~\ref{fig: small dcones steps4} sketches the geometry of the argument of Step~4.
\begin{figure}
  \begin{center}
    \resizebox{7cm}{!}{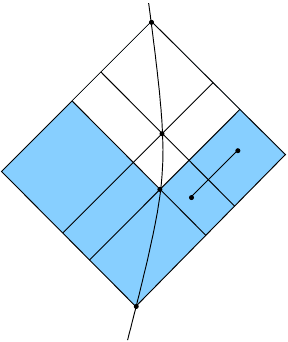}
    \caption{Geometry of the argument of Step~4. The blue region is $\cushion \setminus \ovl{\Future(\gamma(a))}$.}
  \label{fig: small dcones steps4}
  \end{center}
\end{figure}

The case $y \in \cushion \setminus \Past(\gamma(b))$ is similar.

\medskip
{\bfseries Step~5: $\bld{Z_{\DC} \subset \DC}$.}\\ Prove that 
$y=(t_y,x_y) \notin \cushion$ implies  $y \notin Z_{\DC}$. After Steps~2,
3 and 4, only two cases require study: $y \in \Past(\gamma(b))
\setminus (\ovl{\Past(\gamma(a))}\cup \cushion)$ and $y\in
\Future(\gamma(a)) \setminus (\ovl{\Future(\gamma(b))}\cup \cushion)$; see
Figure~\ref{fig: small dcones steps}.

Suppose $y \in \Past(\gamma(b)) \setminus (\ovl{\Past(\gamma(a))}\cup
\cushion)$. Claim: there exists a future null geodesic that starts at $y$,
never meets $\ovl{\DC}$, and reaches a point $z\in
{\ovl{\Future(\gamma(a))}}^c \cap {\ovl{\Past(\gamma(b))}}^c$. One then
constructs a solution as in \eqref{eq:boom_WF} implying that $y \notin
Z_{\DC}$.

To prove the claim choose $a'$ and $b'$ such that $-\eps\leq a' < a$
and $b < b'\leq \eps$.  Set $\DC' = \DC(\gamma(a'),
\gamma(b'))$. Then, \eqref{eq: small DC proj condition}--\eqref{eq:
  small DC proj condition-bis} hold for $\alpha = a'$ and $\beta=
b'$. As $y \in \Past(\gamma(b')) \setminus
(\ovl{\Past(\gamma(a'))}\cup \cushion)$, then
Lemma~\ref{lemma:non_plunging_argument} implies  one of the following two cases:
\begin{enumerate}
\item either $\Lambda^+(y)$ does not plunge into $\Future(\gamma(a'))$;
\item or $\Lambda^+(y)$ plunges into $\Future(\gamma(a'))$ yet not through $\DC'$.
\end{enumerate}
In the first case, there exists a future null geodesic that
starts at $y$ and never enters $\Future(\gamma(a'))$. It reaches a
point $z \in \Past(\gamma(b'))^c$. One has $z\in
{\ovl{\Future(\gamma(a))}}^c \cap {\ovl{\Past(\gamma(b))}}^c$.
In the second case, there exists a future null geodesic
$\rho$ that starts at $y$, never meets $\DC'$, and enters
$\Future(\gamma(a'))$ at $z= \rho(s_z) \in
\d\Future(\gamma(a'))$. One has $z \notin \ovl{\Future(\gamma(a))}$ since $\d\Future(\gamma(a')) \cap \ovl{\Future(\gamma(a))} = \emptyset$. Then, $\rho(s) \in \Future(\gamma(a'))$ for $s
> s_z$. Thus $\rho(s) \notin \Past(\gamma(b'))$ for $s > s_z$, giving
$z \notin \Past(\gamma(b'))$ as $\Past(\gamma(b'))$ is open. Hence, as $\ovl{\Past(\gamma(b))}\subset  \Past(\gamma(b'))$, one finds
$z\in {\ovl{\Future(\gamma(a))}}^c \cap {\ovl{\Past(\gamma(b))}}^c$. 
\begin{figure}
  \begin{center}
    \resizebox{7cm}{!}{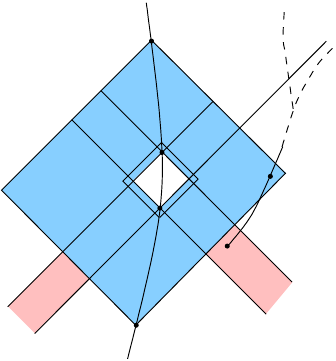}
    \caption{Geometry of the argument of Step~5. Colored regions as in
      Figure~\ref{fig: small DC-2}. Dashed are the possible behaviors of a future null geodesic $\rho$ initiated at $y$. In all cases, $\rho$ does not enter $\DC'$; it may however enter $\cushion$ as represented here.}
  \label{fig: small dcones steps5}
  \end{center}
\end{figure}
    
The claim is proven. Figure~\ref{fig: small dcones steps5} sketches the geometry of the argument of Step~5. The case $y\in \Future(\gamma(a))
 \setminus (\ovl{\Future(\gamma(b))}\cup \cushion)$ is  similar. 
\end{proof}

\subsection{First doldrum example}
\label{sec:First_doldrum_example}
\begin{definition}
\label{def:doldrum2}
Define an operator $P$ on $\R^{1+2}$ with smooth coefficients.  In
polar coordinates it is given by $P = \d_t^2 -c(r)^2 \d_r^2 -
\d_\theta^2$ for $r\ge 1/4$ with $c(r)> 0$, and $c=\eps\ll 1$ in $r\ge
1/2$.
\end{definition}
The associated metric has the  temporal symmetries of Section~\ref{sec:Dcones with special temporal symmetries}.  This section treats timelike curves that
are not vertical. The speeds of propagation   in $r\ge 1/2$ are sketched in Figure~\ref{fig: doldrum example 1}.
\begin{figure}
\begin{center}
  \resizebox{5cm}{!}{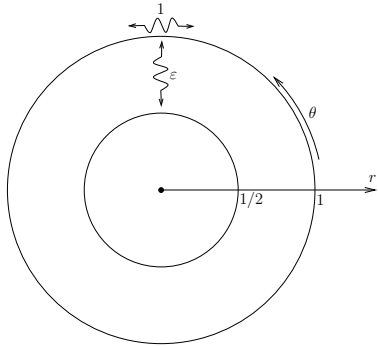}
\caption{First doldrums example.}
\label{fig: doldrum example 1}
\end{center}
\end{figure}
\begin{proposition}
  \label{prop:first_doldrum}
  $\phantom{-}$
  \begin{enumerate}[align=left,labelwidth=0.8cm,labelsep=0cm,itemindent=0.8cm,
  leftmargin=0cm,label=\bf{(\roman*)}]
  \item For this operator there are two timelike arcs with the same
    endpoints that are {\it not} homotopic by timelike arcs with fixed
    endpoints. Furthermore, no point of one arc can be reached by an
    homotopy intiated at the other arc.
  \item Denote by $\DC$ the \dcone generated by the endpoints.  There
    is a solution of $Pu=0$ that vanishes on a \nhd $\O \subset \DC$
    of one of the open timelike arcs and is not identically zero
    inside $\DC$.
  \end{enumerate}
\end{proposition}
\begin{remark}
  \label{rem:first_doldrum}
  The first part gives an example of an arc $\gamma(]0,T[)$ for
    which $\cah^\infty \O$ and thus $\cah^\infty\big(
    \gamma(]0,T[)\big)$ are proper subsets of $\DC$ for $\caO$ a
    small \nhd of $\gamma(]0,T[)$. Compare with Proposition~\ref{prop:
      homotopic double-cones}. Here, the metric has the structure of that
    proposition. However, the arc $\gamma$ is not vertical. The
  second part of Proposition~\ref{prop:first_doldrum} shows that $\DC$
  is not a lower bound for $Z_{\gamma(]0,T[)}$.  Compare
  also with the case of small \Dcones for which $\cah^1\big(
  \gamma(]a,b[\big) = \DC\big(\gamma(a), \gamma(b)\big) =
    Z_{\gamma(]a,b[)}$; see Theorem~\ref{thm:small_double_cones} and
    its proof. The proof of the Proposition shows that the
    construction of the two timelike arcs requires a \suff long time;
    the associated \dcone is then not small.
\end{remark}
\begin{figure}
  \begin{center}
    \subfigure[ \label{fig: double-cone-circle1}]
    {\resizebox{5.4cm}{!}{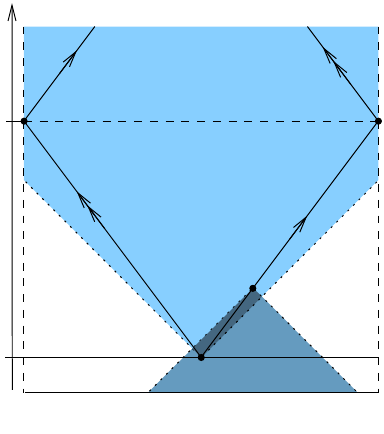}}
    \qquad \quad 
    \subfigure[ \label{fig: double-cone-circle2}]
    {\resizebox{5.4cm}{!}{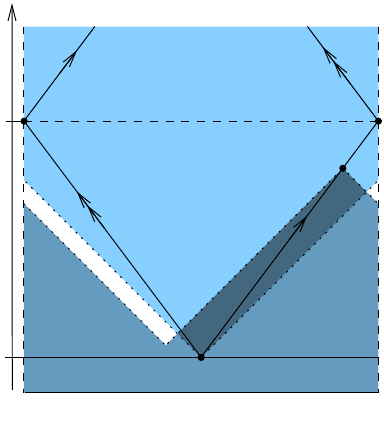}}
    
    \subfigure[$t=T$ \label{fig: double-cone-circle3}]
    {\resizebox{5.5cm}{!}{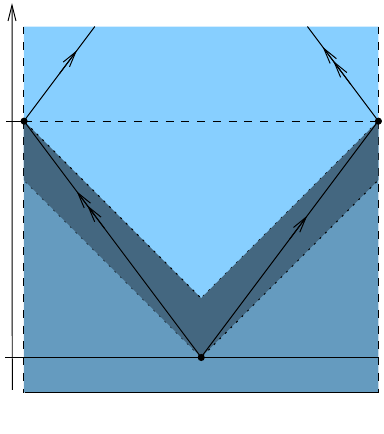}}
    \caption{$\DC(\gamma^+(0),\gamma^+(t))\cap \{ r=1\}$ in dark blue,
      $\Future(\gamma^+(0))$ in ligher blue, and $\Past\big(\gamma^+(t)
      \big)$ in intermediate blue, for various values of $t$.}
  \label{fig: Lack of continuity double cones}
  \end{center}
\end{figure}

\begin{proof}
Define the timelike curves on the cylinder $r=1$ in polar coordinates, 
$\gamma^\pm(t) = (t, 1 , \pm 3 t/4)$.   Then, at $T = 4 \pi
/3$,  $\gamma^+(T) = \gamma^-(T) = (T,1,\pi)$. 
Any homotopy of $\gamma^+([0,T])$ to $\gamma^-([0,T])$ with fixed endpoints must cross
the $t$-axis so pass over $r=0$. A timelike curve that hits the
origin is stuck in $1/2\leq r\leq 3/4$ for at least $C/\eps$ units of
time.  For $\eps$ small, this shows that it cannot get from initial point $(0, 1 , 0)$ 
to final points $(T,1,\pi)$. This proves (i).

Figure~\ref{fig: Lack of continuity double cones} sketches $\gamma^+$,
$\gamma^-$, and the evolution of $\DC\big(\gamma^+(0), \gamma^+(t)\big)$
as $t$ increases.

\begin{figure}
\begin{center}
  \resizebox{6.4cm}{!}{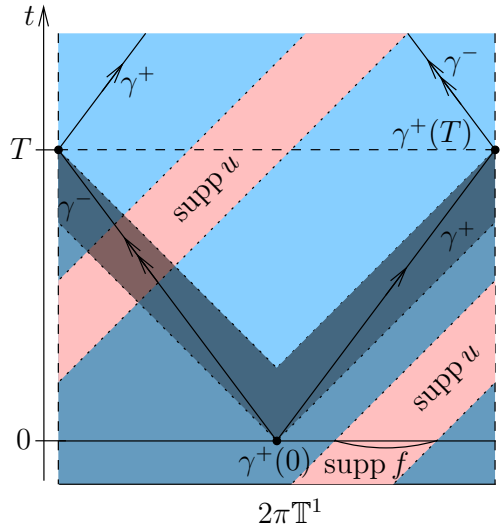}
\caption{A case where a double-cone $\DC(\gamma^+(0),\gamma^+(T))$, in
  dark blue, is not a \dd. The pink band is $\supp u\cap \{r=1\}$. The
  two regions intersect.}
\label{fig: DC not domain of determinacy}
\end{center}
\end{figure}
\medskip
Choose $\zeta \in \Cinf(\R)$ so that $\zeta(r)=0$ for $r<1/8$
and $\zeta(r)=1$ for $r\ge 1/4$, and $f\in \Cinfc(]\pi /4,3\pi /4[)$ with
      $f(\pi/2) =1$.  Define $u_1=\zeta(r) f(\theta-t)$.  Then $u_1$
      is smooth on $\R^{1+2}$, $\supp(Pu_1)\subset \{r\leq 1/4\}$, and
\begin{align}
   \label{eq:repair}
   \supp u_1 \subset \big\{(t,r,\theta):\,
   \pi/4 < \theta-t < 3\pi /4 \big\}.
\end{align}
Define $u_2$ to be the solution of $Pu_2 =-Pu_1$ with
$u_2|_{t=0}=\d_t u_2|_{t=0}=0$.  Define $u=u_1+u_2$ so $Pu=0$.

For $1/2\le r\le 1$ the support of $u_2$ expands radially at speed
$\le \eps$, that is,
\begin{align*}
(\supp  u_2 ) \cap \big\{ r\geq 1/2, t\geq  0 \big\} 
\ \subset\
\big\{ r-1/2 \,\leq \, \eps\,t  \big\}.
\end{align*}
Therefore, for $\eps$ small, $\supp u_2$ does not reach $
\{r=1\} \cap \{0\le t\le T\}$.  Then, $\supp u_2$ does not contribute
to $\supp u$, the pink band in Figure~\ref{fig: DC not domain of
  determinacy}. One checks that $u$ vanishes in a \nhd of
$\gamma^+([0, T])$ and that $\gamma^-([0, T])$ meets $\supp u$. This
completes the proof of (ii).
\end{proof}

\subsection{A doldum example with plunging}
\label{sec:Second_doldums_example}

\begin{proposition}
  \label{prop:second_doldrum}
  Set $\gamma(t)=(t, 0_{\R^d})$.
  There are metrics $g_{(t,x)} = -dt^2 + c(x)^{-2} |dx|^2$ on $\caL = \R^{1+d}$ so that
  $\DC(\gamma(0), \gamma(t))$ is strictly smaller than
  $Z_{\gamma(]0,t[)}$ for $t>0$ \suff large. 
\end{proposition}

When the hypotheses of the next propostion hold,
 $z\notin \DC(\gamma(a), \gamma(b))$ and $z \in Z_{\gamma(]a,b[)}$.
 Proposition \ref{prop:second_doldrum} is proved by constructing 
 examples satisfying the hypotheses.

\begin{proposition}
  \label{prop: double-cone not upper-bound}
  Suppose $\gamma(s)$ is a future timelike curve with
  $ \DC(\gamma(a), \gamma(b))  \subset Z_{\gamma(]a,b[)}$, and
  a point 
  $z \notin
  \DC(\gamma(a), \gamma(b))$.
 
\begin{enumerate}[align=left,labelwidth=0.8cm,labelsep=0cm,itemindent=0.8cm,
  leftmargin=0cm,label=\bf{(\roman*)}]
\item If $\ovl{\Future(z)}$ plunges into $\Future(\gamma(a))$ through $\DC(\gamma(a),
  \gamma(b))$ then $\ovl{\Future(z)} \setminus \Future(\gamma(a))
  \subset  Z_{\gamma(]a,b[)}$. \vskip.1cm
\item 
  If $\ovl{\Past(z)}$ plunges into $\Past(\gamma(b))$ through $\DC(\gamma(a),
\gamma(b))$ then $\ovl{\Past(z)} \setminus \Past(\gamma(b))
\subset Z_{\gamma(]a,b[)}$.
\end{enumerate}
\end{proposition}

\begin{proof}[\bf Proof of Proposition~\ref{prop: double-cone not upper-bound}]
  Treat (i).  The other
  is similar.
  %
\begin{figure}
\begin{center}
  \resizebox{12.4cm}{!}{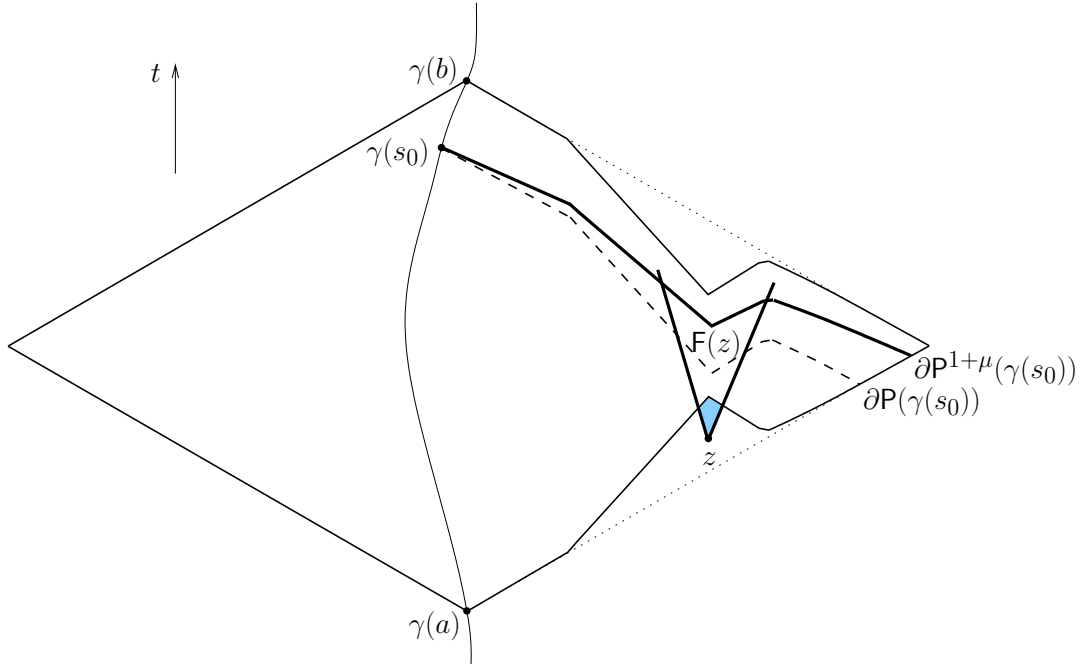}
\caption{Geometry of the proof of Proposition~\ref{prop: double-cone not upper-bound}. The compact $K=\ovl{\Future(z)} \setminus \Future(\gamma(a))$ is the colored region.}
\label{fig:plunging_prop}
\end{center}
\end{figure}

  Define $K = \ovl{\Future(z)} \setminus \Future(\gamma(a))$, compact by
  Lemma~\ref{lemma: plunging1}. 
   Lemma~\ref{lemma: inclusion future sets}  implies 
   that every future causal curve starting
  in $K$ reaches $\d \Past(\gamma(b)) \cap
  \Future(\gamma(a))$. 
   It suffices to show that
   if  $\O$ is an  open \nhd of $\gamma(]a,b[)$
  and $u$ is a solution of  $Pu=0$ with $u=0$ on $\caO$,
         that $u=0$ on $K$.
      
      Lemma~\ref{lemma: plunging2}, the continuity of 
      $(y,x) \mapsto
  \atimeP_{y}(x)$, and the compactness of $K$ imply that 
   there is $a< s_0 < b$ sufficiently close to $b$, so that $K \subset
  \Past(\gamma(s_0))$.
  
  With $\mu>0$ and $\delta=1+\mu$,
      $g^{\delta}=g^{1+\mu}$ 
   is the faster metric
  from  Definition~\ref{def: speedup slowdown metric}. 
  and 
  $\Past^{1+\mu}(\gamma(s_0))$  the past of $\gamma(s_0)$ with
  respect to $g^{1+\mu}$.  
  Lemma~\ref{lemma: continuity arrival time speed
    parameter}  implies that 
    there exists $\mu>0$ such that $K \subset
  \Past^{1+\mu}(\gamma(s_0))$ and $\Past^{1+\mu}(\gamma(s_0)) \cap
  \Future(\gamma(a)) \subset \DC(\gamma(a), \gamma(b))$.
  All future causal curves starting
  at a point of $K$ reach $\d \Past^{1+\mu}(\gamma(s_0))\cap \Future(\gamma(a))$.
  Figure~\ref{fig:plunging_prop} sketches the geometry of the argument.

  By Proposition~\ref{prop: JMR},
  $\Sigma_0 = \d\Past^{1+\mu}(\gamma(s_0))$ is the graph of the
  Lipschitz function
  $x \mapsto \atime^{+}_{g^{1+\mu}, \gamma(s_0)} (x)$. If
  $h(y) = t -\atime^{+}_{g^{1+\mu}, \gamma(s_0)} (x)$ then Lemma
  \ref{lemma: tangent plane future-set} implies that
  $d h(y)\in T^*_y(\R^{1+d})$ is null for almost all $y\in \Sigma_0$
  for $g^{1+\mu}$.  Since the open cone $\Timelike_y^{*,\delta, +}$
  decreases with $\delta$, $d h(y)\in \Timelike_y^{*,1+\mu/2, +}$ for
  almost all $y\in \Sigma_0$. Recall that $\Timelike_y^{*,\delta, +}$
  is the dual cone of $\Timelike_y^{\delta, +}$; see Section~\ref{app:speeding up and slowing down metrics}.
  
    Uniform
  continuity of $y \mapsto \Timelike_y^{*,1+\mu/2, +}$ on a compact sets,
  implies that there is an $\eta>0$ with the following property.
 {\sl If $L$ a compact \nhd of $\ovl{\DC(\gamma(a),\gamma(b))}$,
  $y =(t,x), y'=(t',x')$,  $y, y' \in\Sigma_0 \cap L$, and, 
  $ \Norm{x-x'}{}\leq \eta$,  then,}
  \begin{align}
    \label{eq: unif cont Gamma}
    d h(y')
    \,\in\, \Timelike_y^{*,+} 
    \,=\,
     \Timelike_y^{*,1, +} \,.
   \end{align}
   Choose
  $\chi\in\Cinfc(\R^d)$, with $\int \chi=1$, and 
  for $\eps>0$,
  set $\chi_\eps =
  \eps^{-d} \chi(x/\eps)$.  Smoothing  with convolution in $x$, define
  $\atime_\eps(x) = \chi_\eps *
  \atime^{+}_{g^{1+\mu}, \gamma(s_0)} (x)$ and $h_\eps(y) = t - \atime_\eps(x) = \chi_\eps * h (y)$.
  Then,
  \begin{align*}
    d \atime_\eps(x) \,=\, \chi_\eps *
    d \atime^{+}_{g^{1+\mu},\gamma(s_0)} (x), \qquad d h_\eps(y)\, =\, \chi_\eps * d h (y).
  \end{align*}
  Denote by $\Sigma_\eps$ the hypersurface given by $h_\eps =0$. For
  $\eps$ sufficiently small, $\Sigma_\eps\cap \Future(\gamma(a)) \subset
  \DC(\gamma(a),\gamma(b))$.
  By hypothesis $ \DC(\gamma(a),\gamma(b))\subset Z_{\gamma_{(]a,b[)}}$,
  so $u=0$ on  $\Sigma_\eps\cap \Future(\gamma(a))$.
 
    For $\eps$
  small, \eqref{eq: unif cont Gamma}  shows that
  $dh_\eps(y)$ is a 
     convex linear combination of elements of
  the convex set
  $\Timelike_y^{*,+}$.   Thus,  $d h_\eps(y) \in \Timelike_y^{*,+}$, proving that 
  $\Sigma_\eps$ is spacelike in $\DC(\gamma(a),\gamma(b))$.

   Any causal curve starting at a
  point of $K$ reaches $\Sigma_\eps\cap \Future(\gamma(a))$.  It follows
  from sharp finite speed for $P$
  with vanishing Cauchy data at
  $\Sigma_\eps\cap \Future(\gamma(a))$,
  that $u=0$ on $K$; use \eqref{eq:sharp}. This completes  the proof.
\end{proof}

\begin{proof}[\bf Proof of Proposition~\ref{prop:second_doldrum}]
Suppose $q = (2, 0, \dots, 0)\in \R^d$ and $\eps>0$.  Choose a smooth function
$c(x)$ such that
\begin{itemize}[align=left,labelwidth=*,labelsep=*,itemindent=*,
  leftmargin=0cm]
\item  $\eps \leq c(x) \leq 1$;
\item $c(x) = 1$ for $x$ outside the open ball $\BB$ centered at $q$ with radius $1$; 
\item $c(x) = \eps$ in the ball centered at $q$ with radius $9/10$,
\end{itemize}
yielding the metric $g_{(t,x)} = -dt^2 + c(x)^{-2} |dx|^2$ with the 
symmetries of Section~\ref{sec:Dcones with special temporal
  symmetries}.   Figure~\ref{fig: doldrums1}
sketches $c(x)$.
\begin{figure}
  \begin{center}
    \resizebox{4cm}{!}{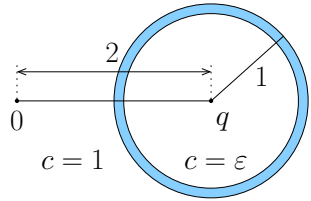}
    \caption{The function $c(x)$ for the second doldrums example.}
  \label{fig: doldrums1} 
  \end{center}
\end{figure}

\begin{figure}
  \begin{center}
    \resizebox{9cm}{!}{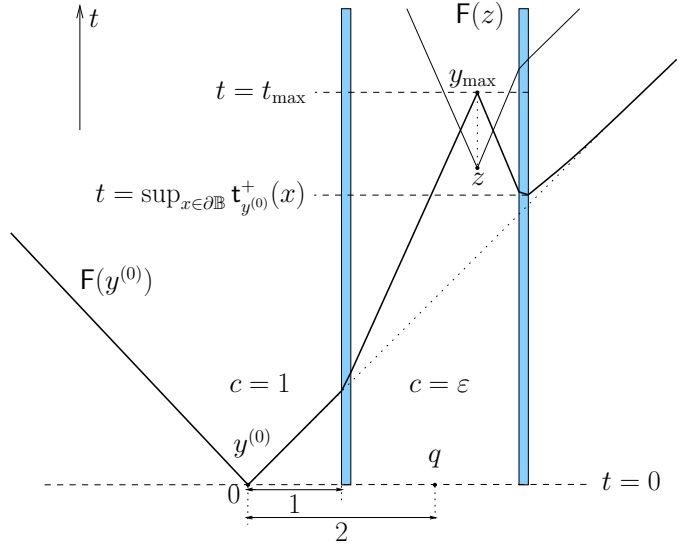}
    \caption{Future set $\Future(\y{0})$ for the doldrums example and a
      point $z$ whose future plunges into $\Future(\y{0})$.}
  \label{fig: doldrums2} 
  \end{center}
\end{figure}

\medskip
If $x \in \d\BB$, the shortest path connecting the origin in $\R^d$ to
$x$ and not entering $\BB$ is of length less than $T = 1 + 2 \pi$, 
a crude estimate, implying $\atimeF_{\y{0}} (x) \leq T$, with
$\y{0}$ the origin in $\R^{1+d}$.  Choosing $\eps>0$ \suff small
yields $\atimeF_{\y{0}} (q) > T$.
Since $x \mapsto \atimeF_{\y{0}} (x)$ is continuous, there exists
$x_{\max}\in \BB$ such that
\begin{align*}
  t_{\max} = \atimeF_{\y{0}} (x_{\max}) =  \sup_{x \in \BB} \atimeF_{\y{0}} (x).
\end{align*}
The point $y_{\max} = (t_{\max}, x_{\max})$ is indicated in
Figure~\ref{fig: doldrums2}.  The set $\d \Future(\y{0})$ is not
differentiable at $y_{\max}$. Otherwise its tangent hyperplane would be
spacelike contradicting Lemma~\ref{lemma: tangent plane future-set}.

\vskip.1cm
Consider $z = (t_z, x_{\max})$ with $\sup_{x \in \d\BB} \atimeF_{\y{0}}
(x) < t_z < t_{\max}$. We claim that
\begin{align}
  \label{eq: claim doldrum2}
  z \notin \ovl{\Future(\y{0})},
  \quad
  {\rm and},
  \quad 
  \ovl{\Future(z)} \ \text{plunges into} \ \Future(\y{0}).
\end{align}
Figure~\ref{fig: doldrums2} is helpful for intuition.
 From the continuity of $x \mapsto \atimeF_{\y{0}} (x)$ and $t_z
 < t_{\max}$ one deduces that $z \notin \ovl{\Future(\y{0})}$.
 From the definition of $x_{\max}$, one has
 \begin{align}
   \label{eq: claim doldrum2-2}
   t > t_{\max}  \ \et \ x \in \ovl{\BB}
   \ \  \Imply \ \
   (t,x) \in \Future(\y{0}).
 \end{align}
 Suppose $y=(t,x) \in \ovl{\Future(z)}$ with $x \notin \ovl{\BB}$.
 Lemma~\ref{lemma: alternative future} implies that 
  there
 exists a future causal curve $\rho(s)$ starting at $z$ and
 ending at $y$.  By continuity, it reaches a
 point $y'= (t', x')$ such that $t_z < t'< t$ and $x' \in \d\BB$.
 As $t_z > \sup_{\d\BB} \atimeF_{\y{0}}$,  $y' \in \Future
 (\y{0})$ implying $y \in \Future(\y{0})$ by Lemma~\ref{lemma: inclusion future sets} since $y \in
 \ovl{\Future(y')}$. Hence,
      \begin{align}
   \label{eq: claim doldrum2-3}
   (t,x) \in \ovl{\Future(z)}
   \ \et \ x \notin \ovl{\BB}
   \ \ \Imply \ \ 
   (t,x) \in \Future(\y{0}).
      \end{align}
      Together, \eqref{eq: claim doldrum2-2} and \eqref{eq: claim
        doldrum2-3} imply that $\ovl{\Future(z)}\setminus \Future(\y{0})$ is compact. 
           Lemma~\ref{lemma: plunging1}
        implies
      the second part of \eqref{eq: claim doldrum2}. The claim is proved.
 
\medskip
Introduce the
curve $\gamma(t) = (t,0)$ along the time axis. One has $\gamma(0) = \y{0}$ and $\ovl{\Future(z)} \setminus \Future(\y{0}) \subset
\Past(\gamma(b))$  for $b
>0$ \suff large, meaning that $\ovl{\Future(z)}$ plunges into $\Future(\gamma(0))$ through
$\DC\big(\gamma(0),\gamma(b)\big)$. Proposition~\ref{prop: Dcones
  DD split metric}
  implies that $\DC\big(\gamma(0),\gamma(t)\big)\subset Z_{\gamma(]0,[t)}$.
  Thus, for $t=b$ \suff large, the assumptions of
      Proposition~\ref{prop: double-cone not upper-bound} are
  satisfied, concluding the construction.
\end{proof}

\begin{figure}
  \begin{center}
    \subfigure[\label{fig: doldrums3a}]
    {\resizebox{5cm}{!}{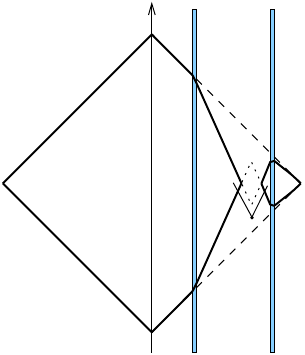}}
    \qquad    \qquad 
     \subfigure[\label{fig: doldrums3b}]
     {\resizebox{6.34cm}{!}{\includegraphics{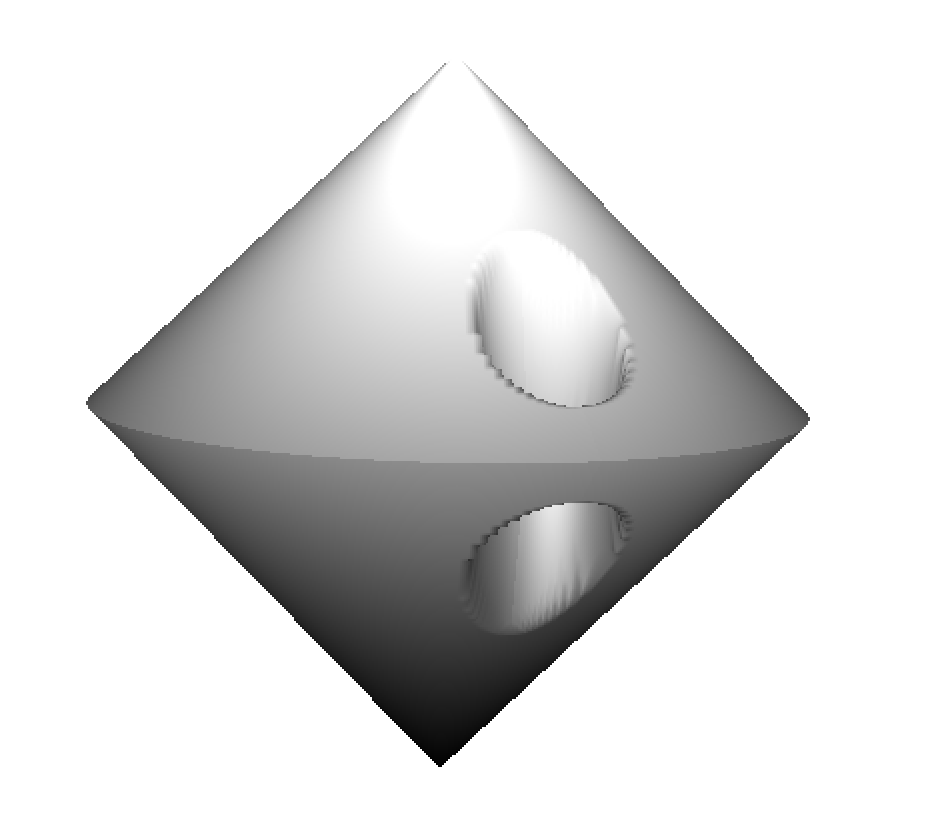}}}
    \caption{Sketches of a \dcone with nonzero genus: (a) cross section,
      (b) a 3d-rendition.}
  \label{fig: doldrums3} 
  \end{center}
\end{figure}

\medskip
This example has additional consequences.
\begin{proposition}
  \label{prop:second_doldrum_DC_with_hole}
  There are metrics $g_{(t,x)} = -dt^2 + c(x)^{-2} |dx|^2$ 
  and vertical timelike $\gamma(t)=(t,0)$ so that 
  the boundary of $\DC(\gamma(-t), \gamma(t))$ has nonzero genus.
\end{proposition}
With the contruction in the proof of Proposition~\ref{prop:second_doldrum}, Figure~\ref{fig: doldrums3a} serves as a proof of
Proposition~\ref{prop:second_doldrum_DC_with_hole}. A \dcone with
genus one is obtained. A 3d rendition is given in Figure~\ref{fig: doldrums3b} to ease intuition. The example can be adapted to achieve any
finite genus for a \dcone.

The point $z$ in Figure~\ref{fig: doldrums3a} gives the following proposition.
\begin{proposition}
  \label{prop: DC with hole no plunging}
  There are metrics $g$, \dcones $\DC = \DC(\y{0}, \y{1})$, and point $z$ such that
  $\Lambda^+(z)$  plunges into $\Future(\y{0})$ through $\DC$
  but $\ovl{\Future(z)}$  does not plunge into $\Future(\y{0})$ through $\DC$.
\end{proposition}
This proves that the properties 1 and 2 in Lemma~\ref{lemma: plunging2} are not equivalent.

\appendix
\section{Aspects of Lorentzian geometry}
\label{app:Lorentzian_geometry}

\subsection{Lorentzian vector spaces}
 \label{sec:Lorentzian vector spaces}
 
\begin{definition}
  \label{def: Lorentzian vector spaces}
  A {\bf Lorentzian vector space} is a pair $(\VV,g)$ where $\VV$ is a
  finite dimensional real vector space with $\dim (\VV) \geq 2$ and
  $g$ is a nondegenerate symmetric bilinear form on $\VV$ with
  signature $-1,1,\dots,1$. One calls $g$ a  {\bf Lorentzian
    metric}.

  Suppose $\bfv \in \VV$. It is called {\bf timelike} if
  $g(\bfv,\bfv)<0$, {\bf null} if $g(\bfv,\bfv)=0$, and
  {\bf spacelike} if $g(\bfv,\bfv)>0$.
Two vectors $\bfv$, $\bfw$ are {\bf orthogonal}, denoted
  $\bfv\perp\bfw$, if $g(\bfv,\bfw)=0$.  
\end{definition}

\begin{remark}
  \label{rem: metric}
  \begin{enumerate}[align=left,labelwidth=0.8cm,labelsep=0cm,itemindent=0.8cm,
      leftmargin=0cm,label=\bf{(\roman*)}]
  \item \label{rem:sig} If $\WW\subset \VV$ is a subspace on which
    $g\leq 0$, then $\dim \WW \le 1$ because of the signature of $g$.
  \item \label{rem:null} If $\bft$ is timelike and $\bfv$ is
    spacelike, they have a null convex linear combination: set
    $\bfw(s) = (1-s) \bft + s\bfv$ and note that $g\big( \bfw(0),
    \bfw(0)\big) <0$ and $g\big( \bfw(1), \bfw(1)\big) >0$.
  \item \label{remark: orthogonal spacelike vector}
    If $g(\bfv, \bfv) \neq 0$ and $\bfu \in \VV$, one has $\bfv
  \perp\bfu + \alpha \bfv$ for some $\alpha\in \R$.
  \end{enumerate}
\end{remark}
\begin{lemma}
  \label{lem:vspace}
  If $\bft$ is timelike and $\bfv \perp \bft$, then $\bfv$ is spacelike.
\end{lemma}
\begin{definition}
  For $d\ge 1$ the {\bf Minkowski space} $\Minkowski^{1+d}$ is 
  $\R^{1+d}_{t,x}$  with the Lorentzian  metric $g_{\Minkowski} \big( (t,x),
  (t,x)\big) = -t^2 +\sum_kx_k^2$.
\end{definition}
\begin{example} The linear transformatons of $\Minkowski^{1+d}$ to
  itself that preserve the metric and map $(1,0,\dots,0)$ to itself
  are the orthogonal transformations of $\RR^d_x$.
\end{example}
\begin{lemma}
  \label{lem:darboux}
  Suppose that $(\VV, g)$ is a Lorentzian vector space of dimension
  $1+d$ and that $\bft\in \VV$ satisfies $g(\bft,\bft)=-1$.  Then there is
  an isomorphism of Lorentzian vector spaces from $\VV$ to $\Minkowski^{1+d}$ that sends
  $\bft$ to $(1,0,\dots,0)$.  The map is unique up to an orthogonal
  transformation of the $x_1, \dots, x_d$ coordinates in
  $\Minkowski^{1+d}$.
\end{lemma}
This result is   useful as  it lets one move constructions from $\VV$ to 
$\Minkowski^{1+d}$.
\begin{lemma}
  The set of timelike vectors $\Cone$ is an open cone made of two
  convex connected components with the map $\bfv\mapsto -\bfv$
  interchanging them. If $\bfv, \bfw \in \Cone$, they lie in the same
  component if and only if $g(\bfv, \bfw )< 0$.
\end{lemma}


\medskip
Subspaces of a Lorentzian vector space fall into
four disjoint 
families.
\begin{enumerate}[align=left,labelwidth=*,labelsep=0cm,itemindent=*,leftmargin=0cm,label=\bf{(\roman*)}]
\item\label{subspace1}Those on which $g$ is strictly positive definite.
\item\label{subspace2}Those on which $g$ is positive semidefinite and not
 strictly positive definite.
\item\label{subspace3}Those on which $g(\bfv,\bfv)$ takes both positive and negative
  values.
\item\label{subspace4}Those on which $g$ is strictly negative definite. 
\end{enumerate}
Examples of two dimensional subspaces of $\Minkowski^{1+2}$ are
illustrated in Figure~\ref{fig: Examples of noncharacteristic
  hyperplanes}.

\subsection{Orthogonals to subspaces}

 \begin{definition}
    \label{def: Lorentzian orthogonal}
 If $\WW$ is a subspace of $\VV$, its {\bf orthogonal} with respect to $g$ is 
\begin{align*}
 \WW^\perp :=
\{ \bfv\in \VV; \ 
 \forall \bfw\in \WW \ \ g(\bfv,\bfw) =0\}.
 \end{align*}
  \end{definition}
  
  \noindent
   Nondegeneracy of $g$   implies  that 
 $\dim \WW+\dim \WW^\perp=\dim \VV$.

 \begin{remark}
   \label{rem:lorentzorthogonal}
   \begin{enumerate}
   \item $\WW\cap \WW^\perp$ is a linear subspace.  If
     $\bfu\in \WW\cap \WW^\perp$ then $g(\bfu,\bfu)=0$.  Remark~\ref{rem: metric}-\ref{rem:sig} implies that $\dim (\WW\cap \WW^\perp) \leq 1$.

   \item In cases~\ref{subspace1} and \ref{subspace4} of the
     classification of subspaces, $\WW\cap \WW^\perp=\{0\}$ and
     $\WW\oplus \WW^\perp=\VV$.  In case~\ref{subspace4}, Lemma
     \ref{lem:vspace} shows that the nonzero elements of $\WW^\perp$
     are spacelike.
 \end{enumerate}
 \end{remark}

\subsection{Hyperplanes}
\begin{example}
  \label{ex: Minkowski2}
  Consider the space  $\Minkowski^{1+d}$.
  A hyperplane $\tau t + \xi_1x_1 + \xi_2 x_2 + \xi_3 x_3=0$
  has equivalent equation $g_{\Minkowski}\big( (-\tau,\xi)\,,\,(t,x)\big) =0$.
     For D'Alembert's operator $\partial_t^2 - \sum_k \partial_k^2$,  this subspace $\HH$ is 
   \begin{equation*}
   \begin{matrix} 
  {\rm characteristic} &\Equiv & 
  \tau^2 \,=\, \sum_k \xi_k^2 
  &\Equiv &g_{\Minkowski}\big( (\tau,\xi), (\tau,\xi)\big)=0  ,
  \cr
   {\rm noncharacteristic}  &\Equiv &
  \tau^2  \,\neq\, \sum_k \xi_k^2 
  &\Equiv &g_{\Minkowski}\big( (\tau,\xi), (\tau,\xi)\big)\neq
  0 ,
  \cr
   {\rm spacelike}  &\Equiv &
  \tau^2 \, >\, \sum_k \xi_k^2
  &\Equiv &g_{\Minkowski}\big( (\tau,\xi), (\tau,\xi)\big)< 0 .
  \end{matrix}
 \end{equation*}
   The remaining option, $\tau^2<\sum_k \xi_k^2 $, that is,
   $g_{\Minkowski}\big( (\tau,\xi), (\tau,\xi)\big)> 0$, 
    yields hyperplanes that are noncharacteristic-nonpacelike.
  \end{example}

  \begin{definition}  
    \label{def: hyperplanes}
    A hyperplane $\HH$ of a Lorentzian vector space has orthogonal
    space $\HH^\perp$ that is one dimensional.  Denote by $\bfv$ a
    basis for $\HH^\perp$.
    $\HH$ is called {\bf characteristic} when
        $g(\bfv,\bfv)=0$. Otherwise $\HH$ is called 
          {\bf noncharacteristic}.     

          A noncharactristic $\HH$ is {\bf spacelike}
          when $g(\bfv,\bfv)<0$, that is, $\bfv$ is timelike.
         The {\bf noncharacteristic-nonspacelike}
            hyperplanes are those with $g(\bfv,\bfv)>0$, that is,
            $\bfv$ is  spacelike.
  \end{definition}
\begin{figure}
  \begin{center}
    \subfigure[ \label{fig: spacelike hyperplane 0}]
    {\resizebox{3.5cm}{!}{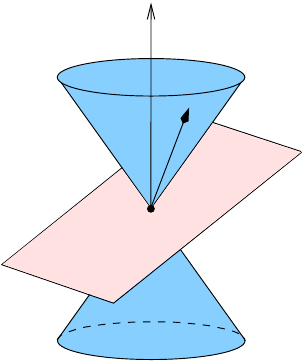}}
    \qquad    \qquad 
     \subfigure[ \label{fig: spacelike hyperplane 1}]
               {\resizebox{2.8cm}{!}{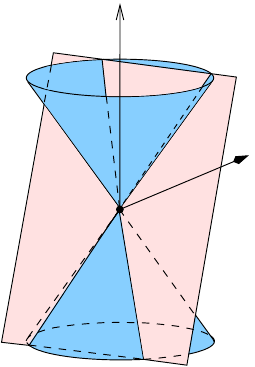}}
 \qquad    \qquad 
     \subfigure[ \label{fig: spacelike hyperplane 2}]
               {\resizebox{3.5cm}{!}{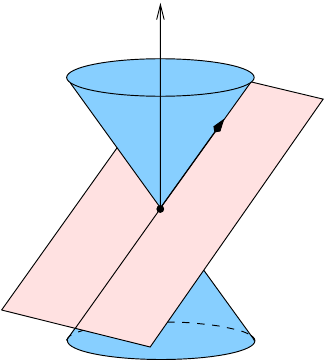}}
               \caption{Hyperplanes in $\Minkowski^{1+2}$: (a)~spacelike hyperplane,
                 (b)~noncharacteristic-nonspacelike hyperplane,
                 (c)~characteristic hyperplane.}
  \label{fig: Examples of noncharacteristic hyperplanes}
  \end{center}
\end{figure}  
  \begin{proposition}
    \label{prop:hyper}
    Suppose that $\HH$ is  a hyperplane of $\VV$.
    \begin{enumerate}[align=left,labelwidth=0.2cm,labelsep=0.2cm,itemindent=0.4cm,
        leftmargin=0cm,label=\bf{(\roman*)}]
    \item The following four statements are equivalent.     
      \begin{enumerate}[align=left,labelwidth=0.2cm,labelsep=0.2cm,itemindent=0.8cm,leftmargin=0cm,
          label=\bf{\alph*.}]

        \item \label{prop:hyperplane1} $\HH$ is charateristic
        \item \label{prop:hyperplane1a}  $\HH \cap \HH^\perp =\Span\{\bfn\}$ with $\bfn$ a
          nonzero null vector.
      \item \label{prop:hyperplane1b}  there is a
        unique null line contained in $\HH$; it is $\HH^\perp$.
      \item \label{prop:hyperplane1c}  $g_{|\HH}$ positive semidefinite and not
        positive definite.\\[-8pt]    
      \end{enumerate}
    \item \label{prop:hyperplane2} $\HH$  is noncharacteristic and spacelike $\Equiv$
      all its nonzero vectors are spacelike, that is, $g_{|\HH}$ is 
      positive definite. \\[-8pt]
    \item The following three statements are equivalent.

      \begin{enumerate}[align=left,labelwidth=0.2cm,labelsep=0.2cm,
          itemindent=0.8cm,leftmargin=0cm,label=\bf{\alph*.}]
      \item  $\HH$ is noncharacteristic-nonspacelike
      
      \item \label{prop:hyperplane3a} $\HH$
        contains a timelike vector.
      \item \label{prop:hyperplane3b}  $g_{|\HH}(\bfv, \bfv)$
        takes both positive and negative values
        if $\dim(\VV) \geq 3$.\\[-8pt]
      \end{enumerate} 
    \item \label{prop:hyperplane4} $\HH$ is characteristic or $\HH$ is
      spacelike $\Equiv$ every timelike vector is transverse to $\HH$. \\[-8pt]
    \item \label{prop:hyperplane5} If $\dim(\VV) \geq 3$, $\HH$ is
      noncharacteristic-nonspacelike or $\HH$ is characteristic
      $\Equiv$ $\HH$ contains a null line.
    \end{enumerate}
  \end{proposition}
The three possible configurations for a hyperplane, that is,
spacelike, characteristic, and nonchararacteristic-nonspacelike, are
illustrated in Figure~\ref{fig: Examples of noncharacteristic
  hyperplanes} for  $\Minkowski^{1+2}$.
\begin{lemma}
  \label{lemma: time-like cone}
  The set of timelike vectors $\Cone$ is an open cone made of two
  convex connected components $\Cone_1$ and $\Cone_2$ with the map
  $\bfv\mapsto -\bfv$ interchanging $\Cone_1$ and $\Cone_2$. If
  $\bfv, \bfw \in \Cone$, they lie in the same component if and only if $g(\bfv, \bfw )< 0$.
\end{lemma}

 \subsection{Lorentzian manifolds and space-time}
 \label{sec:Strict hyperbolicity and Lorentzian manifold}

If $\caL$ is a $(1+d)$-dimensional manifold equipped with a smooth
metric $g_y$ one says that $(\caL,g)$ is Lorentzian if for all $y \in
\caL$, $(T_y\caL, g_y)$ is a Lorentzian vector field.
A vector field $\bfv$ is called
\begin{enumerate}[align=left,labelwidth=*,labelsep=0cm,itemindent=*,leftmargin=0cm,label=\bf{(\roman*)}]
\item timelike if $\bfv_y$ is timelike for all $y\in \caL$. 
\item null if $\bfv_y$ is null for all $y\in \caL$. 
\item spacelike if $\bfv_y$ is spacelike for all $y\in \caL$. 
\end{enumerate}
The Lorentz manifold $(\caL,g)$ is said to be time-oriented if $\caL$
admits a continuous timelike vector field. Denote by $\bft$ such a
vector field. Once oriented a Lorentz manifold is called a
space-time. The setting of Section~\ref{sec:Hyperbolicity and Lorentz
  structure} yields a space-time with $\bft= \omega^\sharp$, with
$\sharp$ the musical isomorphism inverse of $\flat$, that is, in local
coordinates $(\omega^\sharp)^i = g_y^{ij}\, \omega_j$. In
section~\ref{sec:First-arrival_future_past}, $\bft = \d_t$ gives the
time orientation.

$T_y \caL\setminus 0$ is the disjoint union of {\it five connected cones} as
stated in Definition~\ref{def:fivecones}.
\begin{definition}
  \label{def: timelike causal curves}
  Let $I\subset \R$ be an interval and $I \ni s \mapsto \gamma(s)
  \in\caL$ be a Lipschitz curve.
  \begin{enumerate}[align=left,labelwidth=*,labelsep=*,itemindent=*,
  leftmargin=0cm,label=\bf{(\roman*)}]
  \item
    $\gamma$ is  a {\bf timelike curve} if 
    $\gamma '(s)$ is timelike at $\gamma(s)$ \pp $x \in I$. 
  \item
    $\gamma$ is an {\bf causal curve} if  $\gamma '(s)$
    is timelike or null at $\gamma(s)$ \pp $x \in I$. 
  \item
      $\gamma$ is a {\bf null} curve if 
        $\gamma '(s)$ is null at $\gamma(s)$  \pp $x \in I$.  
  \end{enumerate}
  A timelike, causal, or null curve  is said to be future (\resp past)
   if $\gamma '(s)$ is moreover future (\resp past) \pp $x \in I$.  
\end{definition}
A timelike curve is a causal curve, yet causal curves may not
be timelike. In particular, a null curve is a causal curve yet
nowhere timelike.

\bigskip
Consider a hypersurface  $\Sigma$ locally given by $f(y) =0$ with $df \neq 0$. If $y \in \Sigma$ then $T_y \Sigma$ is an
hyperplane of $T_y \caL$.
Following Definition~\ref{def: hyperplanes}
one says that $\Sigma$ is
\begin{enumerate}[align=left,labelwidth=0.8cm,labelsep=0cm,itemindent=0.8cm,
  leftmargin=0cm,label=\bf{(\roman*)}]
\item  characteristic at $y$ if $T_y \Sigma$ is characteristic
  in $(T_y \caL,g_y)$;
\item spacelike at $y$ if $T_y \Sigma$ is spacelike in
  $(T_y \caL,g_y)$;
\item noncharacteristic-nonspacelike at $y$ if $T_y \Sigma$ is
  noncharacteristic-nonspacelike in $(T_y \caL,g_y)$.
\end{enumerate}
Set $\nablaL f (y) = \big(d
f(y)\big)^\sharp$. One calls $\nablaL f$ the Lorentzian gradient  vector
field
of $f$. In local coordinates 
\begin{align*}
  \nablaL f (y) = \sum_{0\leq i \leq d}  (\nablaL f)^i(y) \, \d_{y_i}
  \ \ \avec \ \ 
  (\nablaL f)^i(y)= \sum_{0\leq j \leq d} g^{i j}(y)\d_{y_j} f(y).
\end{align*}
One has $\nablaL f (y)\neq 0$ and 
$g_y ( \nablaL f (y) , v) = \dup{df(y)}{v} =0$  if $v \in T_y
\Sigma$. One deduces the following result. 
\begin{lemma}
  \label{lemma: nabla L f}
  One has $\big(T_y \Sigma)^\perp  = \Span\big\{
\nablaL f (y)\big\}$
 and 
\begin{enumerate}[align=left,labelwidth=0.8cm,labelsep=0cm,itemindent=0.8cm,
  leftmargin=0cm,label=\bf{(\roman*)}]
\item $\Sigma$ characteristic at $y$  $\Equiv$  $\nablaL f (y)$ null 
  $\Equiv$  $g_y^* \big( d f
  (y) , d f (y) \big) = 0$; $d f (y)$
  characteristic.
\item $\Sigma$  spacelike at $y$ $\Equiv$  $\nablaL f (y)$ timelike.
\item $\Sigma$ noncharacteristic-nonspacelike at $y$ $\Equiv$ $\nablaL f (y) $ spacelike.
\end{enumerate}
\end{lemma}

\subsection{Case $\bld{\caL = \R^{1+d}}$}
\label{sec:case L=R1+d}
Here, we give useful properties of the space-time $(\caL,g)$ for $\caL=
\R^{1+d}$ under Hypothesis~\ref{hyp: metric}.
Points in $\caL$ are denoted $y = (t, x)$ with $t \in \R$ and $x \in
\R^d$ and $\bft= \d_t$ provides the time orientation. One also writes $y = (y_0, y_1, \dots, y_d)$ with $y_0
= t$ and $y_i = x_i$, $i=1, \dots, d$. 
\begin{lemma}
  \label{lemma: property g R1+d}
  Assume Hypothesis~\ref{hyp: metric} holds. Suppose $y \in \caL=
  \R^{1+d}$ and $\bfv \in T_y \caL$. Write $\bfv=(v^0, \bfv_x) \in
  \R\times \R^d$, with $\bfv_x=(v^1, \dots, v^d)$.  The following
  properties hold.
  \begin{enumerate}[align=left,labelwidth=0.8cm,labelsep=0cm,itemindent=0.8cm,
      leftmargin=0cm,label=\bf{(\roman*)}]
  \item If $v^0=0$, then $g_y(\bfv, \bfv) \geq C \Norm{\bfv_x}{}^2$, for $C$ uniform in $y$. 
  \item Suppose $\bfv$ is timelike or null. Then, $g_y(\bfv, \d_t)$
    has the opposite sign of $v^0$. In other words $\bfv$ is future if
    $v^0>0$ and $\bfv$ is past if $v^0<0$.
  \end{enumerate}
\end{lemma}
By Lemma~\ref{lemma: property g R1+d}, if $\gamma(s)$ is a timelike
(\resp causal, \resp null) curve with $\gamma(s) = (t(s), x(s))$,
it is future if and only if $t'(s)>0$ \pp $x \in I$, and past if and
only if $t'(s)<0$ \pp $x \in I$.

Similarly to the second part of Lemma~\ref{lemma: property g R1+d} one has the following property.
\begin{lemma}
  \label{lemma: property g R1+d bis}
  Assume Hypothesis~\ref{hyp: metric} holds. Suppose $y \in \caL=
  \R^{1+d}$ and $\omega \in T_y^* \caL$. Write $\omega=(\omega^0,
  \omega_x) \in \R\times \R^d$.  Suppose $g^*_y (\omega, \omega) \leq
  0$. Then, $g^*_y(\omega, dt)$ has the opposite sign of 
  $\omega^0>0$.
\end{lemma}

\section{Hypersurface regularity and unique continuation}
\label{sec:hypersurface Robbiano-Tataru}

In the references
\cite{Robbiano:91,Tataru:95,Hoermander:97,RZ:98,Tataru:99},
noncharacteristic-nonspacelike hypersurfaces across which unique
continuation is proven are $\Con^2$ or $\Cinf$.  The next proposition
states that a $\Con^1$-hypersurface suffices for unique continuation to hold. 
\begin{proposition}
  \label{prop: UCP from Cinf to C1}
  Suppose that the unique continuation property of Definition~\ref{def:unique_continuation} holds for 
  $\Cinf$ noncharacteristic-nonspacelike hypersurfaces. Then, it holds for
  such a $\Con^1$  hypersurface.
\end{proposition}
\begin{proof}
Consider an embedded noncharacteristic-nonspacelike $\Con^1$
hypersurface $\Sigma$, $m \in \Sigma$, a bounded open \nhd $V$ of $y$,
and $u \in H^2(V)$ such that $P u=0$ and $u =0$ on one side of
$\Sigma$. Upon reducing $V$, suppose there are smooth coordinates
$(s,z) \in \R\times \R^d$, such that $m=(0,0)$ and $T_m \Sigma$
identifies with $\{s=0\}$. Locally, $\Sigma$ is given by $s = f(z)$
with $f$ a $\Con^1$ function. Choose $f$ such that $u$ vanishes in
$\{s> f(z)\}$. At $y=(f(z),z)$, a nonzero conormal vector is $N_y =d f
(z) -d s$.  One has $g^*_{y} (N_{y}, N_{y})\geq C_0>0$ if $y \in
\Sigma\cap V$.

Upon reducing $V$ and taking $V=]-S,S[\times B_{\R^d}(0,R_0)$ assume
    that $|f| < S$ on $B_{\R^d}(0,R_0)$.  Consider $\chi \in
    \Cinfc(\R^d)$, with $\int \chi =1$, and set $\chi_\delta(z)=
    \delta^{-d} \chi(z/\delta)$ for $\delta>0$. For $b>0$ and $R_1=
    R_0/2$, set
\begin{align*}
  h(a, z) = \chi_\delta * f(z)+ a+ b \Norm{z}{}^2
  \ \ \pour \ \
  a \in \R \ \et \ z \in  \ovl{B_{\R^d}(0,R_1)}, 
\end{align*}
well defined for $\delta$ small. Set 
$\caH_a = \{ (h(a,z),z);\ z\in  \ovl{B_{\R^d}(0,R_1)}\}$.
The hypersurface $\caH_a$ is smooth.
For $y \in \caH_a$, set $N^a_y = d_z h(a,z) - d s$ conormal to $\caH_a$ at $y$.
For $b$ and  $|a|$ small, say $0< b \leq b_0$,
$|a|\leq a_0$,  and $\delta$ small  then $\caH_a \subset V$ and 
\begin{align*}
  g^*_{y} (N^{a}_y, N^{a}_y) \geq C_0/2, \qquad  \pour \ y \in
  \caH_a, 
\end{align*}
since $\chi_\delta * f \to f$ in the $\Con^1$ topology.
Set $b=b_0$, $a_1= \min (a_0, b_0 R_1^2/2)$,
$0< \eps < a_1$, and $\delta$ \suff small such that 
$\sup_{B_{\R^d}(0,R_1)} |\chi_\delta * f - f| \leq \eps$.
Then, the three following properties hold:
\begin{enumerate}[align=left,labelwidth=0.8cm,labelsep=0cm,itemindent=0.8cm,
  leftmargin=0cm,label=\bf{(\roman*)}]
  \item If $\Norm{z}{}\leq R_1$ then $h(a_1,z) > f(z)$.
  \item If $\Norm{z}{} = R_1$ and $|a|\leq a_1$ then
  $h (a,z) > f(z)$.
  \item $h(-a_1,0) <  f(0)$.
\end{enumerate}
 Set $\M =
\ovl{B_{\R^d}(0,R_1)}$ and $F:[-a_1, a_1]\times \M \to V$ given by
$F(a,z) = (h(a,z),z) \in \caH_a$, yielding a continuous
noncharacteristic-nonspacelike {\it smooth} hypersurface deformation
that sweeps a \nhd of $m$, with
\begin{align*}
  F([-a_1, a_1] \times \d\M) \cup F(\{a_1\} \times \M) \subset  \{ s> f(z)\} \ \ \where \ u \ \text{vanishes}.
  \end{align*}
Then, the assumed unique continuation property in the smooth
case implies $u=0$ in a \nhd of $m=(0,0)$.
\end{proof}

\section{Faster and slower metrics}
\label{app:speeding up and slowing down metrics}

Use the notation of Section~\ref{sec:elements_pseudo-Riemannian} and Appendix~\ref{app:Lorentzian_geometry}.

Consider $\delta>0$.  For a tangent vector $\bfu= (u_t, u_x)$ set
$S_\delta \bfu = ( u_t/\delta, u_x)$. For $\delta>1$, $S_\delta \bfu$
is faster than $\bfu$.  For  $\delta <1$, $S_\delta \bfu$
is slower than $\bfu$.  Set
\begin{align*}
  \gd_y \big( \bfu, \bfv\big) = g_y \big( S_{1/\delta}  \bfu ,
  S_{1/\delta}  \bfv\big), \qquad \bfu, \bfv \in T_y \caL.
\end{align*}
It is a Lorentzian metric for which $\d_t$ is a future timelike vector.
If $\delta>1$ and $\bfu$ is null for $g_y$, then $S_{1/\delta}  \bfu$ is
timelike for $g_y$. Thus $\bfu$ is timelike for $\gd_y$. Conversely,
if $\bfu$ is null for $\gd_y$, it is spacelike for $g_y$. Hence, a
null curve for $\gd$ is a spacelike curve for $g$. The opposite for $0 <
\delta < 1$. 
\begin{definition}
  \label{def: speedup slowdown metric}
  One says that the metric $g$ is made {\bf faster} if changed into $\gd$ for
  $\delta >1$. One says that $g$ is made {\bf slower} if changed into $\gd$ for
  $0< \delta < 1$. 
\end{definition}
For the metric $\gd$, denote $\Null_y^{\delta,\pm}$ the cones of
forward and backward null vectors at $y$.  Denote by
$\Timelike_y^{\delta,\pm}$ the open convex cone of forward
(resp. backward) timelike vectors at $y$ for the metric $\gd$. Observe
that $\Timelike_y^{\delta,\pm}$ increases as $\delta$ increases.
\begin{remark}
  \label{remark: speedup slowdown metric}
  In the particular case where $g = -dt^2 + \gR$ with $\gR$ a
  Riemannian metric on $\R^d$, then $\gd = -\delta^2 dt^2 + \gR$. Note
  that $(\gd)^* = \delta^{-2} ( - d\tau^2 + \delta^2
  \gR^*)$. Figure~\ref{fig: speed-up slow-down} shows the effect of
  making a metric faster in the Minkowski case.
\end{remark}
Denote by $\Timelike_y^{*,\delta,\pm}$ the open convex cone of forward
(resp. backward) timelike covectors at $y$ for the dual metric
$(\gd)^*$ on $T^* \caL$. Observe that $\Timelike_y^{*,\delta,\pm}$
decreases as $\delta$ increases.

\begin{figure}
  \begin{center}
  \resizebox{2.2cm}{!}{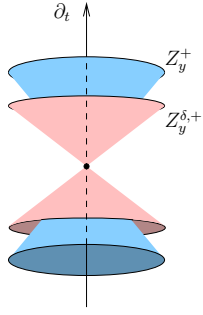}
  \caption{Faster metric $\gd$, $\delta >1$, for $g = - dt^2 +
    |dx|^2$.  $\Null_y^{+}$ denotes the cone of forward null vectors of $g$ and
    $\Null_y^{\delta,+}$ that of $\gd$.}
  \label{fig: speed-up slow-down}
\end{center}
\end{figure}

For $y \in \caL$ and $x \in \R^d$, denote by $\atimeFd{\delta}_y(x)$
the first-arrival time function for the metric $\gd$.
\begin{lemma}
  \label{lemma: continuity arrival time speed parameter}  
  The function $\delta
  \mapsto \atimeFd{\delta}_y(x)$ is 
  continuous and is  decreasing.
\end{lemma}
\begin{proof}
  By \eqref{eq: first arrival time},  $\delta \mapsto
  \atimeFd{\delta}_y(x)$ is nonincreasing. Having
  $\atimeFd{\delta}_y(x) = \atimeFd{\delta'}_y(x)$ for $\delta\neq \delta'$ yields a
  contradiction by 
  Proposition~\ref{prop: minimizing curves}: a minimizing curve cannot be
  null for both metrics.

  Suppose $\delta_n \searrow \delta>0$. Set $t_n =
  \atimeFd{\delta_n}_y(x)$ and $y_n = (t_n, x)$. By
  Proposition~\ref{prop: minimizing curves}
  there is a future null geodesic connecting $y$ and $y_n$ for the
  metric $g^{\delta_n}$, with initial tangent vector $\bfv_n \in
  \Null_y^{\delta_n,+}$. Then $\bfu_n = S_{\delta/\delta_n} \bfv_n
  \in \Null_y^{\delta,+}$.  Consider the future null geodesic for
  the metric $g^{\delta}$ initiated at $y$ with initial
  trangent vector $\bfu_n$. It reaches a point $z_n$ such that $z_n
  =y_n+ o(1)$.  Then, there exists a future null geodesic for
  $g^{\delta}$ connecting $z_n$ and
  $z_n'=(t_n',x)$ with $t_n'=\atimeFd{\delta}_{z_n}(x)$. Hence, there is a future
  Lipshitz null curve connecting $y$ and $z_n'$ for the metric
  $g^\delta$ and $z_n'=y_n+ o(1)$. Thus, $t_n' = t_n + o(1)$. One concludes that
  \begin{align*}
   \atimeFd{\delta_n}_y(x) \leq \atimeFd{\delta}_y(x) \leq
   \atimeFd{\delta_n}_y(x) + o(1),
  \end{align*}
  where the first inequality is the monotonicity proven
 above. This gives continuity  at $\delta^+$. 

  Suppose $\delta_n \nearrow \delta$. Interchanging the roles played
  by the two metrics one finds a Lipschitz null curve for the metric $g^{\delta_n}$ that
  starts at $y$ and ends at $(t_n,x)$ with $t_n \leq
  \atimeFd{\delta}_y(x) + o(1)$. One obtains 
$\atimeFd{\delta}_y(x) \leq \atimeFd{\delta_n}_y(x)
   \leq \atimeFd{\delta}_y(x) + o(1)$,
hence continuity at $\delta^-$. 
\end{proof}
\begin{corollary}
  \label{cor: adjusting speed}
  Suppose $y = (\ut,\ux) \in \caL$ and  $\R^d \ni x \neq \ux$.
  If  $t > \ut$, then there exists $\delta>0$ such
  that $t = \atimeFd{\delta}_y(x)$.
\end{corollary}
\begin{proof}
  If $\delta \to 0_+$,  then $\atimeFd{\delta}_y(x) \to
  +\infty$. If $\delta \to +\infty$, then  $\atimeFd{\delta}_y(x) \to
  \ut$. Conclude with Lemma~\ref{lemma: continuity arrival time speed
    parameter} and the intermediate value theorem.
\end{proof}

\section*{Use of Generative-AI tools declaration}
The authors declare that they have not used any Artificial Intelligence (AI) tools in the creation of this article.

\bibliographystyle{acm}
\bibliography{./references}
\end{document}